\documentclass{article}
\usepackage{amsfonts,amssymb,amsmath}
\usepackage{amsthm}
\usepackage{authblk}
\usepackage{graphicx} 
\usepackage{xcolor}
\usepackage[margin=2.6cm]{geometry}
\usepackage{ulem}
\newtheorem{theorem}{Theorem}[section]
\newtheorem{definition}[theorem]{Definition}
\newtheorem{proposition}[theorem]{Proposition}
\newtheorem{remark}[theorem]{Remark}

\numberwithin{equation}{section}

\renewcommand{\arraystretch}{1}
\title{Structure-preserving preconditioning of discrete space-fractional diffusion equations with variable coefficient and $\theta$-Method }
\author[1]{Muhammad Faisal Khan}
\author[1]{Asim Ilyas}
\author[2]{Rolf Krause}
\author[1,3]{Stefano Serra-Capizzano}
\author[4]{Cristina Tablino-Possio}
\affil[1]{Department of Science and High Technology, University of Insubria, Italy; email: ailyas1@uninsubria.it; mfkhan@uninsubria.it; s.serracapizzano@uninsubria.it}
\affil[2]{Division Computer, Electrical and Mathematical Sciences and Engineering (CEMSE), King Abdullah University of Science and Technology (KAUST), Saudi Arabia; email: rolf.krause@kaust.edu.sa}
\affil[3]{Department of Information Technology, University of Uppsala, Sweden; email: stefano.serra@it.uu.se}
\affil[4]{Department of Mathematics and Applications, University of Milano - Bicocca, Italy; email: cristina.tablinopossio@unimib.it}
\date{} 
\begin{document}
\maketitle
\begin{abstract}
This paper studies the spectral properties of large matrices and preconditioning of linear systems, arising from the finite difference discretization of a time-dependent space-fractional diffusion equation with a variable coefficient $a(x)$ defined on $\Omega \subset \mathbb{R}^{d}$, $d=1,2$. The model involves a one-sided Riemann--Liouville fractional derivative multiplied by the function $a$, discretized using the shifted Grünwald formula in space and the $\theta$-method in time. The resulting all-at-once systems exhibit a $(d+1)$-level Toeplitz-like matrix structure, with $d=1,2$ denoting the space dimension, while the additional level is due to the time variable.
\par
A preconditioning strategy is developed based on the structural properties of the discretized operator. Using the generalized locally Toeplitz (GLT) theory, we analyze the spectral distribution of the unpreconditioned and preconditioned matrix sequences. The main novelty is that the analysis fully covers the case where the variable coefficient $a$ is nonconstant. Numerical results are provided to support the GLT based theoretical findings, and some possible extensions are briefly discussed.
\end{abstract}
\section{Introduction}
In this work, we focus on solving the following class of evolutionary partial differential equations (PDEs)
\begin{equation}\label{eq1}
\begin{cases}
\displaystyle \frac{\partial}{\partial t}u(x_{1},x_{2},t) = \mathcal{L} u(x_{1},x_{2},t) + f(x_{1},x_{2},t), & (x_{1},x_{2},t) \in \Omega \times (0,T],  \\
u(x_{1},x_{2},t) = 0, & (x_{1},x_{2}) \in \partial \Omega,\; t \in [0,T], \\
u(x_{1},x_{2},0) = \phi(x_{1},x_{2}), & (x_{1},x_{2}) \in \Omega,
\end{cases}
\end{equation}
where \( \Omega =   {(a_1,b_1) \times (a_2,b_2)} \subset \mathbb{R}^{2} \), \( f(x_{1},x_{2},t) \) is a given source term,  \( u(x_{1},x_{2},t) \) denotes the unknown solution and \( \phi(x_{1},x_{2}) \) is the prescribed initial condition. The form of equation \eqref{eq1} encompasses various PDE models depending on the choice of the spatial operator $\mathcal{L}$, including classical heat equations \cite{hon2024, lin2021, mcdonald2018}, space-fractional diffusion models \cite{donatelli2016, she2023, LinNgSun2018}, and other nonlocal dynamics. \\
In the present paper, we consider the following specific form of the operator $\mathcal{L} = a(x_{1},x_{2}) \,   {\left({}_{a_1} D_{x_{1}}^{\alpha_1}
+ {}_{a_2} D_{x_{2}}^{\alpha_2} \right)}$,
where the function \( a(x_{1},x_{2})  \) is a smooth variable diffusion coefficients,   { \( {}_{a_1} D_{x_{1}}^{\alpha_1} \) } denotes the Riemann-Liouville fractional derivative of order   {$\alpha_1 \in (1,2)$ with respect to $x_1$} and   {\( {}_{a_2} D_{x_{2}}^{\alpha_2} \)} denotes the Riemann-Liouville fractional derivative of order   {$\alpha_2\in (1,2)$ with respect to $x_2$}.
The Riemann--Liouville fractional derivative   {in the one-dimensional setting} is defined as follows
\begin{equation}\label{eq3}
{}_a D_x^\alpha U(x) = \frac{1}{\Gamma(2 - \alpha)} \frac{d^2}{dx^2} \int_a^x \frac{U(\xi)}{(x - \xi)^{\alpha - 1}} \, d\xi,
\end{equation}
where \( \Gamma(\cdot) \) denotes the Gamma function.
\par
  {Thus, substituting} this choice of $\mathcal{L}$ into equation \eqref{eq1}, we obtain the following space-fractional diffusion equation (SFDE)
\begin{equation}\label{eq2}
\begin{cases}
\displaystyle \frac{\partial}{\partial t} u(x_{1},x_{2},t) = a(x_{1},x_{2})   {\left({}_{a_1} D_{x_{1}}^{\alpha_1}
+ {}_{a_2} D_{x_{2}}^{\alpha_2} \right)} u(x_{1},x_{2},t) + f(x_{1},x_{2},t), & (x_{1},x_{2},t) \in \Omega \times (0,T],  \\
u(x_{1},x_{2},t) = 0, & (x_{1},x_{2}) \in \partial \Omega,\; t \in [0,T], \\
u(x_{1},x_{2},0) = \phi(x_{1},x_{2}), & (x_{1},x_{2}) \in \Omega,
\end{cases}
\end{equation}
with the one-dimensional space version given by
\begin{equation}\label{eq2-1d}
\begin{cases}
\displaystyle \frac{\partial}{\partial t} u(x,t) = a(x) {}_a D_{x}^\alpha  u(x,t) + f(x,t), & (x,t) \in (a_1,b_1) \times (0,T],  \\
u(x,t) = 0, & x \in \partial (a_1,b_1),\; t \in [0,T], \\
u(x,0) = \phi(x), & x \in (a_1,b_1).
\end{cases}
\end{equation}
\\
The SFDE has become an important model for describing anomalous diffusion phenomena, particularly superdiffusive behavior, where a cloud of particles disperses more rapidly than predicted by classical diffusion models \cite{Meerschaert2002, Tadjeran2006}. In general, obtaining exact analytical solutions for such equations is quite difficult, especially when the diffusion coefficient varies spatially. To address the latter issue, a variety of numerical methods have been developed to approximate the solution of SFDEs such as \eqref{eq2}; see \cite{Tadjeran2006,Brociek2016,ChenDeng2014,MeerschaertTadjeran2006,Tadjeran2007,Tian2015,Wang2010}. However, due to the inherent nonlocality of the Riemann–Liouville (RL) fractional derivative, these discretization schemes typically lead to the formation of dense linear systems. Solving these dense systems with direct solvers becomes computationally prohibitive as the grid resolution increases. The observed challenge highlights the need of designing fast and efficient solvers for the resulting linear systems of the SFDE.\\
In the current work, we study SFDEs of the form \eqref{eq2}-\eqref{eq2-1d},  where the left-hand side features the time derivative, and the right-hand side involves a spatially varying coefficient $a$ multiplied by the left-sided Riemann–Liouville (RL) fractional derivative. To discretize the temporal component, we apply the $\theta$-scheme, which yields a flexible and widely used time integration framework. The resulting semi-discrete formulation leads to a linear system at each time step, where the structure of the coefficient matrix reflects the nature of the underlying operators: the variable coefficient $a$ gives rise to a diagonal matrix, while the discretization of the RL derivative contributes with a $d$-level nonsymmetric Toeplitz matrix, with $d=1$ for \eqref{eq2-1d} and $d=2$ for \eqref{eq2}.
Indeed we proceed by considering
\begin{itemize}
\item first the one-dimensional case of \eqref{eq2-1d},
\item second  the two-dimensional case of \eqref{eq2}.
\end{itemize}
By assembling the system across all time levels, we arrive at an all-at-once formulation, where the global linear system is expressed using Kronecker products. Specifically, for $d=1$, setting $N$ the number of uniformly spaced points and $M$ the number of uniformly spaced points, the matrix structure takes the form {$ G_{\alpha,N} \otimes I_{M} + I_{N} \otimes Q_{\theta,M},$
where \( Q_{\theta, M} \in \mathbb{R}^{M \times M} \) } is the time-stepping matrix derived from the \(\theta\)-scheme, and \( G_{\alpha,N} \) denotes the spatial operator, which takes the form of a diagonal-times–nonsymmetric Toeplitz matrix.
 The conditioning of the resulting matrix is strongly influenced by the discretization parameters, particularly by the ratio \(  \gamma=\Delta t / (\Delta x)^{\alpha} \). As this ratio increases, the system becomes increasingly ill-conditioned, posing significant challenges for iterative solvers.
\\
For $d=2$, setting   {$\textbf{n} = (N_1,N_2)$} the number of uniformly spaced points in the variable $x_j, j=1,2$, the structure is similar i.e.,
  {$G_{\boldsymbol{\alpha}, \textbf{n}} \otimes I_{M} + I_{\textbf{n}} \otimes Q_{\theta,M}$},
  {$G_{\boldsymbol{\alpha}, \textbf{n}}$} denoting the spatial operator, showing a two-level diagonal-times–nonsymmetric Toeplitz structure, {$I_{\textbf{n}}$ } being the identity of size $N_1N_2$. Also for $d=2$, the conditioning of the resulting matrix is strongly influenced by the discretization parameters, particularly by the ratios {\(  \gamma_1=\Delta t / (\Delta x_{1})  {^{\alpha_1}} \) and \(  \gamma_2=\Delta t / (\Delta x_{2})  {^{\alpha_2}} \)}. In this work, we restrict our attention to the case of a rectangular domain. However, the present analysis can be used in the more general setting of $\Omega \subset {[a_1,b_1]\times [a_2,b_2]}$, by using e.g. the immersed domain approach in \cite{Rosita-immersed}; see also references therein.
 \\
The aim of the present paper is twofold: first, to develop and analyze an effective preconditioning strategy specifically tailored for the one-sided SFDE; and second, to complete the theoretical framework by performing a comprehensive spectral and singular value analysis of both the original and preconditioned matrix sequences. Our analysis is based on the Generalized Locally Toeplitz (GLT) theory \cite{GaroniSerra2017,GaroniSerra2018}, which provides a powerful framework for describing the asymptotic spectral distribution of matrix sequences.
\par
The work is organized as follows. In Section \ref{sec:spectr-tools} we set the notation and we introduce the necessary tools, including the essentials of the GLT theory.
Section \ref{sec:analysis} contains the structural and spectral analysis of the relevant matrix sequences, while in Sections \ref{sec:prec} and \ref{sec:prec2D} we introduce proper preconditioners and we study their spectral behavior. In Section \ref{sec:num} we report spectral visualizations and numerical experiments, which are critically discussed.
Section \ref{sec:end} contains final remarks and a short discussion on open problems and future lines of investigation.
\section{Spectral tools}\label{sec:spectr-tools}
In this introductory section, we present the tools required to carry out the spectral analysis of the matrices arising in our study, derived in the theory of multilevel GLT matrix sequences. The non-block setting characterized by scalar-valued symbols is comprehensively discussed in the \cite{GaroniSerra2017,GaroniSerra2018}, whereas the block framework, associated with matrix-valued symbols, is treated in \cite{Barbarino2020a,Barbarino2020b}: the reader is referred to \cite{tutorial1,tutorial2} for application oriented tutorials on the practical use of the GLT theory. In the context of the current work, we focus on the case where the block size is $r=1$ and the dimensionality of the underlying spatial domain is $d=1,2$. Accordingly, our analysis is conducted within the framework of 1-level non-block GLT sequences. Nonetheless, we present the spectral tools in a general form to facilitate potential future extensions to higher-dimensional or block-structured problems.
\subsection{Notation and terminology}
\textbf{Matrices and matrix sequences}. Given a square matrix $A \in \mathbb{C}^{m \times m}$, we denote by $A^*$ its conjugate transpose and by $A^\dagger$ the Moore–Penrose pseudoinverse of $A$. Recall that $A^\dagger = A^{-1}$ whenever $A$ is invertible. The singular values and eigenvalues of $A$ are denoted respectively by $ \sigma_1(A), \cdots, \sigma_m(A) $ and $\lambda_1(A), \cdots, \lambda_m(A) $.\\
Regarding matrix norms, $\|\cdot\|$ refers to the spectral norm and for $1 \leq p \leq \infty$, the notation $\|\cdot\|_p$ stands for the Schatten $p$-norm defined as the $p$-norm of the vector of singular values. Note that the Schatten $\infty$-norm, which is equal to the largest singular value, coincides with the spectral norm $\|\cdot\|$; the Schatten 1-norm since it is the sum of the singular values is often referred to as the trace-norm; and the Schatten 2-norm coincides with the Frobenius norm. Schatten $p$-norms, as important special cases of unitarily invariant norms are treated in detail in a wonderful book by Bhatia \cite{Bhatia1997}.\\
Finally, the expression \textbf{matrix sequence} refers to any sequence of the form $\{A_n\}_n$, where $A_n$ is a square matrix of size $d_n$ with $d_n$ strictly increasing so that $d_n \to \infty$ as $n \to \infty$. A \textbf{r-block matrix sequence}, or simply a matrix sequence if $r$ can be deduced from context is a special $\{A_n\}_n$ in which the size of $A_n$ is $d_n = r\varphi_n$, with $r \geq 1 \in \mathbb{N}$ fixed and $\varphi_n \in \mathbb{N}$ strictly increasing.\\
\textbf{Multi-index notation}. To effectively deal with multilevel structures it is necessary to use multi-indices, which are vectors of the form $\textbf{i} = (i_1, \cdots, i_d) \in \mathbb{Z}^d$. The related notation is listed below.
\begin{itemize}
    \item $\mathbf{0}, \mathbf{1}, \mathbf{2}, \dots$ are vectors of all zeroes, ones, twos, etc.
    \item $\textbf{h} \leq \textbf{k}$ means that $h_r \leq k_r$ for all $r = 1, \cdots, d$. In general, relations between multi-indices are evaluated componentwise.
    \item Operations between multi-indices, such as addition, subtraction, multiplication, and division, are also performed componentwise.
    \item The multi-index interval $[\textbf{h}, \textbf{k}]$ is the set $\{\textbf{j} \in \mathbb{Z}^d : \textbf{h} \leq \textbf{j} \leq \textbf{k}\}$. We always assume that the elements in an interval $[\textbf{h}, \textbf{k}]$ are ordered in the standard lexicographic manner
  \[
\Biggr[ \cdots
\biggr[\Bigr[\bigl(j_1, \cdots, j_d\bigl)\Bigr]
_{j_d = h_d, \cdots, k_d}\biggr]
_{j_{d-1} = h_{d-1}, \cdots, k_{d-1}}
\cdots \Biggr]_{j_1 = h_1, \cdots, k_1}.
\]
    \item $\textbf{j} = \textbf{h}, \cdots, \textbf{k}$ means that $\textbf{j}$ varies from $\textbf{h}$ to $\textbf{k}$, always following the lexicographic ordering.
    \item $\textbf{m} \to \infty$ means that $\min(\textbf{m}) = \min_{j=1, \cdots, d} m_j \to \infty$.
    \item The product of all the components of $\textbf{m}$ is denoted as $\nu(\textbf{m}) := \prod_{j=1}^{d} m_j$.
\end{itemize}

A \textbf{multilevel matrix sequence} is a matrix sequence $\{A_{\textbf{n}} \}_n$ such that $n$ varies in some infinite subset of $\mathbb{N}$, $\textbf{n} = \textbf{n}(n)$ is a multi-index in $\mathbb{N}^d$ depending on $n$, and $\textbf{n} \to \infty$ when $n \to \infty$. This is typical of many approximations of differential operators in $d$ dimensions.\\
\textbf{Measurability.} All the terminology from measure theory, such as “measurable set”, “measurable function”, “almost everywhere (a.e.)”, etc., refers to the Lebesgue measure in $ \mathbb{R}^{t} $, denoted with $\mu_t $. A matrix-valued function $f : D \subseteq \mathbb{R}^{t} \to \mathbb{C}^{r \times r} $ is said to be measurable (resp., continuous,
Riemann-integrable, in $ L^p(D) $, etc.) if all its components $ f_{\alpha\beta} : D \to \mathbb{C} $,  $ \alpha, \beta = 1, \dots, r $, are measurable (resp., continuous, Riemann-integrable, in $ L^p(D) $, etc.). If $ f_m, f : D \subseteq \mathbb{R}^{t} \to \mathbb{C}^{r \times r} $ are measurable, we say that
$ f_m $ converges to $ f $ in measure (resp., a.e., in $ L^p(D) $, etc.) if $ (f_m)_{\alpha\beta} $ converges to
$ f_{\alpha\beta} $ in measure (resp., a.e., in $ L^p(D) $, etc.) for all $ \alpha, \beta = 1, \cdots, r $.
\subsection{Distribution and clustering}
This subsection presents the notions of distribution and clustering of a matrix sequence, both in the sense of the eigenvalues and singular values, along with some related definitions and results. In what follows, $ C_c(\mathbb{R}) $ is the space of continuous complex-valued functions with bounded support on $ \mathbb{R} $,
and $ C_c(\mathbb{C}) $ is defined in the same way.
\begin{definition}\label{def:sing and eig}
Let $\{A_n\}_n$ be a matrix sequence, where $A_n$ is of size $d_n$, and let $\psi: D \subset \mathbb{R}^{t} \to \mathbb{C}^{r \times r}$ be a measurable function defined on a set $D$ with $0 < \mu_t(D) < \infty$.
\begin{itemize}
    \item We say that $\{A_n\}_n$ has an \textbf{(asymptotic) singular value distribution} described by $\psi$, written as $\{A_n\}_n \sim_{\sigma} \psi$, if
    \begin{equation*}
        \lim_{n \to \infty} \frac{1}{d_n} \sum_{i=1}^{d_n} F(\sigma_i(A_n)) = \frac{1}{\mu_t(D)} \int_D \displaystyle\frac{\sum_{i=1}^{r} F(\sigma_i(\psi(\textbf{x})))}{r} d\textbf{x}, \quad \forall F \in C_c(\mathbb{R}).
    \end{equation*}
    \item We say that $\{A_n\}_n$ has an \textbf{(asymptotic) spectral (eigenvalue) distribution} described by $\psi$, written as $\{A_n\}_n \sim_{\lambda} \psi$, if
    \begin{equation*}
        \lim_{n \to \infty} \frac{1}{d_n} \sum_{i=1}^{d_n} F(\lambda_i(A_n)) = \frac{1}{\mu_t(D)} \int_D \displaystyle\frac{\sum_{i=1}^{r} F(\lambda_i(\psi(\textbf{x})))}{r} d\textbf{x}, \quad \forall F \in C_c(\mathbb{C}).
    \end{equation*}
\end{itemize}
If $\psi$ describes both singular value and eigenvalue distribution of $\{A_n\}_n$, we write $\{A_n\}_n \sim_{\sigma, \lambda} \psi$.
\end{definition}
The informal meaning behind the spectral distribution definition is the following: if $\psi$
 is continuous, then a suitable ordering of the eigenvalues $\{\lambda_j(A_n)\}_{j=1,\cdots,d_n}$,
 assigned in correspondence with an equispaced grid on $D$, reconstructs approximately the $r$ surfaces $
 \textbf{x} \to \lambda_i(\psi(\textbf{x})), i =1,...,r.$ For instance, in the simplest case where $t =1$ and $D =[a,b]$,
$d_n = nr$, the eigenvalues of $A_n$ are approximately equal - up to few potential outliers - to $
 \lambda_i (\psi(x_j)),$ where $x_j = a + j \frac{(b-a)}{n}, j =1,\cdots,n, i =1,\cdots,r.$ If $t =2$ and $D = [a_1,b_1] \times [a_2,b_2]$, $d_n = n^2r$, the eigenvalues of $A_n$ are approximately equal - again up to few potential outliers - to $
 \lambda_i (\psi(x_{j_1},y_{j_2})), i =1,\cdots,r,$
 where $ x_{j_1} = a_1 + j_1 \frac{(b_1 - a_1)}{n}, y_{j_2} = a_2 + j_2 \frac{(b_2 - a_2)}{n}, j_1, j_2 =1,\cdots,n.$ For $t \geq 3$, a similar reasoning applies.\\
 Now let us proceed to the notion of clustering. We recall that the $\varepsilon$-expansion of $S \subseteq \mathbb{C}$
 is the set $ B(S, \varepsilon) := \bigcup_{z\in S} B(z, \varepsilon),$ where $
 B(z, \varepsilon) := \{w \in \mathbb{C} : |w - z| < \varepsilon\}$
 is the complex disk with center $z$ and radius $\varepsilon > 0$. \\

\begin{definition}  Let $\{A_n\}_n$ be a matrix sequence, with $A_n$ of size $d_n$, and let $S \subseteq \mathbb{C}$ be nonempty and closed. The sequence $\{A_n\}_n$ is \textit{strongly clustered} at $S$ in the sense of the eigenvalues if, for any $\varepsilon > 0$ and as $n$ tends to infinity, the number of eigenvalues of $A_n$ outside $B(S, \varepsilon)$ is bounded by a constant $q_{\varepsilon}$ independent of $n$. In symbols,
\begin{equation*}
q_{\varepsilon}(n, S) := \# \{ j \in \{1, \cdots, d_n\} : \lambda_j(A_n) \notin B(S, \varepsilon) \} = O(1), \quad \text{as } n \to \infty.
\end{equation*}
The sequence $\{A_n\}_n$ is \textit{weakly clustered} at $S$ if, for each $\varepsilon > 0$,
\begin{equation*}
q_{\varepsilon}(n, S) = o(d_n), \quad \text{as } n \to \infty.
\end{equation*}
\end{definition}
A corresponding definition can be given for the singular values, with $S \subseteq \mathbb{R}^+$. If $\{A_n\}_n$ is strongly or weakly clustered at $S$ and $S$ is not connected, the connected components of $S$ are called \textit{sub-clusters}.
\par
The case of spectral single point clustering, where $S$ consists of a single complex number $s$, retains special significance in the theory of preconditioning.\\
In the following remark, we clarify the deep connection between the two definitions above; in fact, clustering can be seen as a special case of distribution. First, recall that given a measurable function $g : D \subseteq \mathbb{R}^{t} \to \mathbb{C}$, the \textit{essential range} of $g$ is defined as $\mathcal{E}\mathcal{R}(g) := \{ z \in \mathbb{C} : \mu_t (\{ g \in B(z, \varepsilon) \} )> 0 \text{ for all } \varepsilon > 0 \},$ where $\{ g \in B(z, \varepsilon) \} := \{ x \in D : g(x) \in B(z, \varepsilon) \}.$ Generalizing the concept to a matrix-valued function $\psi : D \subset \mathbb{R}^{t} \to \mathbb{C}^{r \times r}$, the essential range is the union of the essential ranges of the eigenvalue functions $\lambda_i(\psi)$, for $i = 1, \cdots, r$.
\par
\begin{remark}\label{rem:clustering}
If $\{A_n\}_n \sim_\lambda \psi$, then $\{A_n\}_n$ is weakly clustered at the essential range of $\psi$ (see \cite{Golinskii2007}, Theorem 4.2). Furthermore, if $\mathcal{E}\mathcal{R}(\psi) = \{s\}$ with $s$ a fixed complex number, then $\{A_n\}_n \sim_\lambda \psi$ iff $\{A_n\}_n$ is weakly clustered at $s$ in the sense of the eigenvalues.\\
These statements can be translated to the singular value setting as well, with obvious minimal modifications. For instance, if $\mathcal{E}\mathcal{R}(|\psi|) = s$ with $s$ a fixed nonnegative number, then $\{A_n\}_n \sim_\sigma \psi$ iff $\{A_n\}_n$ is weakly clustered at $s$ in the sense of the singular values.
\end{remark}
\subsection{Matrix Structures} \label{subsec:matrix_structures}
In this subsection we introduce the three types of matrix structures that constitute the basic building blocks of the GLT $*$-algebra.
\par
\textbf{Zero-distributed sequences:} Zero-distributed sequences are defined as matrix sequences $\{A_n\}_n$ such that $\{A_n\}_n \sim_\sigma 0$. Note that, for any $r \geq 1$, $\{A_n\}_n \sim_\sigma 0$ is equivalent to
$\{A_n\}_n \sim_\sigma O_r$, where $O_r$ is the $r \times r$ zero matrix. The following theorem, taken from \cite{GaroniSerra2017,  Serra2001}, provides a useful characterization for detecting this type of sequences. We use
the natural convention $1/\infty = 0$.
\begin{theorem}
Let $\{A_n\}_n$ be a matrix sequence, with $A_n$ of size $d_n$. Then
\begin{itemize}
    \item $\{A_n\}_n \sim_\sigma 0$ if and only if $A_n = R_n + N_n$ with ${\operatorname{rank}(R_n)}/{d_n} \to 0$ and $ ||N_n|| \to 0$ as $n \to \infty$;
    \item $\{A_n\}_n \sim_\sigma 0$ if there exists $p \in [1,\infty]$ such that $
        \frac{\|A_n\|_p}{(d_n)^{1/p}} \to 0$ as $n \to \infty.$
\end{itemize}
\end{theorem}
\textbf{ Multilevel block Toeplitz matrices:} Given \(\textbf{n} \in \mathbb{N}^d \), a matrix of the form
\begin{equation*}
[A_{\textbf{i}-\textbf{j}}]_{\textbf{i,j=1}}^{\textbf{n}} \in \mathbb{C}^{\nu(\textbf{n})r \times \nu(\textbf{n})r},
\end{equation*}
with blocks \( A_\textbf{k} \in \mathbb{C}^{r \times r} \), \( \textbf{k} \in \{-(\textbf{n}-1), \dots, \textbf{n}-1\} \), is called a multilevel block Toeplitz matrix, or, more precisely, a \( d \)-level \( r \)-block Toeplitz matrix.\\
Given a matrix-valued function \( f : [-\pi,\pi]^d \to \mathbb{C}^{r \times r} \) belonging to \( L^1([- \pi, \pi]^d) \), the \( \mathbf{n} \)-th Toeplitz matrix associated with \( f \) is defined as
\begin{equation*}
T_\mathbf{n}(f):=[\hat{f}_{\mathbf{i-j}}]_{\mathbf{i,j=1}}^{\mathbf{n}} \in \mathbb{C}^{\nu(\mathbf{n})r \times \nu(\mathbf{n})r},
\end{equation*}
where
\begin{equation*}
\hat{f}_\mathbf{k} = \frac{1}{(2\pi)^d} \int_{[-\pi,\pi]^d} f(\mathbf{\theta}) e^{-\hat{\iota} (\mathbf{k}, \mathbf{\theta)}} d\mathbf{\theta} \in \mathbb{C}^{r \times r}, \quad \mathbf{k} \in \mathbb{Z}^d,
\end{equation*}
are the Fourier coefficients of \( f \), in which \( \hat{\iota} \) denotes the imaginary unit, the integrals are computed componentwise, and \( (\mathbf{k}, \mathbf{\theta}) = k_1\theta_1 + \cdots + k_d\theta_d \). Equivalently, \( T_\mathbf{n}(f) \) can be expressed as
\begin{equation*}
T_\mathbf{n}(f) =
\sum_{|j_1|<n_1} \cdots \sum_{|j_d|<n_d} [J_{n_1}^{(j_1)}
\otimes \cdots \otimes J^{(j_d)}_{n_d}] \otimes \hat{f}(j_1, \cdots, j_d),
\end{equation*}
where \( \otimes \) denotes the Kronecker tensor product between matrices, and \( J^{(l)}_{m} \) is the matrix of order \( m \) whose \( (i,j) \) entry equals 1 if \( i - j = l \) and zero otherwise.\\
 \( \{T_{\mathbf{n}}(f)\}_{\mathbf{n}\in\mathbb{N}^d} \) is the family of (multilevel block) Toeplitz matrices associated with \( f \), which is called the generating function.
\par
\textbf{Block diagonal sampling matrices:} Given $d \geq 1$, $\mathbf{n} \in \mathbb{N}^d$ and a function $a : [0,1]^d \to \mathbb{C}^{r \times r}$, we define the multilevel block diagonal sampling matrix $D_{\mathbf{n}}(a)$ as the block diagonal matrix
\begin{equation}\label{eq:dna}
    D_{\mathbf{n}}(a) = \underset{\mathbf{i=1,\dots,n}}{\operatorname{diag}} a \left( \mathbf{\frac{i}{n}} \right) \in \mathbb{C}^{\nu(\mathbf{n})r \times \nu(\mathbf{n})r}.
\end{equation}
\subsection{The $\ast$-Algebra of multilevel block GLT sequences}\label{subsec:glt_algebra}
Let $d,r \geq 1$ be fixed integers. A $d$-level $r$-block GLT sequence, or simply a GLT
sequence if we do not need to specify either $d$ or $r$, is a special $d$-level  $r$-block matrix sequence
equipped with a measurable function $
    \kappa :[0,1]^d \times[-\pi,\pi]^d \to \mathbb{C}^{r\times r}, d \geq 1,$
called the GLT symbol. The symbol is essentially unique, in the sense that if $\kappa, \varsigma$ are two symbols of the same
GLT sequence, then $\kappa = \varsigma$ a.e. We write \( \{A_n\}_n \sim_{\mathrm{GLT}} \kappa \) to denote that $\{A_n\}_n$ is a GLT
sequence with symbol $\kappa$.\\
It can be proven that the set of multilevel block GLT sequences is the $*$-algebra
generated by the three classes of sequences defined in Subsection \ref{subsec:matrix_structures}: zero-distributed,
multilevel block Toeplitz and block diagonal sampling matrix sequences. The GLT class
satisfies several algebraic and topological properties that are treated in detail in \cite{GaroniSerra2017, GaroniSerra2018, Barbarino2020a, Barbarino2020b}. Here, we focus on the main operative properties listed below that represent a complete characterization of GLT sequences, equivalent to the full constructive definition.
\par
\textbf{GLT 1.} If $\{A_n\}_n \sim_{\mathrm{GLT}} \kappa$ then $\{A_n\}_n \sim_{\sigma} \kappa$ in the sense of Definition \eqref{def:sing and eig}, with $t =2d$ and
$D=[0,1]^d \times[-\pi,\pi]^d$. If moreover each $A_n$ is Hermitian, then $\{A_n\}_n \sim_{\lambda} \kappa$, again in
the sense of Definition \eqref{def:sing and eig} with $t =2d$.
\par
\textbf{GLT 2.} We have
\begin{itemize}
    \item  $\{T_{\mathbf{n}}(f)\}_\mathbf{n} \sim_{\mathrm{GLT}} \kappa(\mathbf{x},\mathbf{\theta}) = f(\mathbf{\theta}) \quad \text{if } f : [-\pi,\pi]^d \to \mathbb{C}^{r\times r} \text{ is in } L^1([-\pi,\pi]^d); $
   \item $\{D_\mathbf{n}(a)\}_\mathbf{n} \sim_{\mathrm{GLT}} \kappa(\mathbf{x},\mathbf{\theta}) = a(\textbf{x}) \quad \text{if } a : [0,1]^d \to \mathbb{C}^{r\times r} \text{ is Riemann-integrable}; $
    \item $\{Z_n\}_n \sim_{\mathrm{GLT}} \kappa(\mathbf{x},\mathbf{\theta}) = O_r \text{ if and only if } \{Z_n\}_n \sim_{\sigma} 0.$
\end{itemize}
\textbf{GLT 3.} If $\{A_n\}_n \sim_{\mathrm{GLT}} \kappa$ and $\{B_n\}_n \sim_{\mathrm{GLT}} \varsigma$ then
\begin{itemize}
    \item $\{A_n^*\}_n \sim_{\mathrm{GLT}} \kappa^*; $
    \item $ \{\alpha A_n+\beta B_n\}_n \sim_{\mathrm{GLT}} \alpha \kappa+\beta \varsigma \quad \forall \alpha,\beta \in \mathbb{C}; $
    \item $ \{A_n B_n\}_n \sim_{\mathrm{GLT}} \kappa \varsigma; $
    \item $ \{A_n^\dagger\}_n \sim_{\mathrm{GLT}} \kappa^{-1} \quad \text{provided that } \kappa \text{ is invertible a.e.}$
\end{itemize}
\textbf{GLT 4.} $\{A_n\}_n \sim_{\mathrm{GLT}} \kappa$ if and only if there exist $\{B_{n,j}\}_n \sim_{\mathrm{GLT}} \kappa_j$ such that
$ \{\{B_{n,j}\}_n\}_j \xrightarrow{a.c.s. \text{ wrt } j} \{A_n\}_n \quad \text{and} \quad \kappa_j \to \kappa \text{ in measure}
$ (for the notion of approximating class of sequences and the related a.c.s. topology refer to \cite{GaroniSerra2017}).
\textbf{GLT 5} If $\{A_n\}_n \sim_{\mathrm{GLT}} \kappa$ and $A_n = X_n + Y_n$, where
\begin{itemize}
    \item each $X_n$ is Hermitian,
    \item $\| Y_n \|_2 = o(n^{1/2})$,
\end{itemize}
then
$
\{ A_n \}_n \sim_{\lambda} \kappa.
$
Note that, by \textbf{GLT 1}, it is always possible to obtain the singular value distribution from the GLT symbol, while the eigenvalue distribution can be deduced only if the involved
matrices are Hermitian. However, interesting tools are available for matrix sequences in which the non-Hermitian part is somehow negligible; see \cite{Golinskii2007,BarbarinoSerra2020} and references therein.
The main result is reported below; more advanced tools can be found in \cite{BarbarinoSerra2020}.
\begin{theorem}{\rm \cite[Corollary 4.1]{GaroniSerra2017} \cite[Corollary 2.3]{GaroniSerra2018} }
Let \( \{X_n\}_n \), \( \{Z_n\}_n \) be matrix sequences, with \( X_n, Z_n \) of size \( d_n \), and set \( Y_n = X_n + Z_n \). Assume that the following conditions are met
\begin{enumerate}
    \item \( \|X_n\|, \|Z_n\| \leq C \) for all \( n \), where \( C \) is a constant independent of \( n \);
    \item every \( X_n \) is Hermitian and \( \{X_n\}_n \sim_{\lambda} \kappa \);
    \item \( \|Z_n\|_1 = o(d_n) \).
\end{enumerate}
Then \( \{Y_n\}_n \sim_{\lambda} \kappa \). Moreover, if \( \|Z_n\|_1 = O(1) \), the range of \( \kappa \) is a strong cluster for the eigenvalues of \( \{Y_n\}_n \).
\end{theorem}
\section{Discretized problem and GLT analysis}\label{sec:analysis}
In this section, we present the discretization of the variable-coefficient space-fractional diffusion equation \eqref{eq2}, posed on the spatial domain $\Omega = {[a_1, b_1] \times [a_2, b_2] \subset \mathbb{R}^2}$. The equation is first reformulated on a discrete space-time grid, leading to a sequence of linear systems governed by structured coefficient matrices that encode the nonlocal behavior of the model. We then analyze the spectral properties of these matrices using the GLT framework. In particular, we compute the GLT symbol of the resulting matrix sequence, which characterizes the asymptotic distribution of eigenvalues and singular values.
\subsection{Discretization}

We first treat the one-dimensional space setting and then the two-dimensional space setting.

\subsubsection{One-dimensional}
Let $M$ and $N$ be positive integers representing temporal steps and the number of interior spatial, respectively. The spatial step size is defined as $\Delta x = {(b_1 - a_1)}/{(N + 1)},$ and the temporal step size is $\Delta t = {T}/{M}$. The spatial grid points are given by $x_i = a_1 + i \Delta x$, for $0 \leq i \leq N+1$, while the temporal grid points are defined as $ t_{m} = m \Delta t$, for $0 \leq m \leq M$.\\
To approximate the left-sided Riemann--Liouville fractional derivative, we adopt the shifted Grünwald finite difference scheme \cite{MeerschaertTadjeran2006}, which is widely recognized for its improved stability and accuracy compared to the standard Grünwald approach. The continuous fractional operator \( {}_a D_x^\alpha u(x,t) \) is discretized at each interior point \( x_i \) using the following expression
  {\begin{equation}\label{equ4}
{}_a D_x^\alpha u(x, t)\big|_{x = x_i, \, t = t_{m}}= \frac{1}{(\Delta x)^{\alpha}} \sum_{k=0}^{i+1} g_{i - k + 1}^{(\alpha)} u(x_{k},t_{m})+ \mathcal{O}(\Delta x),
\end{equation}}
where
\begin{equation}\label{equ5}
  {g_0^{(\alpha)} = 1, \quad g_{k+1}^{(\alpha)} = \left(1 - \frac{\alpha + 1}{k+1} \right) g_{k}^{(\alpha)}, \quad k \ge 0.}
\end{equation}
This recurrence can be rewritten in closed form as
\begin{equation}
\label{eq:g-binomial}
  {g^{(\alpha)}_k = (-1)^{k} \binom{\alpha}{k}, \quad k \ge 0,}
\end{equation}
where
\[
\binom{\alpha}{k} = \frac{\alpha (\alpha-1) \cdots (\alpha-k+1)}{k!}
\]
denotes the generalized binomial coefficient.
\begin{proposition}\label{alpha coff}
Let $\alpha \in (1,2)$ and $g_k^{(\alpha)}$ be defined as in \eqref{equ5}. Then we have
  {\[
\begin{cases}
g_0^{(\alpha)} = 1, \quad g_1^{(\alpha)} = -\alpha, \quad
g_0^{(\alpha)} > g_2^{(\alpha)} > g_3^{(\alpha)} > \dots > 0, \\[4pt]
\displaystyle \sum_{k=0}^{\infty} g_k^{(\alpha)} = 0,\quad
\displaystyle \sum_{k=0}^{n} g_k^{(\alpha)} < 0, \quad n \ge 1.
\end{cases}
\]}
\end{proposition}
When combined with the variable diffusion coefficient \( a(x) \), the spatial discretization for the fractional diffusion term at time level   {\( m \)} becomes
\begin{equation}\label{equ6}
  {a(x) \, {}_a D_x^\alpha u(x,t)\big|_{x = x_i, \, t = t_{m}}
= \frac{a(x_{i})}{(\Delta x)^{\alpha}} \sum_{k=0}^{i+1} g_{i - k + 1}^{(\alpha)} u(x_{k},t_{m}) + \mathcal{O}(\Delta x).
}\end{equation}
  {Finally,  by coupling  the \(\theta\)-method \cite{pan2016} to discretize the temporal derivative in  \eqref{eq2} with  the approximation  in \eqref{equ6}}, we obtain the following time-stepping scheme
  {\begin{equation}\label{equ7}
\frac{u_i^{(m+1)} - u_i^{(m)}}{\Delta t} = \theta  \left( \frac{a(x_i)}{(\Delta x)^\alpha} \sum_{k=0}^{i+1} g_k^{(\alpha)} u_{i-k+1}^{(m+1)} \right) + (1 - \theta) \left( \frac{a(x_i)}{(\Delta x)^\alpha} \sum_{k=0}^{i+1} g_k^{(\alpha)} u_{i-k+1}^{(m)} \right) + f_i^{(m,\theta)},
\end{equation}}
where $0\leq\theta\leq 1$ is a weighting parameter, \( u_i^{(m)} \) approximates \( u(x_i, t_{m}) \),
and   {$f_i^{(m,\theta)} =  \theta f(x_i,t_{m+1}) + (1-\theta)f(x_i,t_{m}) $}.
The value of \( \theta \) determines the nature of the method: \( \theta = 0 \) corresponds to the Forward Euler method, \( \theta = 1 \) yields the Backward Euler method, and \( \theta = \frac{1}{2} \) gives the Crank--Nicolson scheme. \par
  {By defining $u_N^{(m)}=\left [u_i^{(m)}\right ]^T_{i=1,\ldots,N}$, that is the numerical approximation of the solution \( u(x, t_m) \) at all points of the interior spatial grid   {\( x_i, {i=1, \ldots,N}\)}, and accordingly, the vector $f_N^{(m,\theta)}=\left[f_i^{(m,\theta)}\right]^T_{i=1,\ldots,N}$, and by multiplying both sides of \eqref{equ7} by \( \Delta t \) and simplifying, we obtain}
\begin{equation}\label{equ8}
u_N^{(m+1)}-u_N^{(m)}=   {\theta G_{\alpha,N} \, u_N^{(m+1)} + (1 - \theta)  G_{\alpha,N} \, u_N^{(m)} + \Delta t  f_N^{(m,\theta)}, \quad m = 1, \dots, M,}
\end{equation}
where
\begin{equation}\label{D*G}
    G_{\alpha,N}= \gamma {\mathcal D}_{N}(a) \overline{G}_{\alpha, N},
    \end{equation}
with $\gamma = {\Delta t}/{(\Delta x)^{\alpha}}$,
${\mathcal D}_{N}(a) = \mathrm{diag} \left( a ( x_i) \right)_{i=1,\ldots,N}$  denoting the \( N \)-th diagonal sampling matrix and
\begin{equation} \label{eq:Gbar}
 \overline{G}_{\alpha, N} =
\begin{bmatrix}
g_1^{(\alpha)} & g_0^{(\alpha)}             & 0             & \cdots & 0 \\
g_2^{(\alpha)} & g_1^{(\alpha)} & g_0^{(\alpha)}            & \ddots & \vdots \\
\vdots & g_2^{(\alpha)} & g_1^{(\alpha)} & \ddots & 0 \\
\vdots        & \ddots        & \ddots        & \ddots & g_0^{(\alpha)}  \\
g_{N}^{(\alpha)} &\cdots & \cdots & g_2^{(\alpha)} & g_1^{(\alpha)}
\end{bmatrix} \in \mathbb{R}^{N \times N},
\end{equation}
  {denoting the Hessenberg Toeplitz matrix, that is
\begin{equation}\label{eq8bis}
\left (I_N - \theta  G_{\alpha,N} \right ) u_N^{(m+1)} = \left (I_N + (1 - \theta)  G_{\alpha,N} \right ) u_N^{(m)}
+ \Delta t  f_N^{(m,\theta)}.
\end{equation}}
In terms of notations, notice that ${\mathcal D}_{N}(a)$ coincides with $D_N(a)$ as in (\ref{eq:dna}) with $d=r=1$ if the domain of definition of the function $a$ is $[0,1]$. In general $[a_1,b_1]\neq [0,1]$ and hence ${\mathcal D}_{N}(a)=D_N(\hat a)$ with $\hat a(\hat x)=a(a_1 + (b_1-a_1)\hat x)$ and $x=a_1 + (b_1-a_1)\hat x$, $\hat x\in [0,1]$: this specification is useful for applying the GLT axioms when deducing the Weyl distributions of the related matrix sequences, as indicated in Remark \ref{rem:extending2}.

Equation \eqref{eq8bis} leads to the following   {all-at-once formulation} linear system
  {\begin{align}
\begin{bmatrix}
I_{N} - \theta  G_{\alpha,N} & & & \\
-(I_{N} + (1 - \theta)  G_{\alpha,N}) & I_{N} - \theta  G_{\alpha,N} & & \\
& \ddots & \ddots & \\
& & -(I_{N} + (1 - \theta)  G_{\alpha,N}) & I_{N} - \theta  G_{\alpha,N}
\end{bmatrix}
\begin{bmatrix}
u_N^{(1)} \\
u_N^{(2)} \\
\vdots \\
u_N^{(M)}
\end{bmatrix} & \nonumber \\
& \hskip -6cm =
\begin{bmatrix}
\Delta t \, f_N^{(1,\theta)} + (I_{N}  + (1 - \theta)  G_{\alpha,N}) u_N^{(0)}  \\
\Delta t \, f_N^{(2,\theta)} \\
\vdots \\
\Delta t \, f_N^{(M,\theta)}
\end{bmatrix}. \label{equ9}
\end{align}}
This system can be compactly expressed using Kronecker products in the following form
  {\begin{equation}\label{equ10}
\bar{A} \bar{\mathbf{u}} = \bar{\mathbf{f}},
\end{equation}
}%
where
  {\begin{align}
\bar{A} &= (I_M \otimes (I_N - \theta  G_{\alpha,N}) - Z_{M} \otimes (I_N + (1 - \theta)  G_{\alpha,N})) \in \mathbb{R}^{NM \times NM}, \label{eq:Abar} \\
%
Z_{N} &=
\begin{bmatrix}
0      &        &        &        \\
1      & 0      &        &        \\
 & \ddots & \ddots &        \\
     &  & 1      & 0
\end{bmatrix}\in \mathbb{R}^{M \times M} \nonumber,
\end{align}
}
and $\bar{\mathbf{u}}, \bar{\mathbf{f}} \in \mathbb{R}^{NM \times 1}$.
%
%
The coefficient matrix in \eqref{eq:Abar} can therefore be expressed compactly using Kronecker products as
  {\[
\bar{A} = H_{\theta,M} \otimes   G_{\alpha,N} + H_{M} \otimes I_N.
\]}
where the matrices   {$H_{M}$} and $H_{\theta,M}$ are defined as follows
\[
  {H_{M}} =
\begin{bmatrix}
1           &             &             &        \\
-1          & 1           &             &        \\
            & \ddots      & \ddots      &        \\
            &             & -1          & 1
\end{bmatrix}
\in \mathbb{R}^{M \times M}, \quad
  {H_{\theta,M} = -
\begin{bmatrix}
\theta      &             &             &        \\
1 - \theta  & \theta      &             &        \\
            & \ddots      & \ddots      &        \\
            &             & 1 - \theta  & \theta
\end{bmatrix}
}\in \mathbb{R}^{M \times M}.
\]
To calculate the solution of \eqref{equ10}
in a parallel pattern form, we exchange the order of the Kronecker product in   {\( \bar{A}\)}. This leads to an equivalent system
  {
\begin{equation}\label{equ11}
\widetilde{A}  \widetilde{\mathbf{u}}=\widetilde{\mathbf{f}},
\end{equation}}
where
  {\[
\widetilde{A} = G_{\alpha,N} \otimes H_{\theta,M} + I_N \otimes H_{M},
\]}
and   {\( \widetilde{\mathbf{f}} \)} is obtained by reshaping   {\( \bar{\mathbf{f}} \in \mathbb{R}^{NM} \)} into an \( N \times M \) matrix and then stacking its rows into a single column $MN$- by-$1$ vector. Equivalently,   {\( \widetilde{\mathbf{f}} \)} is formed by stacking the columns of the transpose of the reshaped matrix.\\
By pre-multiplying both sides of \eqref{equ11} with \( (I_N \otimes   {H_{\theta,M})}^{-1} \), we obtain an equivalent formulation
  {\begin{equation}\label{equ12}
{A_{\alpha,N,M}} \mathbf{u} = \mathbf{f},
\end{equation}
}where
  {\[
{A_{\alpha,N,M}} =  G_{\alpha,N} \otimes I_M + I_N \otimes H_{\theta,M}^{-1} H_{M} =  G_{\alpha,N} \otimes I_M + I_N \otimes Q_{\theta, M},
\]
}
with the matrix \( Q_{\theta, M} = H_{\theta,M}^{-1} H_{M}\) defined as
\begin{equation}\label{Q mat}
  {Q_{\theta, M} = -
\begin{bmatrix}
\frac{1}{\theta} & & &  &  \\
\frac{-1}{\theta^{2}} & \frac{1}{\theta} & &  &  \\
\frac{-(\theta-1)}{\theta^3} & \frac{-1}{\theta^{2}} & \frac{1}{\theta} &  &  \\
\vdots & \ddots & \ddots & \ddots &  \\
\frac{-(\theta-1)^{M-2}}{\theta^{M}} & \cdots & \frac{-(\theta-1)}{\theta^3} & \frac{-1}{\theta^{2}} & \frac{1}{\theta}
\end{bmatrix}
}
\in \mathbb{R}^{M \times M}.
\end{equation}
The transformed right-hand side is given by
  {\(
\mathbf{f} = (I_N\otimes H_{\theta,M})^{-1} \, \widetilde{\mathbf{f}}=(I_N\otimes H_{\theta,M}^{-1}) \, \widetilde{\mathbf{f}}.
\)}
\subsubsection{  {Two-dimensional} }
Let $ N_1 $, $N_{2}$ and $ M $ be positive integers representing the number of interior spatial and temporal steps, respectively. The spatial step sizes are defined as $\Delta x_{j} = {(b_j - a_j)}/{(N_{j} + 1)}$, $j=1,2$.
The temporal step size is $\Delta t = {T}/{M}.$ The spatial grid points are given by  $x_{1,i_1} = a_1 + i_1 \Delta x_{1}$, for $0 \leq i_1 \leq N_{1}+1$, and $x_{2,i_2} = a_2 + i_2 \Delta x_{2}$, for $0 \leq i_2 \leq N_{2}+1$, while the temporal grid points are defined as $t_{m} = m \Delta t$, for $0 \leq m \leq M$.\\
As in the one-dimensional setting, we adopt the shifted Grünwald finite difference scheme \cite{MeerschaertTadjeran2006} to approximate the left-sided Riemann-Liouville fractional derivatives.
The continuous fractional operator \( {}_{a_1} D_{x_{1}}^{\alpha_1} u(x_1,x_2,t) \) is discretized at each interior point \( x_{1,i_1} \) using the expression below
\begin{equation}\label{equ44}
{}_{a_1} D_{x_{1}}^{\alpha_1}u(x_{1}, x_{2},t) \big|_ {x_1= x_{1,i_1}, x_2 = x_{2,i_2}, t = t_{m}} = \frac{1}{(\Delta x_{1})^{\alpha_1}} \sum_{k=0}^{i_1+1} g_{i_1 - k + 1}^{(\alpha_1)} u(x_{1,k},x_{2,i_2},t_{m})+ \mathcal{O}(\Delta x_1),
\end{equation}
with related properties of $g^{(\alpha_1)}_{k}$  already given in the one-dimensional setting in Proposition~\ref{alpha coff}.
Similarly, to discretize the \( {}_{a_2} D_{x_{2}}^{\alpha_2} u(x_1,x_2,t) \) at each interior point \( x_{2,i_2} \), we deduce the following relation
\begin{equation}\label{equ66}
{}_{a_2} D_{x_{2}}^{\alpha_2} u(x_{1}, x_{2},t) \big|_{x_1= x_{1,i_1}, x_2 = x_{2,i_2}, t = t_{m}}= \frac{1}{(\Delta x_{2})^{\alpha_2}} \sum_{l=0}^{i_2+1} g_{i_2 - l + 1}^{(\alpha_2)} u(x_{1,i_2},x_{2,l},t_{m})+ \mathcal{O}(\Delta x_2),
\end{equation}
with related properties of $g^{(\alpha_2)}_{l}$  already given in the one-dimensional setting in Proposition~\ref{alpha coff}.
\par
Applying the spatial discretization \eqref{equ44}–\eqref{equ66} and the temporal discretization, we use the \(\theta\)-method \cite{pan2016} on \eqref{eq2}, which yields the following time-stepping scheme
\begin{align}\label{equ777}
\frac{u_{i_1,i_2}^{(m+1)} - u_{i_1,i_2}^{(m)}}{\Delta t} & = \ \theta \ a(x_{1,i_1},x_{2,i_1}) \left( \frac{1}{(\Delta x_1)^{\alpha_1}} \sum_{k=0}^{i_1+1} g_k^{(\alpha_1)} u_{i_1-k+1,i_2}^{(m+1)} + \frac{1}{(\Delta x_2)^{\alpha_2}} \sum_{l=0}^{i_2+1} g_l^{(\alpha_2)} u_{i_1,i_2-l+1}^{(m+1)}\right) \nonumber\\
& \hskip -2cm + (1 - \theta) \ a(x_{1,i_1},x_{2,i_1}) \left( \frac{1}{(\Delta x_1)^{\alpha_1}} \sum_{k=0}^{i_1+1} g_k^{(\alpha_1)} u_{i_1-k+1,i_2}^{(m)} + \frac{1}{(\Delta x_2)^{\alpha_2}} \sum_{l=0}^{i_2+1} g_l^{(\alpha_2)} u_{i_1,i_2-l+1}^{(m)} \right) + f_{i_1,i_2}^{m,\theta},
\end{align}
where
 $u^{(m)}_{i_1,i_2}$ approximates $u(x_{1,i_1}, x_{2,i_2}, t_m)$ and
$f_{i_1,i_2}^{(m,\theta)} =  \theta f(x_{1,i_1},x_{2,i_2},t_{m+1}) + (1-\theta)f(x_{1,i_1},x_{2,i_2},t_{m})$.
By defining the vectors $u_{\textbf{n}}^{(m)}=
\left [u_{i_1,i_2}^{(m)}\right ]^T_{i_1=1,\ldots,N_1, \ i_2=1,\ldots,N_2}$, $f_{\textbf{n}}^{(m,\theta)}=\left[f_{i_1,i_2}^{(m,\theta)}\right]^T_{i_1=1,\ldots,N_1, \ i_2=1,\ldots,N_2}$ according to the lexicographic order and by
multiplying both sides of \eqref{equ777} by \( \Delta t \) and simplifying, we obtain
\begin{equation}\label{equ8}
u_{\textbf{n}}^{(m+1)}-u_{\textbf{n}}^{(m)}= \theta  G_{\boldsymbol{\alpha},\textbf{n}} \, u_{\textbf{n}}^{(m+1)} + (1 - \theta)   G_{\boldsymbol{\alpha},\textbf{n}} \, u_{\textbf{n}}^{(m)} + \Delta t  f_{\textbf{n}}^{(m,\theta)}, \quad m = 1, \dots, M,
\end{equation}
where
\begin{align}
     G_{\boldsymbol{\alpha},\textbf{n}} & = D_\textbf{n}(a) U_{\boldsymbol{\alpha},\textbf{n}}
    \in \mathbb{R}^{(N_{1} N_{2})\times(N_{1} N_{2})},\label{D*U} \\
     {U_{\boldsymbol{\alpha},\textbf{n}}} & {=   \gamma_1 \overline{G}_{\alpha_1,N_1}\otimes I_{N_2} + I_{N_1} \otimes \gamma_2 \overline{G}_{\alpha_2,N_2}
    \in \mathbb{R}^{(N_{1} N_{2})\times(N_{1} N_{2})}},\label{U mat 2D} \\
 {\mathcal D}_{\textbf{n}}(a) &= \underset{i_2=1,\dots,N_2}{\underset{i_1=1,\dots,N_1}{\operatorname{diag}}}\left( a\left(x_{1,i_1},x_{2,i_2}\right)\right)\in \mathbb{R}^{(N_{1} N_{2})\times(N_{1} N_{2})},
\end{align}
with $\overline{G}_{\alpha_1,N_1}$ and $\overline{G}_{\alpha_2,N_2}$ defined as in~\eqref{eq:Gbar}.
In terms of notations, notice that ${\mathcal D}_{\textbf{n}}(a)$ coincides with $D_{\textbf{n}}(a)$ as in (\ref{eq:dna}) with $d=2, r=1$ if the domain of definition of the function $a$ is $[0,1]^2$. In general $[a_1,b_1]\times [a_2,b_2]\neq [0,1]^2$ and hence ${\mathcal D}_{\textbf{n}}(a)=D_{\textbf{n}}(\hat a)$ with $\hat a(\hat x_1,\hat x_2)=a(a_1 + (b_1-a_1)\hat x_1,a_2 + (b_2-a_2)\hat x_2)$ and $(x_1,x_2)=(a_1 + (b_1-a_1)\hat x_1,a_2 + (b_2-a_2)\hat x_2)$, $(\hat x_1,\hat x_2)\in [0,1]^2$: as in the one-dimensional setting, the latter setting is a key point for deducing the Weyl distributions of the related matrix sequences, using the GLT axioms; see Remark \ref{rem:extending2}.
\par
By combining the equations at all levels  into the single comprehensive linear system, known as the all-at-once formulation, and repeating verbatim the manipulations considered in the one-dimensional setting, which are not formally influenced by the two-dimensional space variable,  we obtain the equivalent formulation
\begin{equation}\label{equ122}
{A_{\boldsymbol{\alpha},\mathbf{n},M}} \mathbf{u} = \mathbf{f},
\end{equation}
with matrix
\[ {A_{\boldsymbol{\alpha},\mathbf{n},M}} = G_{\boldsymbol{\alpha},\textbf{n}} \otimes I_M + I_{\textbf{n}}  \otimes Q_{\theta, M}, \]
where $Q_{\theta, M}$ is define as in \eqref{Q mat},
and transformed right-hand side
\[
\mathbf{f} = (I_{\textbf{n}} \otimes H_\theta)^{-1} \, \widetilde{\mathbf{f}} =(I_{\textbf{n}} \otimes H_\theta^{-1}) \, \widetilde{\mathbf{f}},
\]
where
 \( \widetilde{\mathbf{f}} \) is obtained by reshaping the original right-hand side
 \[ \bar{\mathbf{f}} =
\begin{bmatrix}
\Delta t \, \mathbf{f}_{\textbf{n}}^{(1,\theta)} + (I_{\textbf{n}}  + (1 - \theta) G_{\boldsymbol{\alpha},\textbf{n}}  u_{\textbf{n}}^{(0)} \\
\Delta t \, \mathbf{f}_{\textbf{n}}^{(2,\theta)} \\
\vdots \\
\Delta t \, \mathbf{f}_{\textbf{n}}^{(M,\theta)}
\end{bmatrix}
\in \mathbb{R}^{N_{1}N_{2}M \times 1},
\]
into an \( N_{1}N_{2} \times M \) matrix and then stacking its rows into a single column $MN_{1}N_{2}$- by-$1$ vector. Equivalently, \( \widetilde{\mathbf{f}} \) is formed by stacking the columns of the transpose of the reshaped matrix.\\
\subsection{GLT, singular value, and eigenvalue analysis}
In the current section, we first introduce a basic definition, and then proceed with the GLT analysis of the coefficient matrix sequences derived from the previous approximation methods. More precisely the section is divided into two parts: in the first we study the matrix sequences arising from the space discretizations, both in 1D and 2D. Then, in the second part, we provide the GLT analysis for the global matrix sequences, incorporating also the term coming from the time discretization.
\subsubsection{GLT analysis of the space matrix sequences}
\begin{definition}\label{def1}
The Wiener class is the set of functions $f(\theta) = \sum_{k \in \mathbb{Z}^d} f_k e^{i \langle k, \theta \rangle}$ such that $\sum_{k \in \mathbb{Z}^d} |f_k| < \infty$. The Wiener class forms a sub-algebra of the continuous $2\pi$-periodic functions defined on $[-\pi, \pi]^d$.
\end{definition}
In this view, we begin by reporting the symbol associated with the matrix sequence $\overline{G}_{\alpha, N}$ defined in equation~\eqref{D*G}.
\begin{proposition}\label{symb f_alpha}
Let $\alpha \in (1,2)$. The symbol associated with $\overline{G}_{\alpha, N}$ belongs to the Wiener class, and
\[
f_\alpha(\xi_1) = \sum_{k=-1}^{\infty} g^{(\alpha)}_{k+1} e^{ik\xi_1}
 {=  e^{-i\xi_1}
\left(1 + e^{i(\xi_1 + \pi)}\right)^{\alpha}, \quad \xi_1 \in [-\pi,\pi].
}
\]
\end{proposition}
\begin{proof}
Let us observe that $\overline{G}_{\alpha, N} = [\, g^{(\alpha)}_{i-j+1} \,]_{i,j=1}^N$ with $g^{(\alpha)}_k = 0$ for $k<0$,  and let us define the function  $\sum_{k=-1}^{\infty} g^{(\alpha)}_{k+1} e^{ik\xi_1}$.
When $\alpha \in (1,2)$, it is easy to see that $f_{\alpha}(\xi_1)$ lies in the Wiener class.
In detail, from Proposition~\ref{alpha coff} we know that
\[
 {g^{(\alpha)}_0 = 1 > 0,\quad g^{(\alpha)}_1 = -\alpha < 0,\quad g^{(\alpha)}_k > 0 \ \ \forall k \ge 2,\quad g^{(\alpha)}_k = 0 \ \ \forall k < 0.}
\]
Then
\begin{align}
\sum_{k=-1}^{\infty} \big| g_{k+1}^{(\alpha)} \big|
&= \big| g_{0}^{(\alpha)} \big| + \big| g_{1}^{(\alpha)} \big| + \sum_{k=1}^{\infty} \big| g_{k+1}^{(\alpha)} \big| \nonumber \\
&=g_{0}^{(\alpha)} +\left( - g_{1}^{(\alpha)}\right) + \sum_{k=1}^{\infty} g_{k+1}^{(\alpha)}
\quad \text{(since $g_{0}^{(\alpha)}>0$, $g_{1}^{(\alpha)}<0$, $g_{k+1}^{(\alpha)}>0$ for $k\ge 1$)} \nonumber \\
&= g_{0}^{(\alpha)}  - g_{1}^{(\alpha)} + \sum_{k=1}^{\infty} g_{k+1}^{(\alpha)}. \label{eq:1}
\end{align}
From Proposition~\eqref{alpha coff}, we have $\sum_{k=0}^{\infty} g_{k}^{(\alpha)} = 0.$
Therefore,
\begin{equation}\label{eq:2}
 {\sum_{k=1}^{\infty} g_{k+1}^{(\alpha)} = \sum_{k=2}^{\infty} g_{k}^{(\alpha)} = \sum_{k=0}^{\infty} g_{k}^{(\alpha)} - g_{0}^{(\alpha)} - g_{1}^{(\alpha)}
 =  - g_{0}^{(\alpha)} - g_{1}^{(\alpha)}
= -1 + \alpha
}
\end{equation}
%
 {Now, plugging \eqref{eq:2} in \eqref{eq:1}, we find
\begin{equation}
\sum_{k=-1}^{\infty} \big| g_{k+1}^{(\alpha)} \big|
=  g_{0}^{(\alpha)} -g_{1}^{(\alpha)}  + ( -1 + \alpha )
= 1 + \alpha + ( -1 + \alpha ) = 2\alpha.
\end{equation}}
Since this sum is finite, $f_{\alpha}(\xi_1)$ belongs to the Wiener class for $\alpha \in (1,2)$.
\par
To obtain an explicit formula for the symbol $f_{\alpha}(\xi_1)$, let us recall the definition of $g^{(\alpha)}_k$ given in~\eqref{eq:g-binomial} and rewrite $f_{\alpha}(\xi_1)$ in the following manner
\[
f_{\alpha}(\xi_1) = \sum_{k=-1}^{\infty} g^{(\alpha)}_{k+1} e^{ik\xi_1}
= \sum_{k=0}^{\infty} g^{(\alpha)}_k e^{i(k-1)\xi_1}
= \sum_{k=0}^{\infty}  {(-1)^{k}} \binom{\alpha}{k} e^{i(k-1)\xi_1}.
\]
We now write $(-1)^k$ as $e^{ik\pi}$ so that
\[
 {f_{\alpha}(\xi_1) = \sum_{k=0}^{\infty} \binom{\alpha}{k} e^{i(k-1)\xi_1} e^{ik\pi}
= e^{-i\xi_1} \sum_{k=0}^{\infty} \binom{\alpha}{k} e^{ik(\xi_1+\pi)}.
}
\]
Applying the well-known binomial series
\[
(1+z)^{\alpha} = \sum_{k=0}^{\infty} \binom{\alpha}{k} z^k, \quad z \in \mathbb{C}, \ |z| \le 1,\ \alpha>0,
\]
with $z = e^{i(\xi_1+\pi)}$, we finally deduce a very informative expression i.e.
\[
 {f_{\alpha}(\xi_1) = e^{-i\xi_1} \left( 1 + e^{i(\xi_1+\pi)} \right)^{\alpha}.
}
\]
\end{proof}
In the two-dimensional case,  {$\overline{G}_{\alpha_1,N_{1}}$} has the same symbol as obtained in Proposition~\ref{symb f_alpha}.
Similarly, the matrix  {$\overline{G}_{\alpha_2, N_{2}}$} has the same Toeplitz-like structure as $\overline{G}_{\alpha_1,N_{1}}$, with parameter $\alpha_2$ in place of $\alpha_1$.
Therefore, its associated symbol is given by
\[
 {f_{\alpha_2}(\xi_{2}) =  {e^{-i \xi_{2}}} \left( 1 + e^{i(\xi_{2} + \pi)} \right)^{\alpha_2},
\quad \xi_{2} \in [-\pi, \pi],
}
\]
which is the exact analogue of $f_{\alpha}(\xi_{1})$ obtained in Proposition~\ref{symb f_alpha}.
{
\begin{proposition}\label{symb q_theta}
Let $Q_{\theta, M} = H_{\theta, M}^{-1} H_M$.
Then,
\begin{equation}\label{Q symb}
    \{ Q_{\theta, M} \}_M \sim_{\mathrm{GLT}} q_{\theta} (\xi_{3})
\end{equation}
where
$q_{\theta} (\xi_{3})=\frac{1 - e^{i\xi_3}}{\theta + (1-\theta) e^{i\xi_3}}$.
\end{proposition}
\begin{proof}
In both the one and two-dimensional cases, both $H_M$ and $H_{\theta,M}$ are Toeplitz matrices generated by the symbols
\[
\rho_1(\xi_3) = 1 - e^{i\xi_3}, \qquad \rho_2(\xi_3) = \theta + (1-\theta) e^{i\xi_3},
\]
respectively, where $\xi_3 \in [-\pi, \pi]$. According to the GLT algebra and the standard properties of Toeplitz sequences \cite[Section 6.2]{GaroniSerra2017}, the matrix sequence $Q_{\theta, M} = H_{\theta, M}^{-1} H_M$ is associated with the GLT symbol
\[
q_{\theta} (\xi_{3}) = \frac{\rho_1(\xi_3)}{\rho_2(\xi_3)}.
\]
Therefore,
\[
    \{ Q_{\theta, M} \}_M \sim_{\mathrm{GLT}} q_{\theta} (\xi_{3}).
\]
\end{proof}}
%
\begin{theorem}\label{1D th:}
If \( a : [0, 1] \to \mathbb{R} \) is continuous a.e., then
\begin{itemize}
    \item[$a1)$]
    $\{  {{\mathcal D}_{N}(a)\overline{G}_{\alpha,N}} \}_N \sim_{\mathrm{GLT}} a(x)f_\alpha(\xi_{1})$,
    \item[$a2)$]
    $\{  {{\mathcal D}_{N}(a)\overline{G}_{\alpha,N}} \}_N \sim_{\sigma} a(x)f_\alpha(\xi_{1})$,
\end{itemize}
with $f_\alpha$ as in Proposition~\ref{symb f_alpha}.

\end{theorem}
\begin{proof}
To prove~$a1)$, we first recall that the argument relies on the fundamental properties of the GLT class, in particular the second part in~\textbf{GLT~2} and the third part stated in~\textbf{GLT~3}. More precisely,  we have
\[
{\mathcal D}_N(a)=D_{N}(a) =\underset{i=1,2,\ldots,N }{ \text{diag}}
a(x_{i})
\]
and hence
\[
\{{{\mathcal D}_N(a)=D_N(a)}\}_N \sim_{\mathrm{GLT}} a(x)
\]
by using \textbf{GLT~2}. Furthermore, by Proposition~\ref{symb f_alpha} via axiom \textbf{GLT~2}, we find
\[
\{ {\overline{G}_{\alpha,N}}\}_N \sim_{\mathrm{GLT}} f_\alpha(\xi_{1}).
\]
In conclusion by equation \eqref{D*G} and \textbf{GLT~3}, we obtain $a1)$ i.e.
\[
    \{  {{\mathcal D}_{N}(a) \overline{G}_{\alpha, N}\}_N}  \sim_{\mathrm{GLT}} a(x)f_\alpha(\xi_{1}),
\]
so that, by invoking \textbf{GLT~1}, we infer $a2)$.
\end{proof}
And what about
\begin{itemize}
    \item[$a3)$] $\{  {{\mathcal D}_{N}(a)\overline{G}_{\alpha,N}} \}_N \sim_{\lambda} a(x)f_\alpha(\xi_{1})$?
\end{itemize}
Since \( G_{\alpha,N} \) is not Hermitian, its eigenvalue distribution cannot be immediately deduced from \textbf{GLT~1}, even if other tools are available \cite{eig-nonH}.

To illustrate the relationship between the eigenvalue distributions and the associated symbols, we present two representative examples for which Figures~\ref{fig:1d_const}–\ref{fig:1d_var} indicate that the question in $a3)$ has a positive answer, which is clear in the constant coefficient case. Indeed, since the convergence is slow, we believe that the answer is positive even in the variable coefficient case.

In the one-dimensional case, we first consider the constant-coefficient matrix
\(\overline{G}_{\alpha,N}\) and compare its eigenvalues with sampled values of the symbol
\( f_{\alpha}(\xi_{1}) \) (Figure~\ref{fig:1d_const}). The two graphs almost superpose and hence we observe a positive answer to the question in $a3)$ when $a(x)$ is a constant function. The fact holds also formally because of the results in \cite{eig-nonH}, where the authors study the canonical distribution of Toeplitz matrix sequences for a rich class of parametric generating functions which includes $f_\alpha$ for every $\alpha\in (1,2)$. Hence
\[
\{ \overline{G}_{\alpha,N} \}_N \sim_{\lambda} f_\alpha(\xi_{1}).
\]
We then examine the variable-coefficient case \( D_{N}(a) \,\overline{G}_{\alpha,N} \) with \( a(x) = x^{2} +1\), comparing the eigenvalues with sampled values of the symbol \( a(x) f_{\alpha}(\xi_{1}) \) (Figure~\ref{fig:1d_var}). Also in this more complicate setting, the numerics suggest that the answer to the question in $a3)$ is positive as reported
Figure~\ref{fig:1d_var}.
We also tried the more degenerate case of ${a(x)=x^2}$ and also in this case the agreement is remarkable. In fact, since $a(x)$ is continuous with range $[0,1]$, the range of $a(x)f_\alpha(\xi_{1})$ for fixed $\xi_1$ is exactly the segment connecting the complex zero with the value $f_\alpha(\xi_{1})$: as a consequence the total range of the GLT symbol $a(x)f_\alpha(\xi_{1})$ is exactly the convex hull of the range of $f_\alpha(\xi_{1})$. On the other hand, since $1+x^2$ is continuous and shows a range equal to $[1,2]$ the eigenvalues fill a disconnected (highly nonconvex) region of the form $y\cdot$range$(f_\alpha)$, with $y\in [1,2]$ in agreement with the behavior depicted in Figure~\ref{fig:1d_var}.
Conversely, the theoretical proof is not available yet and it will be the subject of future investigations, starting from the key tools given in \cite{eig-nonH}.
\begin{figure}
\centering
\includegraphics[width=\textwidth]{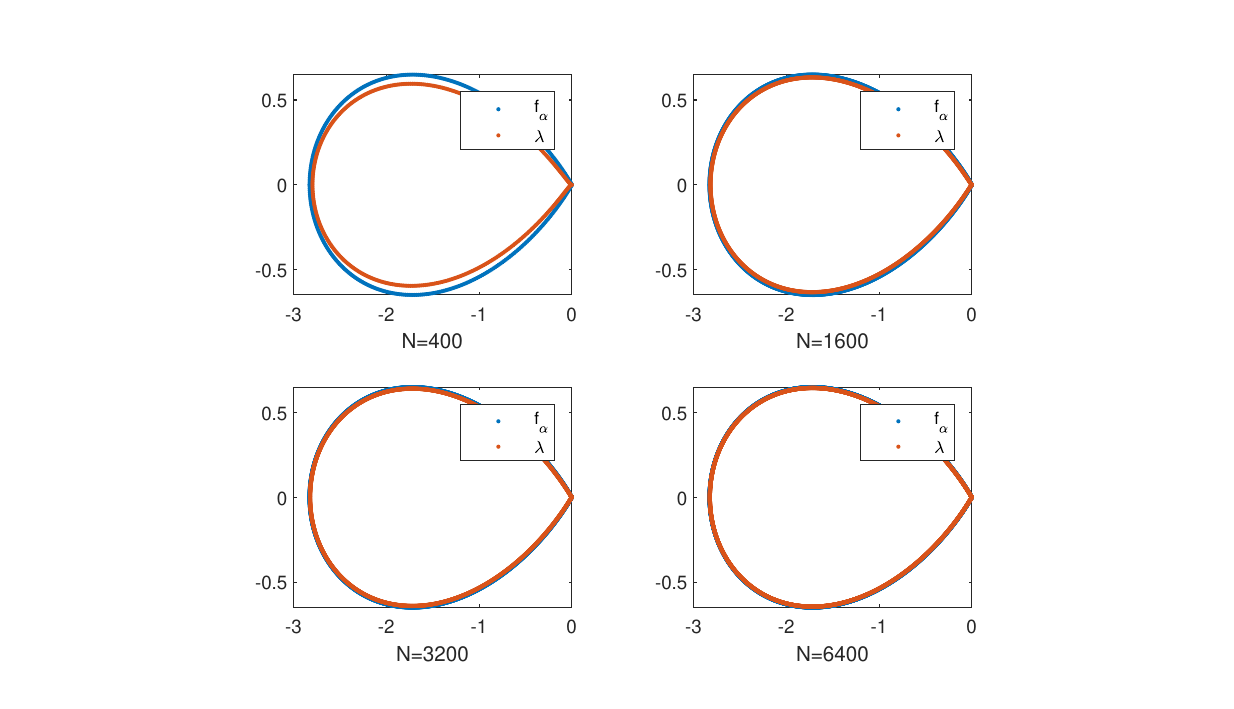}
\caption{Comparison between the symbol $f_{\alpha}(\xi_{1})$ (blue stars) and the eigenvalues of ${\overline{G}}_{\alpha,N}$ (red stars) for different values of $N$ - case $\alpha = 1.5$. 
}
\label{fig:1d_const}
\end{figure}
\begin{figure}
\centering
\includegraphics[width=\textwidth]{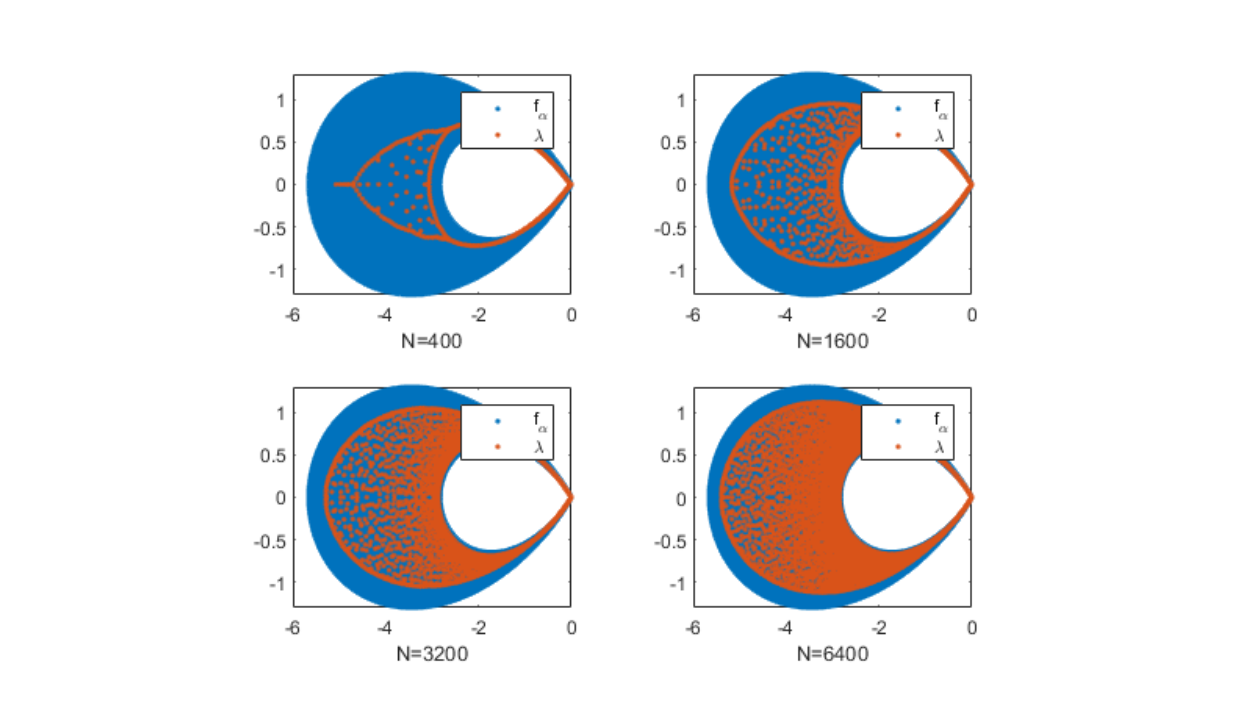}
\caption{Comparison between the symbol $a(x)f_{\alpha}(\xi_{1})$ (blue stars) and the eigenvalues of $D_{N}(a) \,\overline{G}_{\alpha,N}$ (red stars) for different values of $N$ - case $a(x) = x^2+1$, $\alpha = 1.5$. 
}
\label{fig:1d_var}
\end{figure}
\begin{remark}\label{rem:extending1}
 The inclusion of the factor $\gamma=\Delta t / (\Delta x)^{\alpha}$ in the analysis is possible and indeed not difficult, so extending
the scope of Theorem \ref{1D th:}. Assuming that $a : [0, 1] \to \mathbb{R}$ is continuous a.e. and that
\[
\lim_{\Delta t,\Delta x\rightarrow 0^+}\gamma(\Delta t,\Delta x)={\lim_{\Delta t,\Delta x\rightarrow 0^+}\frac{\Delta t}{(\Delta x)^\alpha}=}\ \gamma^*
\]
we have
\begin{eqnarray}\label{gen1-1Dth}
\{ G_{\alpha,N} \}_N & \sim_{\mathrm{GLT}} & \gamma^* a(x)f_\alpha(\xi_{1}), \\ \label{gen2-1Dth}
\{ G_{\alpha,N} \}_N & \sim_{\sigma} & \gamma^* a(x)f_\alpha(\xi_{1}), \\ \label{gen3-1Dth}
\{ {\overline{G}_{\alpha,N}}\}_N & \sim_{\lambda} & \gamma^* f_\alpha(\xi_{1}).
\end{eqnarray}
Alternatively, since we solve linear systems, one can get rid of the scaling factor very easily by writing in full generality
\begin{eqnarray}\label{gen1-bis-1Dth}
\{{\gamma}^{-1} G_{\alpha,N} \}_N & \sim_{\mathrm{GLT}} &  a(x)f_\alpha(\xi_{1}), \\ \label{gen2-bis-1Dth}
\{{\gamma}^{-1} G_{\alpha,N} \}_N & \sim_{\sigma} &  a(x)f_\alpha(\xi_{1}), \\ \label{gen3-bis-1Dth}
\{ {\gamma}^{-1}{\overline{G}_{\alpha,N}}\}_N & \sim_{\lambda} &  f_\alpha(\xi_{1}),
\end{eqnarray}
with the only hypothesis that $a : [0, 1] \to \mathbb{R}$ is continuous a.e. which is equivalent to the Riemann integrability of $a$.
We remark that the relations \eqref{gen1-bis-1Dth}, \eqref{gen2-bis-1Dth}, \eqref{gen3-bis-1Dth} are meaningful also when the {limit $\gamma^*$} is equal to zero, while relations \eqref{gen1-1Dth}, \eqref{gen2-1Dth}, \eqref{gen3-1Dth} become somehow trivial and not informative {for $\gamma^*=0$}.
\par
We finally notice that the statements $a1)$, $a2)$, (\ref{gen1-1Dth}), (\ref{gen2-1Dth}), (\ref{gen1-bis-1Dth}), (\ref{gen2-bis-1Dth}) are all valid if $a$ is assumed complex-valued. However, the interest in the complex-valued case is at the moment quite limited, since in the modeling of the problem the function $a$ is required nonnegative, often strictly positive.
\end{remark}
\par
Now we consider a similar analysis in the  {two-level} space setting,  where we have to consider the whole matrix $U_{\alpha,\mathbf{n}}$, by joining the tools in Theorem \ref{1D th:} and Remark \ref{rem:extending1}.
\begin{theorem}\label{2D th:}
{Let $\gamma_j=\Delta t / (\Delta x)^{\alpha_j}$, $j = 1,2$, with $\alpha_1, \alpha_2 \in (1,2)$.}
If $a : [0,1]^2 \to \mathbb{R}$ is continuous a.e. and
\[
\lim_{\Delta t,\Delta x\rightarrow 0^+}\gamma_j(\Delta t,\Delta x)= {\lim_{\Delta t,\Delta x\rightarrow 0^+} \frac{\Delta t}{(\Delta x)^\alpha_j} =}\  \gamma_j^*<\infty, \ \ \ j=1,2
\]
then
\begin{itemize}
\item [$b1)$] $\{G_{\boldsymbol{\alpha},\mathbf{n}}\}_\mathbf{n} \sim_{\mathrm{GLT}} a(x_1,x_2)\,\big(\gamma_1^* f_{\alpha_1}(\xi_1)
+ \gamma_2^* f_{\alpha_2}(\xi_2)\big)$,
\item [$b2)$] $\{G_{\boldsymbol{\alpha},\mathbf{n}}\}_\mathbf{n} \sim_{\sigma} a(x_1,x_2)\,\big(\gamma_1^* f_{\alpha_1}(\xi_1)
+ \gamma_2^* f_{\alpha_2}(\xi_2)\big)$,
\end{itemize}
{where
\[
f_{\alpha_1}(\xi_1) = { e^{-i\xi_1}}\big(1+e^{i(\xi_1+\pi)}\big)^{\alpha_1}, \quad
f_{\alpha_2}(\xi_2) =  {e^{-i\xi_2}}\big(1+e^{i(\xi_2+\pi)}\big)^{\alpha_2}, \quad
(\xi_1,\xi_2) \in [-\pi,\pi]^2,
\]
and where necessarily $\gamma_2^*=0$ if $\alpha_1>\alpha_2$, and alternatively $\gamma_1^*=0$ if $\alpha_2>\alpha_1$.}
In the special case where $\max\{\gamma_1^*,\gamma_2^*\}=0$, a normalization as in Remark \ref{rem:extending1} leads to
\begin{itemize}
\item [{$b1-1)$}] $\{{\gamma}^{-1}G_{\boldsymbol{\alpha},\mathbf{n}}\}_\mathbf{n} \sim_{\mathrm{GLT}} a(x_1,x_2)\,\big(\hat \gamma_1 f_{\alpha_1}(\xi_1)
+ \hat \gamma_2 f_{\alpha_2}(\xi_2)\big)$,
\item [{$b2-2)$}] $\{{\gamma}^{-1}G_{\boldsymbol{\alpha},\mathbf{n}}\}_\mathbf{n} \sim_{\sigma} a(x_1,x_2)\,\big(\hat \gamma_1 f_{\alpha_1}(\xi_1)
+ \hat \gamma_2 f_{\alpha_2}(\xi_2)\big)$,
\end{itemize}
with $\gamma=\max\{\gamma_1,\gamma_2\}$.
Here $\hat \gamma_1=1$ and $\hat \gamma_2=0$ if $\alpha_1>\alpha_2$, $\hat \gamma_2=1$ and $\hat \gamma_1=0$ if $\alpha_2>\alpha_1$, while the case $\alpha_1=\alpha_2$ leads to $\hat \gamma_1=1$ and $\hat \gamma_2\in (0,1)$ if $\gamma_1>\gamma_2$; $\hat \gamma_2=1$ and $\hat \gamma_1\in (0,1)$ if $\gamma_2>\gamma_1$.
\end{theorem}

\begin{proof}
It is sufficient to prove $b1)$, because $b2)$  follows from $b1)$ and \textbf{GLT~1}. From the definition in~\eqref{D*U}, the matrix sequence $\{G_{\boldsymbol{\alpha},\mathbf{n}}\}_\mathbf{n}$ can be expressed as
\[
\{G_{\boldsymbol{\alpha},\mathbf{n}}\}_\mathbf{n} = \{D_{\mathbf{n}}(a) \, U_{\boldsymbol{\alpha},\mathbf{n}}\},
\]
where $D_{\mathbf{n}}(a)$ is the diagonal sampling matrix and from~\eqref{U mat 2D} the matrix $U_{\boldsymbol{\alpha},\mathbf{n}}$ has the Kronecker sum form
\[
{U_{\boldsymbol{\alpha},\mathbf{n}} = \gamma_1 \overline{G}_{\alpha_1,N_1}\otimes I_{N_2} +
I_{N_1} \otimes \gamma_2 \overline{G}_{\alpha_2,N_2}
}
\]
%
where $\overline{G}_{\alpha_1,N_1}$ and $\overline{G}_{\alpha_2,N_2}$  are one-level Toeplitz matrices generated by the shifted Grünwald coefficients in the $x_1$- and $x_2$-directions, respectively.
{Thus, by the second part of property~\textbf{GLT~2}, we deduce that
\[
\{D_{\mathbf{n}}\}\sim_{\mathrm{GLT}} a(x_1,x_2).
\]}
Using the first part of property~\textbf{GLT~2}, we easily infer that the matrix sequences satisfy the relations
\[
\{\overline{G}_{\alpha_1,N_1}\}_{N_1} \sim_{\mathrm{GLT}}  f_{\alpha_1}(\xi_1), \qquad
\{\overline{G}_{\alpha_2,N_2}\}_{N_2} \sim_{\mathrm{GLT}}  f_{\alpha_2}(\xi_2).
\]
Invoking the $*$-algebra structure of the GLT class expressed by~\textbf{GLT~3}, we obtain that
\[
\{U_{\boldsymbol{\alpha},\mathbf{n}}\}_{n} \sim_{\mathrm{GLT}} \gamma_1 f_{\alpha_1}(\xi_1) + \gamma_2 f_{\alpha_2}(\xi_2).
\]
Finally, by applying axiom \textbf{GLT~3}, third item, to the sequences
$\{D_{\mathbf{n}}\}_{n}$ and $\{U_{\boldsymbol{\alpha},\mathbf{n}}\}_{n}$,
we conclude that
\[
\{G_{\boldsymbol{\alpha},\mathbf{n}}\}_{n} \sim_{\mathrm{GLT}}
a(x_1,x_2)\,\big( {f_{\alpha_1}(\xi_1) + f_{\alpha_2}(\xi_2)}\big).
\]
Then $b2)$ is an obvious consequence of $b1)$ and axiom \textbf{GLT~1}.
\par {Finally,  it is evident that $U_{\boldsymbol{\alpha},\mathbf{n}}$
can be written as}
 \begin{align*}
{   U_{\boldsymbol{\alpha},\mathbf{n}}} & = { \gamma_1 \overline{G}_{\alpha_1,N_1}\otimes I_{N_2} +
I_{N_1} \otimes \gamma_2 \overline{G}_{\alpha_2,N_2}   }\\
    & = {\gamma_1 \left ( \overline{G}_{\alpha_1,N_1}\otimes I_{N_2} +
    (\Delta x)^{\alpha_1-\alpha_2} I_{N_1} \otimes \overline{G}_{\alpha_2,N_2}
    \right)}
 \end{align*}
{so that the thesis follows if $\alpha_1 > \alpha_2$ and vice versa.}
\end{proof}
And what about
\begin{itemize}
\item [$b3)$] $\{G_{\boldsymbol{\alpha},\mathbf{n}}\}_{n} \sim_{\lambda} a(x_1,x_2)\,\big( {{\gamma_1^* f_{\alpha_1}(\xi_1) + \gamma_2^* f_{\alpha_2}(\xi_2)}}\big)$?
\end{itemize}
As in the one-dimensional case, the eigenvalue distribution is not an automatic GLT consequence for non-Hermitian matrices, but it is partially supported by the numerical evidence presented; see Figures~\ref{fig:2d_const}–\ref{fig:2d_var}.
In the two-dimensional case, we first consider the constant-coefficient matrix
 {\( {U}_{\boldsymbol{\alpha},\mathbf{n}} \)} in equation~\eqref{U mat 2D} and compare its eigenvalues with sampled values of the symbol  {\( f_{\alpha_1}(\xi_{1}) + f_{\alpha_2}(\xi_{2}) \)} (Figure~\ref{fig:2d_const}).
We then examine the variable-coefficient case  {\( {D}_{\mathbf{n}}(a) {U}_{\boldsymbol{\alpha},\mathbf{n}} \)} in equation~\eqref{D*U}, with \( a(x_{1},x_{2}) = x_{1}^{2} + x_{2}^{2}+1 \), comparing the eigenvalues with sampled values of the symbol
 {\( a(x_{1},x_{2}) \left( f_{\alpha_1}(\xi_{1}) + f_{\alpha_2}(\xi_{2}) \right) \)} (Figure~\ref{fig:2d_var}).
It is evident that the  {two-dimensional} setting is more intricate and we are not sure that the tools from \cite{eig-nonH} are useful as in the {one-dimensional} setting. The question remains open and it is left for future researches.
\begin{figure}
\centering
\includegraphics[width=\textwidth]{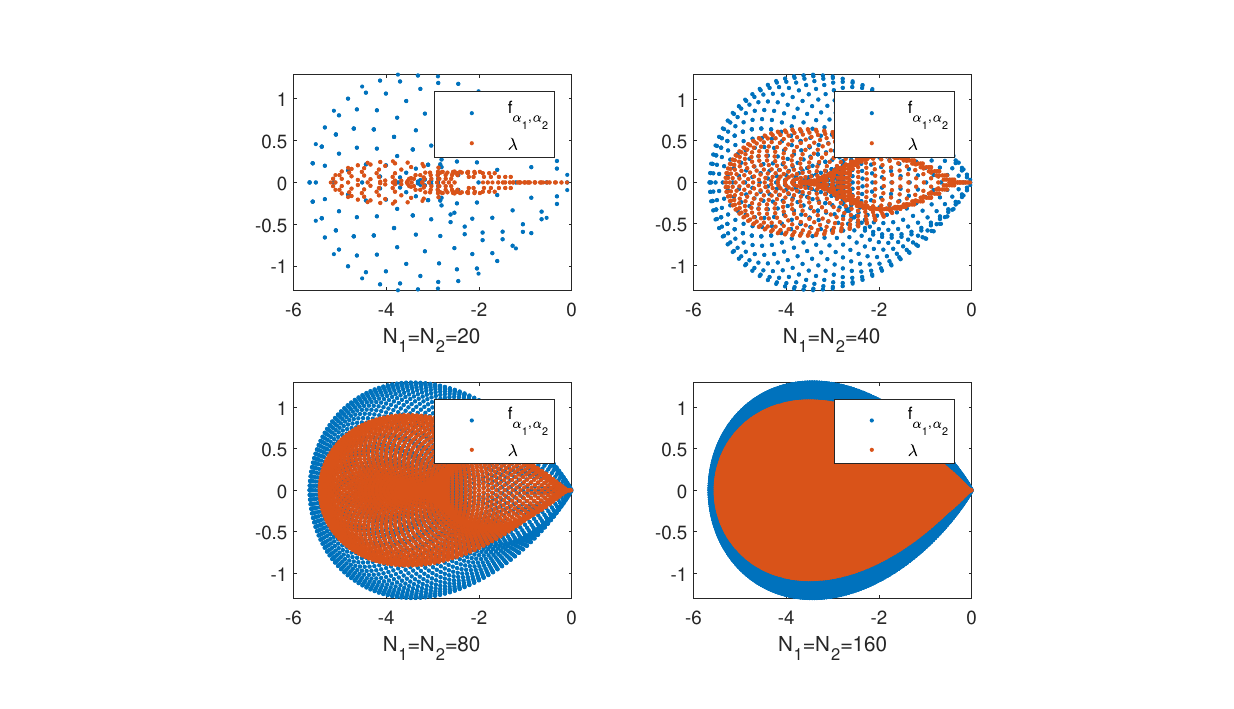}
\caption{Comparison between the symbol $ {f_{\alpha_1}(\xi_{1})+f_{\alpha_2}(\xi_{2})}$ (blue stars) and the eigenvalues of
{$\overline{G}_{\alpha_1,N_1}\otimes I_{N_2} + I_{N_1} \otimes \overline{G}_{\alpha_2,N_2}$} (red stars) for different values of $N_1=N_2$ - case $\alpha_1 = \alpha_2 = 1.5$, {$\gamma_1=\gamma_2=1$}.
}
\label{fig:2d_const}
\end{figure}
\begin{figure}
\centering
\includegraphics[width=\textwidth]{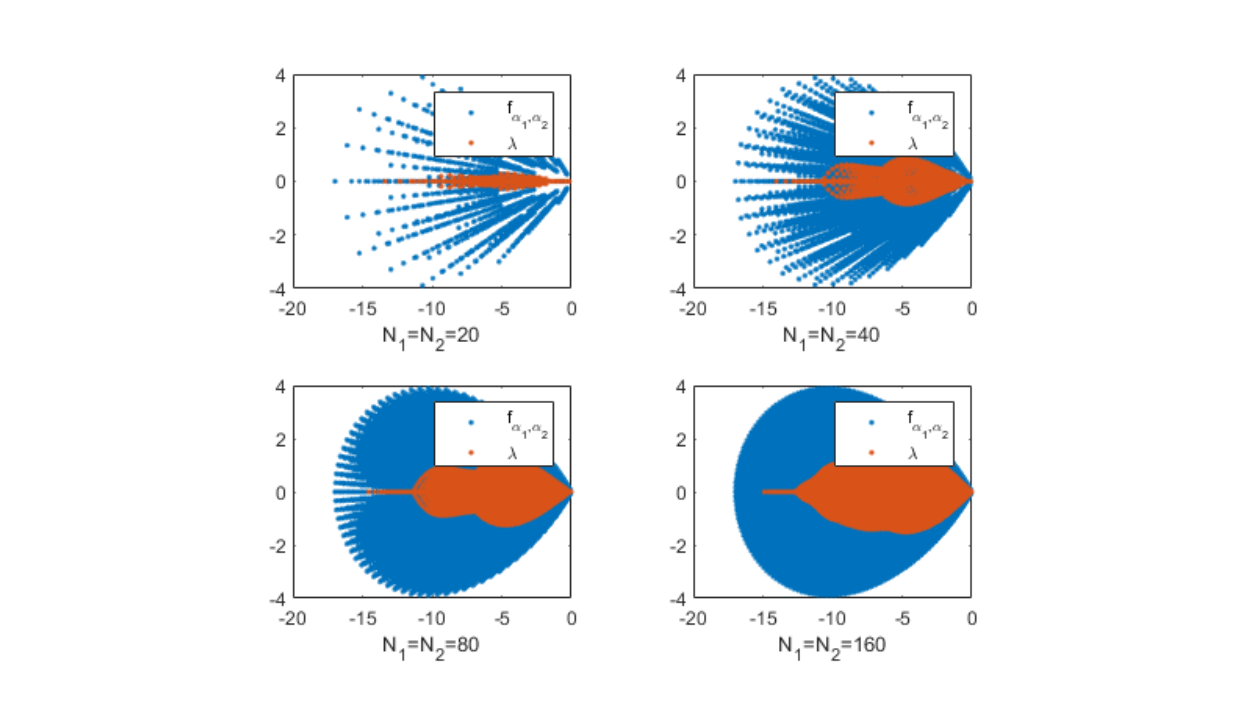}
\caption{Comparison between the symbol $a(x_1,x_2)(f_{\alpha_1}(\xi_{1})+f_{\alpha_2}(\xi_{2}))$ (blue stars) and the eigenvalues of $D_{\mathbf{n}}(a)({\overline{G}_{\alpha_1,N_1}\otimes I_{N_2} + I_{N_1} \otimes \overline{G}_{\alpha_2,N_2}})$ (red stars) for different values of $N_1=N_2$ - case $a(x_1,x_2) = x_1^2+x_2^2+1$, $\alpha_1 = \alpha_2 = 1.5$, {$\gamma_1=\gamma_2=1$.}
. 
}
\label{fig:2d_var}
\end{figure}
\begin{remark}\label{rem:extending2}
Looking back at the original formulation in (\ref{eq2-1d}), we observe that the original domain
{in the one-dimensional space case}
is $[a_1,b_1]$ and in general $[a_1,b_1]\neq [0,1]$.
For enlarging the scope of Theorem \ref{1D th:}, using the GLT theory, it is enough to consider ${\mathcal D}_{N}(a)=D_N(\hat a)$ with $\hat a(\hat x)=a(a_1 + (b_1-a_1)\hat x)$ and $x=a_1 + (b_1-a_1)\hat x$, $\hat x\in [0,1]$. In this way Theorem \ref{1D th:} and Remark \ref{rem:extending1} still hold with $\hat a$ instead of $a$ in the given statements.
In the two-dimensional setting, looking at (\ref{eq2}), again we have $[a_1,b_1]\times [a_2,b_2]\neq [0,1]^2$ in general.
However  ${\mathcal D}_{\textbf{n}}(a)=D_{\textbf{n}}(\hat a)$ with $\hat a(\hat x_1,\hat x_2)=a(a_1 + (b_1-a_1)\hat x_1,a_2 + (b_2-a_2)\hat x_2)$ and $(x_1,x_2)=(a_1 + (b_1-a_1)\hat x_1,a_2 + (b_2-a_2)\hat x_2)$, $(\hat x_1,\hat x_2)\in [0,1]^2$ and consequently Theorem \ref{2D th:} holds with the only change of  $\hat a$ instead of $a$.
\par
The flexibility of the GLT tools is of use in our context because also a general domain could be considered as far as it is Peano-Jordan measurable and in applications all the bounded domains have Riemann integrable characteristic function; see \cite[pp. 395–-399]{serra2003} for the first example and the idea of reduced GLT sequences in \cite[Section 3.1.4]{serra2006}, with the exhaustive theoretical development in \cite{reduced}.
\par
{We finally notice that the statements $b1)$, $b2)$ are valid if $a$ is assumed complex-valued. However, the interest in the complex-valued case is at the moment quite limited, since in the modeling of the problem the function $a$ is required nonnegative, often strictly positive.}
\end{remark}
\subsubsection{GLT analysis of the space-time matrix sequences}
We are ready to provide the GLT analysis for the space-time matrix sequences, both in 1D in space and 2D in space.
\begin{theorem}\label{1D orig mat}
Let $a:[0,1]\to\mathbb{R}$ be continuous a.e. and let $\gamma=\Delta t / (\Delta x)^{\alpha}$, $\gamma^*$ as in Remark \ref{rem:extending1}. Then
\begin{enumerate}
  \item[$c1)$] $\{ {A_{\alpha, N, M}}\}_{N,M} \sim_{\mathrm{GLT}}\gamma^* a(x)f_{\alpha}(\xi_1)+q_{\theta}(\xi_3)$,
  \item[$c2)$] $\{ {A_{\alpha, N,M}}\}_{N,M}\sim_{\sigma}       \gamma^* a(x)f_{\alpha}(\xi_1)+q_{\theta}(\xi_3)$,
\end{enumerate}
with ${A_{\alpha,N,M}}=G_{\alpha, N}\otimes I_M+I_N\otimes Q_{\theta,M}$, taking into account \eqref{equ11}-\eqref{equ12}.
\end{theorem}
\begin{proof}
It is enough to prove $c1)$, since $c2)$ can be directly inferred from $c1$ and axiom \textbf{GLT~1}.
From Theorem~\ref{1D th:} in the one-dimensional setting, we already know that the sequence of matrices
$\{G_{\alpha,N}\}_{N}$ is distributed in the GLT sense as $\gamma^* a(x)\,f_{\alpha}(\xi_1)$.
Moreover, by Proposition~\ref{symb q_theta}, the sequence $\{Q_{\theta,M}\}_{M}$ admits the GLT symbol
\[
q_{\theta}(\xi_3) \;=\; \frac{1-\mathrm{e}^{\mathrm{i}\xi_3}}{\theta+(1-\theta)\,\mathrm{e}^{\mathrm{i}\xi_3}},
\]
which represents the temporal discretization contribution arising from the $\theta$-method. Therefore, by combining these two GLT symbols through the Kronecker sum structure of $ {A_{\alpha,N,M}}$
and invoking the $*$-algebra property of the GLT class (\textbf{GLT~3}), we deduce that
\[
\{ {A_{\alpha,N,M}}\}_{N,M}=\{G_{\alpha,N}\otimes I_M+I_N\otimes Q_{\theta,M}\}_{N,M} \sim_{\mathrm{GLT}} \gamma^*a(x)f_{\alpha}(\xi_1)+q_{\theta}(\xi_3),
\]
which establishes the validity of $c1$. The statement $c2)$ then follows directly from $c1)$ and axiom \textbf{GLT~1}.
\end{proof}
\begin{theorem}\label{2D orig mat}
Let $a:[0,1]^2\to\mathbb{R}$ be continuous a.e. and let {$\gamma_j=\Delta t / (\Delta x)^{\alpha_j}$ and $\gamma^*_j$, $j=1,2$},
 as in Theorem \ref{2D th:}.
{Furthermore, let
\[
 {A_{\boldsymbol{\alpha},\mathbf{n},M}=G_{\boldsymbol{\alpha},\textbf{n}}\otimes I_M + I_{\mathbf{n}}\otimes Q_{\theta,M},}
\]
as defined in~\eqref{equ122}, with $f_{\alpha_1}$ and $f_{\alpha_2}$ taken from Proposition~\ref{symb f_alpha} and its $\alpha_2$–counterpart, and $q_{\theta}$ given by~\eqref{Q symb}.}
  Then
\begin{enumerate}
  \item[$d1)$] $ {\{A_{\boldsymbol{\alpha},\mathbf{n},M}\}_{\mathbf{n},M}}\sim_{\mathrm{GLT}}a(x_1,x_2)(\gamma_1^*f_{\alpha_1}(\xi_1)+\gamma_2^*f_{\alpha_2}(\xi_2))+q_{\theta}(\xi_3);$
  \item[$d2)$] $ {\{A_{\boldsymbol{\alpha},\mathbf{n},M}\}_{\mathbf{n},M}}\sim_{\sigma}a(x_1,x_2)(\gamma_1^* f_{\alpha_1}(\xi_1)+ \gamma_2^* f_{\alpha_2}(\xi_2))+q_{\theta}(\xi_3),$
 \end{enumerate}
{where necessarily $\gamma_2^*=0$ if $\alpha_1>\alpha_2$ and alternatively $\gamma_1^*=0$ if $\alpha_2>\alpha_1$.}
\end{theorem}
\begin{proof}
To establish the result, it is enough to verify $d1)$, since $d2)$ is an immediate consequence of $d1)$ and axiom \textbf{GLT~1}.
From Theorem~\ref{2D th:} we already know that
\[
 {\{G_{\boldsymbol{\alpha},\textbf{n}}\}_{\mathbf{n}}}\sim_{\mathrm{GLT}}a(x_1,x_2)(\gamma_1^* f_{\alpha_1}(\xi_1)+ \gamma_2^* f_{\alpha_2}(\xi_2)),
\]
{where necessarily $\gamma_2^*=0$ if $\alpha_1>\alpha_2$ and alternatively $\gamma_1^*=0$ if $\alpha_2>\alpha_1$.}
On the other hand, equation~\eqref{Q symb} ensures that $
\{Q_{\theta,M}\}_{M}\sim_{\mathrm{GLT}}q_{\theta}(\xi_3).
$
By combining these two symbols according to the Kronecker-sum form of $A_{\boldsymbol{\alpha},\mathbf{n},M}$ and applying the $\ast$-algebra rules of the GLT class (\textbf{GLT~3}), we arrive at
\[
 {\{A_{\boldsymbol{\alpha},\mathbf{n},M}\}_{\mathbf{n},M}
=\{G_{\boldsymbol{\alpha},\textbf{n}}\otimes I_M+I_{\mathbf{n}}\otimes Q_{\theta,M}\}_{\textbf{n},M}\sim_{\mathrm{GLT}}a(x_1,x_2)(\gamma_1^* f_{\alpha_1}(\xi_1)+ \gamma_2^* f_{\alpha_2}(\xi_2))+q_{\theta}(\xi_3),}
\]
which confirms $d1)$. The property $d2)$ then follows from $d1)$ and axiom \textbf{GLT~1}.
\end{proof}
And what about
\begin{itemize}
     \item[$c3)$] $\{ {A_{\alpha,N,M}}\}_{N,M}\sim_{\lambda} \gamma^* a(x)f_{\alpha}(\xi_1)+q_{\theta}(\xi_3)$?
    \item[$d3)$] $\{ {A_{\boldsymbol{\alpha},\mathbf{n},M}\}_{\mathbf{n},M}}\sim_{\lambda}a(x_1,x_2)(\gamma_1^* f_{\alpha_1}(\xi_1)+ \gamma_2^* f_{\alpha_2}(\xi_2))+q_{\theta}(\xi_3)$?
\end{itemize}

{Since the matrices $A_{\alpha,N,M}, {A_{\boldsymbol{\alpha},\mathbf{n},M}}$ are not Hermitian, their eigenvalue distribution is not guaranteed by the use of the GLT theory. Moreover the present case seems different from the previous ones reported in Figures \ref{fig:1d_const}, \ref{fig:1d_var}, \ref{fig:2d_const}, \ref{fig:2d_var}: indeed the numerical results seem to indicate that the statements $c3)$, $d3)$ are not true in the current setting, even when $a$ is chosen as a constant function.}

\section{Preconditioning strategy: 1D space setting}\label{sec:prec}

Since the considered coefficient matrices are non-symmetric and exhibit a Toeplitz-like structure, Krylov subspace methods such as GMRES~\cite{Saad1986GMRES} are well-suited for solving the linear systems in \eqref{equ12}, benefiting from efficient fast matrix–vector products.
To further enhance the efficiency and robustness of these iterative solvers, the use of preconditioning techniques is an essential strategy \cite{Benzi2002Preconditioning}.
\par
Here, we propose a block circulant preconditioner for {$A_{N,M}$}, obtained by replacing the Toeplitz block $\overline{G}_{\alpha,N}$ with a circulant counterpart $C_{\alpha,N}$, while keeping $Q_{\theta,M}$ unchanged. {Since our constant term $\gamma = \Delta t / (\Delta x)^{\alpha}$ is sufficiently large}, we apply the circulant approximation only to $\overline{G}_{\alpha,N}$ and retain $Q_{\theta,M}$ as it is.
In order to preserve the circulant structure, instead of using the diagonal matrix $D_{N}(a)$, we replace it with a constant scalar multiplier given by the average of the diffusion coefficients
\[
d = \frac{1}{N} \sum_{i=1}^{N} d_i,
\]
where $d_i$ are the diagonal entries of $D_N(a)$. Of course the choice of the circulant counterpart is far from unique and several proposals are treated in three decades of relevant literature. From theoretical studies, we know that Strang-type approaches are better than Frobenius optimal approximations, especially for matching the small eigenvalues and especially when the generating functions are smooth. However, here we have lower Hessemberg structures and hence, together with the usual Strang approximation, we take the Strang-type matrix chosen taking into account the largest coefficients, i.e. those on the first column, except the last which is replaced by that on upper main diagonal.
 {More precisely, we consider these two possible circulant approximations, where the first preconditioner is defined as
\begin{equation}\label{precoS}
     {\widetilde{P}_{\alpha,N,M}}=\gamma  d \, \widetilde{C}_{\alpha,N} \otimes I_M + I_N \otimes Q_{\theta,M}
\end{equation}
where $\widetilde{C}_{\alpha,N}$ is the standard Strang circulant approximation, while the latter is defined as
\begin{equation}\label{precog0}
     {\widehat{P}_{\alpha,N,M}}= \gamma d \, \widehat{C}_{\alpha,N} \otimes I_M + I_N \otimes Q_{\theta,M}
\end{equation}
where $\widehat{C}_{\alpha,N}$ is the circulant matrix generated by the vector $\left[g_1^{(\alpha)}, \ldots, g_{N-1}^{(\alpha)}, g_0^{(\alpha)}\right]$,
that is by the first column of $\overline{G}_{\alpha,N}$ where the last entry $g_N^{(\alpha)}$ has been replaced by the coefficient $g_0^{(\alpha)}$, which is nonnegigible while $g_N^{(\alpha)}$ tends to zero as $N$ tends to infinity; see also the Riemann-Lebesgue lemma \cite{zyg-bible}.}
\par
It is well known that any circulant matrix $C_{N}$ can be diagonalized by the Fourier matrix, i.e.,
{\[
C_{N} = F_N \Lambda_{N} F_N^{*},
\]
}
where $F_N$ is the unitary matrix with entries  {$(F_N)_{s,t} = \tfrac{1}{\sqrt{N}} e^{- \mathrm{i} s t 2\pi/ N}, \; 0 \leq s,t \leq N-1$.}
Here $\mathrm{i} = \sqrt{-1}$ is the imaginary unit, and $\Lambda_{N}$ is a diagonal matrix containing the eigenvalues of $C_{N}$.
\begin{theorem}\label{thm:bd_P1}
 {Let {$P  \in\{\widetilde{P}_{\alpha,N,M},\widehat{P}_{\alpha,N,M}\}$} as defined  in~\eqref{precoS} and in~\eqref{precog0}, respectively}. Set $U = F_{N} \otimes I_{M}$.
Then  {$P$} is unitarily  similar to a block diagonal matrix, namely
\begin{equation}\label{permut}
U^{*} P U = \operatorname{blkdiag}_{j=1}^{N} \left(\gamma d\,\lambda_{j} I_{M} + Q_{\theta,M} \right),
\end{equation}
where $\lambda_{j}$, $j=1,\ldots,N$, denote the eigenvalues of the circulant block  {$\tilde{C}_{\alpha,N}$ or $\hat{C}_{\alpha,N}$, respectively}. Moreover, each diagonal block
\begin{equation}\label{spectra}
D_{j} = \gamma d\,\lambda_{j} I_{M} + Q_{\theta,M}, \ \ j=1,\ldots,N,
\end{equation}
is an invertible lower triangular Toeplitz matrix (ILTT).
\end{theorem}
\begin{proof}
Substituting the diagonalization of  {the circulant matrix }into~\eqref{precoS}  {(or into~\eqref{precog0} equivalently)} gives
\[
P = \gamma d\,F_{N}\Lambda_{\alpha,N}F_{N}^{*}\otimes I_{M} + I_{N}\otimes Q_{\theta,M}.
\]
 {Then, using the mixed-product property $(A\otimes B)(C\otimes D) = AC \otimes BD$, it holds that}
\begin{align*}
U^{*} P U &= (F_{N}^{*} \otimes I_{M})
\left( \gamma d\,F_{N}\Lambda_{\alpha,N}F_{N}^{*} \otimes I_{M} + I_{N} \otimes Q_{\theta,M} \right)
(F_{N} \otimes I_{M}) \\
&=  \gamma d\,F_{N}^{*}F_{N}\Lambda_{\alpha,N}F_{N}^{*}F_{N} \otimes I_{M} + F_{N}^{*}F_{N} \otimes Q_{\theta,M}  \\
&=  \gamma d\,\Lambda_{\alpha,N} \otimes I_{M} + I_{N} \otimes Q_{\theta,M}.
\end{align*}
%
%
Since $\Lambda_{\alpha,N}=\operatorname{diag}(\lambda_{1},\ldots,\lambda_{N})$, it follows that
\[
\Lambda_{\alpha,N}\otimes I_{M} = \operatorname{blkdiag}(\lambda_{1}I_{M},\ldots,\lambda_{N}I_{M}),
\qquad
I_{N}\otimes Q_{\theta,M} = \operatorname{blkdiag}(Q_{\theta,M},\ldots,Q_{\theta,M}).
\]
Hence their sum is block diagonal with blocks $\gamma  d\,\lambda_{j}I_{M}+Q_{\theta,M}$, $j=1,\ldots,N$. Finally, note that $Q_{\theta,M}$ is a lower triangular Toeplitz matrix, and adding a scalar multiple of the identity preserves both the Toeplitz and lower triangular structure. Thus, taking into account that $H_{\theta,M}$ is invertible, each block
\[
D_{j} = \gamma d\,\lambda_{j} I_{M} + Q_{\theta,M} = H_{\theta,M}^{-1}\left(\gamma d\,\lambda_{j}  H_{\theta,M} + H_{M}\right)
\]
is lower triangular Toeplitz. Now we can conclude that $D_{j}$ is nonsingular because $\gamma d\,\lambda_{j}  H_{\theta,M} + H_{M}$ is lower bidiagonal and its invertibility is decided by the diagonal elements, which are of the form $1-\gamma d\,\lambda_{j}\theta$ which can not vanish since the eigenvalues $\lambda_j$ have all nonzero imaginary part.
With the latter the proof is concluded.
\end{proof}
\begin{remark}\label{rem:efficiency}
The block diagonalization in Theorem~\ref{thm:bd_P1} shows that applying  {$P^{-1}$} to a given vector $v$ can be performed through the following three steps.
\begin{enumerate}
    \item Step 1: compute $
    \dot{v} = (F_{N}^{*} \otimes I_{M}) v$.

    \item Step 2: compute
    $\ddot{v} = \operatorname{blkdiag}\big(D_{j}^{-1}\big)_{j=1}^{N}\,\dot{v},
    \qquad D_{j} = \gamma d\,\lambda_{j} I_{M} + Q_{\theta,M}$.

    \item Step 3: compute $
    \dddot{v} = (F_{N} \otimes I_{M}) \ddot{v}$.
\end{enumerate}
The resulting total cost is of the order $O(NM\log(N))$ which is optimal in our setting, since it is the same order of the cost of a matrix-vector product with the original coefficient matrix.

From Theorem~\ref{thm:bd_P1}, both Steps~1 and~3 involve Fourier transforms in the spatial dimension and can be performed in $\mathcal{O}(NM \log N)$ operations using FFT.
Step~2 requires solving $N$ independent systems with coefficient matrices $D_{j}$, which are ILTT matrices.
Since the inverse of an ILTT matrix is also ILTT, its first column can be computed in $\mathcal{O}(M \log M)$ operations by the divide-and-conquer inversion algorithm~\cite{Commenges1984}.
Therefore, Step~2 can be carried out in $\mathcal{O}(NM \log M)$ operations.
In total, applying  {$P^{-1} v$} requires $
\mathcal{O}(NM \log N + NM \log M)$,
which is nearly linear in the problem size.
Moreover, the $N$ systems in Step~2 are completely independent and can be solved in parallel, making the preconditioner highly efficient in practice.

Of course, if the matrix $Q_{\theta,M}$ is not stored explicitly as in (\ref{Q mat}), then it can be viewed as  $H_{\theta,M}^{-1}H_{M}$ so that
\[
D_{j} = H_{\theta,M}^{-1}\left(\gamma d\,\lambda_{j}  H_{\theta,M} + H_{M}\right),
\]
which leads to a more efficient computation since  $H_{\theta,M}$ and $\gamma d\,\lambda_{j}  H_{\theta,M} + H_{M}$, $j=1,\ldots,N$, are all lower bidiagonal.
With this choice, for solving $N$ linear systems with coefficient matrices $D_j$, $j=1,\ldots,N$, the related cost is reduced from $\mathcal{O}(NM \log M)$ to $\mathcal{O}(NM)$.
\end{remark}

We conclude the current section with a GLT/clustering analysis of the preconditioned matrix sequences.
\begin{theorem} \label{thm:inv(P)A clustering 1D}
{Let {$P_{N,M}  \in\{\widetilde{P}_{\alpha,N,M},\widehat{P}_{\alpha,N,M}\}$}, as defined  in~\eqref{precoS} and in~\eqref{precog0}, respectively.
Then,
\[
{\{P_{N,M}\}_{N,M}}\sim_{\mathrm{GLT}}\gamma^* \hat d f_{\alpha}(\xi_1)+q_{\theta}(\xi_3),
\]
with {$\hat d=\int_\Omega a(x)$} and $d=d_N$ converging to $\hat d$ as $N$ tends to infinity and
\[
{\{P_{N,M}^{-1} A_{\alpha,N,M}\}_{N,M}}\sim_{\mathrm{GLT}} \frac{\gamma^* a(x)f_{\alpha}(\xi_1)+q_{\theta}(\xi_3)}{\gamma^*\hat d f_{\alpha}(\xi_1)+q_{\theta}(\xi_3)}.
\]
}
\end{theorem}
\begin{proof}
From the literature \cite{rev-CN,superopt-ESC}, it is known that $\{\overline{G}_{\alpha,N}-\widetilde{C}_{\alpha,N}\}_N\sim_\sigma 0$. The same clustering at zero is true for $\{\overline{G}_{\alpha,N}-\widehat{C}_{\alpha,N}\}_N$  since the approximation is even better (the only neglected coefficient is $g_N^{(\alpha)}$). In both cases the clustering is strong since the generating function is continuous and $2\pi$-periodic \cite{rev-CN,superopt-ESC}.
Now by axiom \textbf{GLT 2.}, part 3, $\{\overline{G}_{\alpha,N}-{C}_{\alpha,N}\}_N\sim_\sigma 0$ if and only if $\{\overline{G}_{\alpha,N}-{C}_{\alpha,N}\}_N\sim_{\mathrm{GLT}} 0$, {${C}_{\alpha,N} \in \{\widetilde{C}_{\alpha,N}, \widehat{C}_{\alpha,N}\}$}. Since $\{\overline{G}_{\alpha,N}\}_N$ is also a GLT matrix sequence, by axiom \textbf{GLT 3.}, part 3, it is clear that both the approximating matrix sequences $\{\widetilde{C}_{\alpha,N}\}_N$,
$\{\hat{C}_{\alpha,N}\}_N$ are both GLT matrix sequences with the same GLT symbol as $\{\overline{G}_{\alpha,N}\}_N$, as in Theorem \ref{1D th:} and Remark \ref{rem:extending1}.

By following the very same steps as in the proof of Theorem \ref{1D th:}, Theorem \ref{1D orig mat} and of the reasoning in Remark \ref{rem:extending1}, we end up with the conclusion
\[
{\{P_{N,M}\}_{N,M}}\sim_{\mathrm{GLT}}\gamma^* \hat d f_{\alpha}(\xi_1)+q_{\theta}(\xi_3)
\]
with {$\hat d=\int_\Omega a(x)$} and $d=d_N$ converging to $\hat d$ as $N$ tends to infinity. As a consequence
\[
{\{P_{N,M}^{-1} A_{\alpha,N,M}\}_{N,M}}\sim_{\mathrm{GLT}} \frac{\gamma^* a(x)f_{\alpha}(\xi_1)+q_{\theta}(\xi_3)}{\gamma^*\hat d f_{\alpha}(\xi_1)+q_{\theta}(\xi_3)},
\]
by using axiom \textbf{GLT 3.}, part 4 and part 3.
\par
 Notice that this GLT symbol is equal to $1$ in the constant coefficient case since $a=\hat d$, so ensuring a clustering at $1$ for the preconditioned matrix sequence.
\end{proof}

\section{Preconditioning strategy: 2D space setting}\label{sec:prec2D}

In the two-dimensional case, the spatial discretization yields a two-level Toeplitz structure, while the temporal discretization add a further level so that a three-level matrix is obtained, in the all-at-once formulation. To construct an efficient preconditioner for this system, we approximate only the spatial part using a multilevel circulant matrix, while the temporal block $Q_{\theta,M}$ remains unchanged. The spatial operator {$G_{\boldsymbol{\alpha},\textbf{n}}$} is the diagonal sampling matrix {$D_\textbf{n}(a)$} multiplied with the two-level Toeplitz matrix
\[
{U_{\boldsymbol{\alpha},\textbf{n}}} {=    \gamma_1 \overline{G}_{\alpha_1,N_1}\otimes I_{N_2} + I_{N_1} \otimes \gamma_2 \overline{G}_{\alpha_2,N_2}}.
\]
To form the preconditioner, we replace {$U_{\boldsymbol{\alpha},\textbf{n}}$} with its two-level circulant approximation
\[
{C_{\boldsymbol{\alpha},\textbf{n}} = \gamma_1 C_{\alpha_1,N_1} \otimes I_{N_2} + I_{N_1} \otimes \gamma_2 C_{\alpha_2,N_2},
}\]
where {$C_{\alpha_j,N_j}$,  $j=1,2$} are the circulant matrices {already} considered in the one-dimensional case, {i.e.,
$C_{\alpha_j,N_j} \in \{\widetilde{C}_{\alpha_j,N_j},\widehat{C}_{\alpha_j,N_j}\}$.
}
\par
{Since the constants
$\gamma_1 = \Delta t/(\Delta x_1)^{\alpha_1}$ and
$\gamma_2 = \Delta t/(\Delta x_2)^{\alpha_2}$
are sufficiently large,} the circulant approximation is applied only to the spatial blocks, while the time stepping matrix $Q_{\theta,M}$ is left unchanged.
{As before, to} preserve the circulant structure, the diagonal sampling matrix {$D_\textbf{n}(a)$} is replaced by the spatial average of the diffusion coefficient
\[
d = \frac{1}{N_1N_2} \sum_{{i_1}=1}^{N_1}\sum_{{i_2}=1}^{N_2} a{_{i_1,i_2}},
\]
and the preconditioner is defined as
\begin{equation}\label{eq:P2a}
{\widetilde{P}_{\boldsymbol{\alpha},\textbf{n},M} = d \, \widetilde{C}_{\boldsymbol{\alpha},\textbf{n}} \otimes I_M + I_{\textbf{n}} \otimes Q_{\theta,M},
}\end{equation}
or
\begin{equation}\label{eq:P2b}
{\widehat{P}_{\boldsymbol{\alpha},\textbf{n},M} = d \, \widehat{C}_{\boldsymbol{\alpha},\textbf{n}} \otimes I_M + I_{\textbf{n}} \otimes Q_{\theta,M},
}\end{equation}
{where, clearly, it holds
\[
C_{\boldsymbol{\alpha},\textbf{n}} = F_{\boldsymbol{\textbf{n}}}^* \  \Lambda_{\boldsymbol{\alpha},\textbf{n}} \ F_{\boldsymbol{\textbf{n}}}
\]
with $C_{\boldsymbol{\alpha},\textbf{n}} \in \{\widetilde{C}_{\boldsymbol{\alpha},\textbf{n}},\widehat{C}_{\boldsymbol{\alpha},\textbf{n}}\}$, $F_{\boldsymbol{\textbf{n}}} = F_{N_1} \otimes F_{N_2}$ and $\Lambda_{\boldsymbol{\alpha},\textbf{n}} = \gamma_1 \Lambda_{\alpha_1,N_1} \otimes I_{N_2} + I_{N_1} \otimes \gamma_2 \Lambda_{\alpha_2,N_2}$.
}%

Now we give only the statement of results regarding the invertibility of $P$, the computation cost related to the chosen preconditioner, and the  clustering results for the related preconditioned matrix sequence. The tools and the structure of the proofs follow verbatim those of Section \ref{sec:prec}.

\begin{theorem}
Let {$P \in \{\widetilde{P}_{\boldsymbol{\alpha},\textbf{n},M},\widehat{P}_{\boldsymbol{\alpha},\textbf{n},M}\}$} be defined as in~\eqref{eq:P2a} and as in~\eqref{eq:P2b}, respectively. Set $U = F_{\textbf{n}}\otimes I_{M}$.
Then $P$ is unitarily similar to a block diagonal matrix
\[
U^{*} P U
= \operatorname{blkdiag}_{j=1}^{N_{1}N_{2}}
\left(d\,\lambda_{j} I_{M} + Q_{\theta,M}\right),
\]
where $\lambda_{j}$ are the eigenvalues of
$C_{\boldsymbol{\alpha},\textbf{n}}$, $j=1,\ldots,N_1N_2$.
Each diagonal block
\[
D_{j} = d\,\lambda_{j} I_{M} + Q_{\theta,M}, \ \ \ j=1,\ldots,N_1N_2,
\]
is an invertible lower triangular Toeplitz matrix.
\end{theorem}
\begin{remark}
The computation of a vector $P^{-1}v$ is carried out in three steps.
\begin{enumerate}
\item  Step 1: compute $\dot{v} = (F_{\boldsymbol{\textbf{n}}}^* \otimes I_M) v$.
This requires $O(N_1N_2M \log (N_1N_2))$ operations.
\item Step 2: compute $\ddot{v} =
\operatorname{blkdiag}\big(D_j^{-1}\big)_{j=1}^{N_1N_2} \dot{v},$ \\
where each system involves an ILTT matrix.
The first column of each inverse can be computed in $O(M \log M)$ operations.
The total cost is $O(N_1N_2M \log M)$.
All systems are independent and can be solved in parallel.
\item Step 3: compute
$v = (F_{\boldsymbol{\textbf{n}}} \otimes I_M)\ddot{v}$,
with cost $O(N_1N_2M \log (N_1N_2))$.
\end{enumerate}
Therefore, computing $P^{-1} v$ requires $O\!\left(N_1N_2M\log (N_1N_2)
+ N_1N_2M\log M\right),$
which is nearly linear in the problem size.
\par
Furthermore, in line with the arguments in Remark \ref{rem:efficiency}, for solving $N_1N_2$ linear systems with coefficient matrices $D_j$, $j=1,\ldots,N_1N_2$, the related cost is reduced from $\mathcal{O}(N_1N_2M \log M)$ to $\mathcal{O}(N_1N_2M)$.
\end{remark}
\begin{theorem} \label{thm:inv(P)A clustering 2D}
{Let {$P  \in\{\widetilde{P}_{\boldsymbol{\alpha},\textbf{n},M},
\widehat{P}_{\boldsymbol{\alpha},\textbf{n},M}\}$}, as defined  in~\eqref{eq:P2a} and in~\eqref{eq:P2b}, respectively.
Then,
\[
{\{P_{\textbf{n},M}\}_{\textbf{n},M}}\sim_{\mathrm{GLT}}
\hat d (\gamma_1^*f_{\alpha_1}(\xi_1)+\gamma_2^*f_{\alpha_2}(\xi_2))
+q_{\theta}(\xi_3),
\]
with {$\hat d=\int_\Omega a(x_1,x_2)$} and {$d=d_\textbf{n}$} converging to $\hat d$ as $\textbf{n}$ tends to infinity, and
\[
{\{P_{\textbf{n},M}^{-1} A_{\textbf{n},M}\}_{\textbf{n},M}}\sim_{\mathrm{GLT}} \frac{\gamma^* a(x_1,x_2)(\gamma_1^*f_{\alpha_1}(\xi_1)+\gamma_2^*f_{\alpha_2}(\xi_2))+q_{\theta}(\xi_3)}{\gamma^*\hat \hat{d} (\gamma_1^*f_{\alpha_1}(\xi_1)+\gamma_2^*f_{\alpha_2}(\xi_2))+q_{\theta}(\xi_3)},
\]
where necessarily $\gamma_2^* = 0$ if $\alpha_1 > \alpha_2$ and alternatively $\gamma_1^* = 0$ if $\alpha_2 > \alpha_1$.
}
\end{theorem}

\section{Numerical experiments}\label{sec:num}
We conclude our analysis by numerically evaluating the performance of the proposed preconditioning strategy.
The objective is to examine its capability to cluster the eigenvalues around~1, reduce the condition number,
and enhance the efficiency of the GMRES method in terms of iteration counts and CPU time.

All numerical experiments were conducted in \textsc{MATLAB}~R2022b on a laptop equipped with an 11th~Gen
Intel(R)~Core(TM)~i5-1155G7 CPU @~2.50\,GHz, 16\,GB RAM, running Windows~11~Pro
(version~23H2, build~22631.5189).
\subsection{Eigenvalue distribution}
We first assess the performance of the preconditioners {$\widetilde{P}_{\alpha,N,M}$ and $\widehat{P}_{\alpha,N,M}$} by examining its influence on the eigenvalue distribution and condition number of the coefficient matrix.
Figure~\ref{Fig:Aes1} presents the eigenvalue plots in the complex plane for different values of the fractional order $\alpha$, with the problem size fixed at {$N=M = 2^6$}.
The eigenvalues of the unpreconditioned matrix are widely scattered along the positive real axis, indicating poor spectral clustering and, consequently, a high level of ill-conditioning. {The behaviour is worsening as the fractional diffusion parameter increases as expected.}
%

As shown in Table~\ref{tab:cond-numbers}, the unpreconditioned matrix exhibits rapidly growing 2-norm condition numbers, even for moderate values of $N$ and $M$. Though the condition numbers of the preconditioned matrix remain bounded and stable, it's not so lower.
Indeed, a more refined spectral analysis highlights that the condition number alone does not provide an exhaustive description of the eigenvalues distribution properties. \par
In Figures~\ref{Fig:PcircStranges1}-\ref{Fig:Pcircg0es1} are reported the eigenvalues of preconditioned matrices {$\widetilde{P}_{\alpha,N,M}^{-1}A_{\alpha,N,M}$ and $\widehat{P}_{\alpha,N,M}^{-1}A_{\alpha,N,M}$}
for different values of the fractional order $\alpha$, with the problem size fixed at {$N=M = 2^6$.
It evident as we can observe a pronounced clustered distributions around the unity.
The same behaviour is observed in the non constant coefficient case: the related matrix sequences not cluster at $1$ but at the range of
the nice function reported in Theorem \ref{thm:inv(P)A clustering 1D}.}%
\begin{table}
\centering
\renewcommand{\arraystretch}{1}
\setlength{\tabcolsep}{3pt}
\begin{tabular}{|c c |ccc|ccc|ccc|}
\hline
\multicolumn {2}{|c}{} & \multicolumn {3}{|c}{$\alpha = 1.3$} & \multicolumn {3}{|c}{$\alpha = 1.5$} & \multicolumn {3}{|c|}{$\alpha = 1.8$} \\
\hline
$N$ & $M$ & $A$ & $\widetilde{P}^{-1}A$ & $\widehat{P}^{-1}A$ & $A$ & $\widetilde{P}^{-1}A$ & $\widehat{P}^{-1}A$ & $A$ & $\widetilde{P}^{-1}A$ & $\widehat{P}^{-1}A$ \\
\hline
$2^4$ &$2^4$ &2.013e+02 &1.727e+01 &1.953e+01 &1.634e+02 &2.662e+01 &3.238e+01 &1.483e+02 &6.765e+01 &8.167e+01 \\
$2^4$ &$2^5$ &7.475e+02 &1.728e+01 &1.955e+01 &5.604e+02 &2.664e+01 &3.241e+01 &3.957e+02 &6.770e+01 &8.173e+01 \\
$2^4$ &$2^6$ &2.937e+03 &1.729e+01 &1.955e+01 &2.160e+03 &2.664e+01 &3.241e+01 &1.417e+03 &6.771e+01 &8.174e+01 \\
$2^4$ &$2^7$ &1.170e+04 &1.729e+01 &1.955e+01 &8.562e+03 &2.665e+01 &3.241e+01 &5.520e+03 &6.772e+01 &8.175e+01 \\
\hline
$2^5$ &$2^4$ &2.352e+02 &3.878e+01 &4.367e+01 &2.273e+02 &6.930e+01 &8.360e+01 &3.173e+02 &2.190e+02 &2.604e+02 \\
$2^5$ &$2^5$ &7.854e+02 &3.882e+01 &4.372e+01 &6.177e+02 &6.937e+01 &8.367e+01 &5.477e+02 &2.191e+02 &2.606e+02 \\
$2^5$ &$2^6$ &3.007e+03 &3.883e+01 &4.373e+01 &2.226e+03 &6.938e+01 &8.369e+01 &1.548e+03 &2.192e+02 &2.607e+02 \\
$2^5$ &$2^7$ &1.190e+04 &3.883e+01 &4.373e+01 &8.684e+03 &6.939e+01 &8.370e+01 &5.650e+03 &2.192e+02 &2.607e+02 \\
\hline
$2^6$ &$2^4$ &3.221e+02 &9.102e+01 &1.023e+02 &4.187e+02 &1.878e+02 &2.256e+02 &9.089e+02 &7.328e+02 &8.655e+02 \\
$2^6$ &$2^5$ &8.632e+02 &9.113e+01 &1.024e+02 &7.886e+02 &1.880e+02 &2.258e+02 &1.123e+03 &7.334e+02 &8.661e+02 \\
$2^6$ &$2^6$ &3.096e+03 &9.115e+01 &1.024e+02 &2.379e+03 &1.880e+02 &2.258e+02 &2.057e+03 &7.335e+02 &8.663e+02 \\
$2^6$ &$2^7$ &1.206e+04 &9.116e+01 &1.024e+02 &8.859e+03 &1.880e+02 &2.258e+02 &6.088e+03 &7.336e+02 &8.664e+02 \\
\hline
$2^7$ &$2^4$ &5.478e+02 &2.189e+02 &2.457e+02 &9.705e+02 &5.199e+02 &6.231e+02 &2.959e+03 &2.495e+03 &2.936e+03 \\
$2^7$ &$2^5$ &1.064e+03 &2.191e+02 &2.460e+02 &1.317e+03 &5.204e+02 &6.237e+02 &3.166e+03 &2.497e+03 &2.938e+03 \\
$2^7$ &$2^6$ &3.276e+03 &2.192e+02 &2.461e+02 &2.836e+03 &5.205e+02 &6.238e+02 &4.029e+03 &2.497e+03 &2.939e+03 \\
$2^7$ &$2^7$ &1.227e+04 &2.192e+02 &2.461e+02 &9.270e+03 &5.205e+02 &6.239e+02 &7.811e+03 &2.497e+03 &2.939e+03 \\
\hline
\end{tabular}
\caption{Spectral condition numbers of unpreconditioned $A$ and preconditioned matrices $\widetilde{P}^{-1}A =\widetilde{P}^{-1}_{\alpha,N,M}A_{\alpha,N,M}$ and $\widehat{P}^{-1}A = \widehat{P}^{-1}_{\alpha,N,M}A_{\alpha,N,M}$ - case $a(x)=1$.}
\label{tab:cond-numbers}
\end{table}
%
\begin{figure}%
\centering
\includegraphics[width=\textwidth]{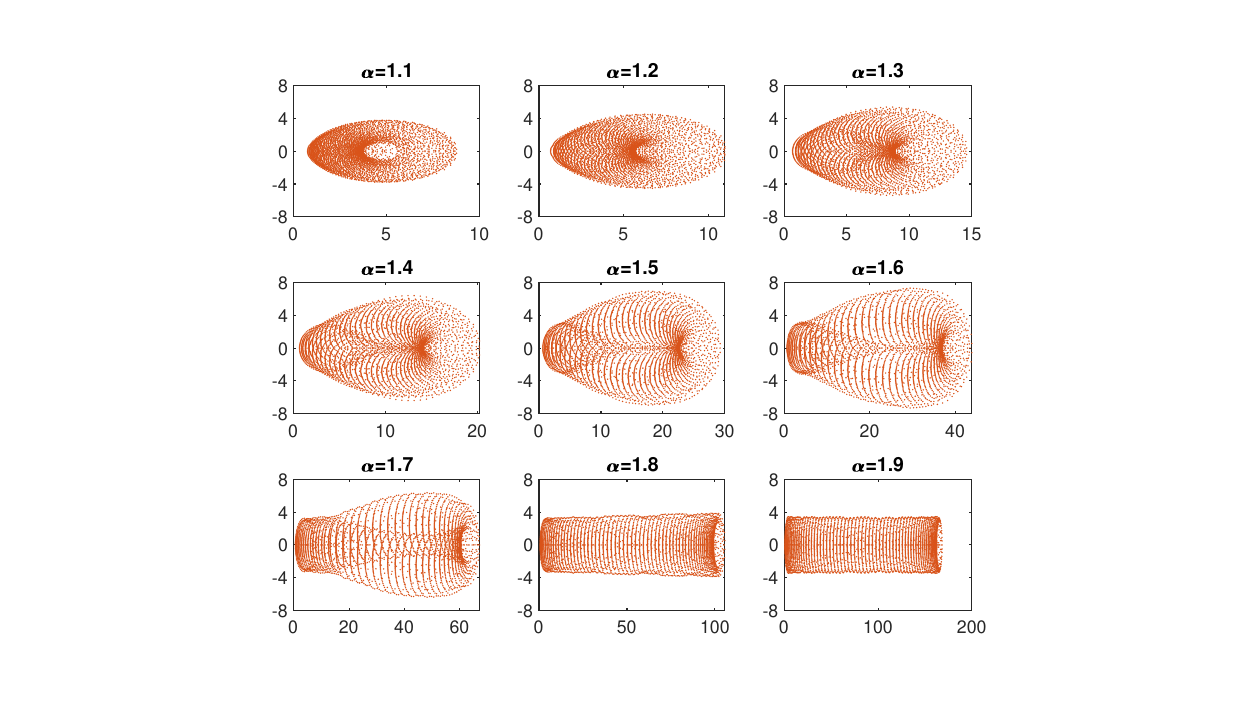}
\caption{Eigenvalues of $A_{\alpha,N,M}$  - $N=M=2^6$, $a(x)=1$.} \label{Fig:Aes1}
\end{figure}
\begin{figure}%
\centering
\includegraphics[width=\textwidth]{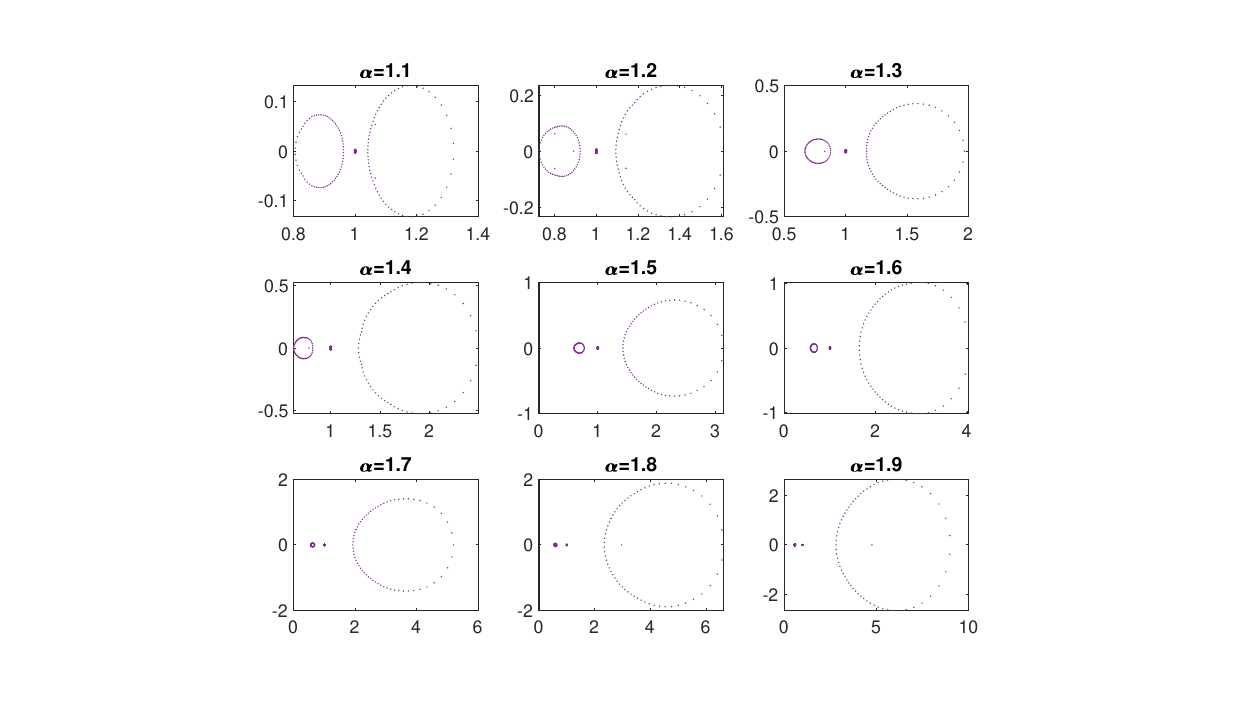}
\includegraphics[width=\textwidth]{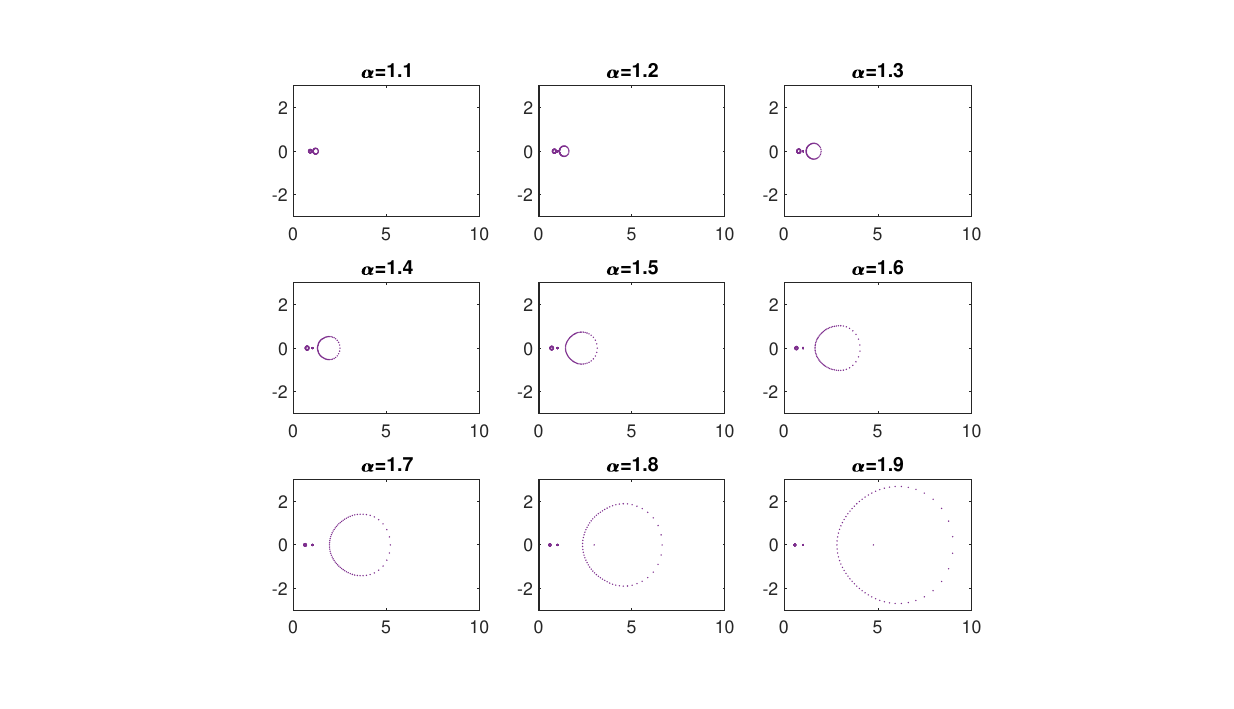}
\caption{Eigenvalues of $\widetilde{P}^{-1}_{\alpha,N,M}A_{\alpha,N,M}$ - $N=M=2^6$, $a(x)=1$.} \label{Fig:PcircStranges1}
\end{figure}
\begin{figure}%
\centering
\includegraphics[width=\textwidth]{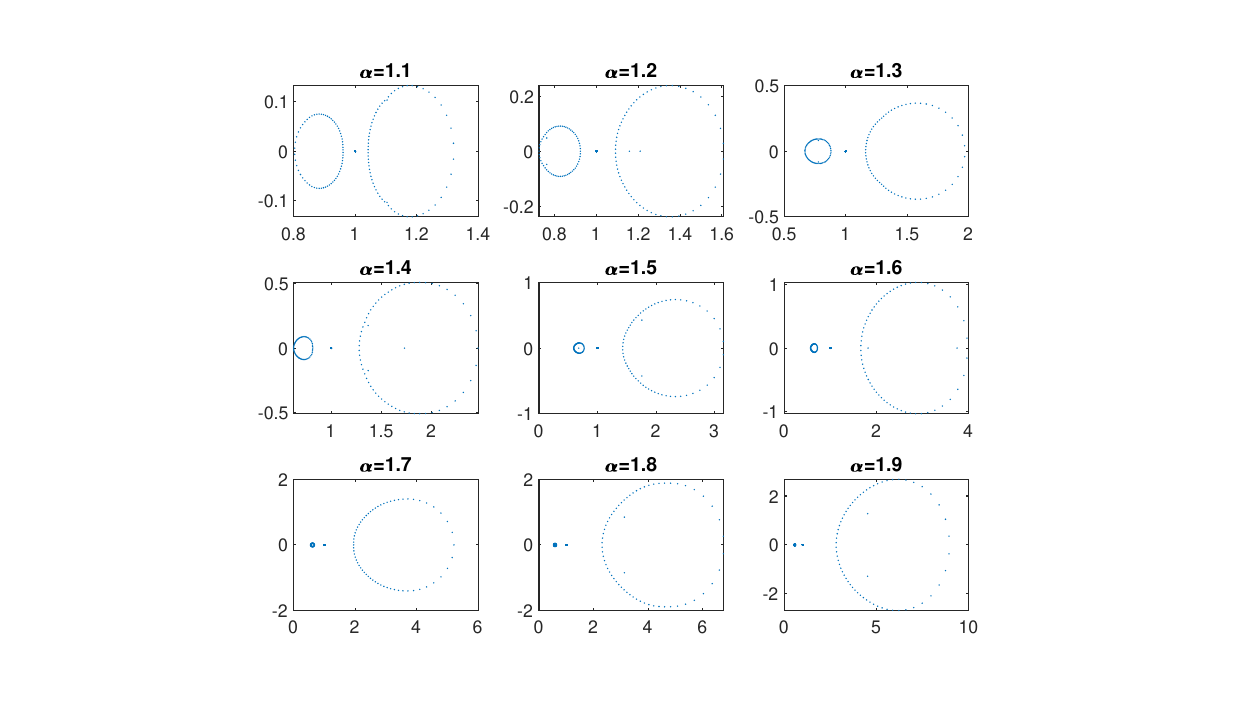}
\includegraphics[width=\textwidth]{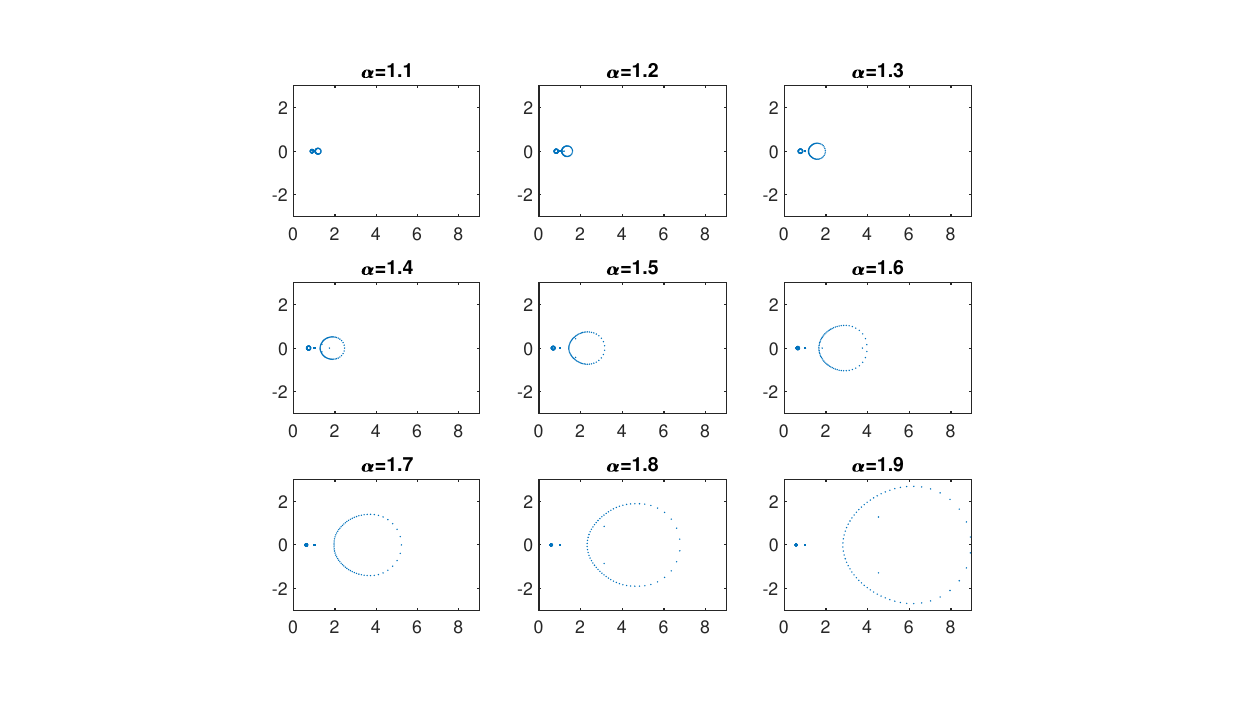}
\caption{Eigenvalues of $\widehat{P}^{-1}_{\alpha,N,M}A_{\alpha,N,M}$  - $N=M=2^6$, $a(x)=1$.} \label{Fig:Pcircg0es1}
\end{figure}
%
%
%
\begin{table}
\centering
\renewcommand{\arraystretch}{1}
\setlength{\tabcolsep}{3pt}
\begin{tabular}{|c c |ccc|ccc|ccc|}
\hline
\multicolumn {2}{|c}{} & \multicolumn {3}{|c}{$\alpha = 1.3$} & \multicolumn {3}{|c}{$\alpha = 1.5$} & \multicolumn {3}{|c|}{$\alpha = 1.8$} \\
\hline
$N$ & $M$ & $A$ & $\widetilde{P}^{-1}A$ & $\widehat{P}^{-1}A$ & $A$ & $\widetilde{P}^{-1}A$ & $\widehat{P}^{-1}A$ & $A$ & $\widetilde{P}^{-1}A$ & $\widehat{P}^{-1}A$ \\
\hline
$2^4$ &$2^4$ &1.678e+02 &1.535e+01 &1.805e+01 &1.465e+02 &2.556e+01 &3.285e+01 &1.549e+02 &8.477e+01 &1.080e+02 \\
$2^4$ &$2^5$ &5.939e+02 &1.536e+01 &1.806e+01 &4.582e+02 &2.558e+01 &3.288e+01 &3.493e+02 &8.484e+01 &1.081e+02 \\
$2^4$ &$2^6$ &2.307e+03 &1.536e+01 &1.807e+01 &1.724e+03 &2.558e+01 &3.288e+01 &1.168e+03 &8.485e+01 &1.081e+02 \\
$2^4$ &$2^7$ &9.162e+03 &1.536e+01 &1.807e+01 &6.796e+03 &2.558e+01 &3.288e+01 &4.472e+03 &8.486e+01 &1.081e+02 \\
\hline
$2^5$ &$2^4$ &2.168e+02 &3.207e+01 &3.756e+01 &2.417e+02 &6.240e+01 &7.951e+01 &4.020e+02 &2.592e+02 &3.245e+02 \\
$2^5$ &$2^5$ &6.392e+02 &3.209e+01 &3.759e+01 &5.411e+02 &6.245e+01 &7.957e+01 &5.805e+02 &2.594e+02 &3.247e+02 \\
$2^5$ &$2^6$ &2.365e+03 &3.210e+01 &3.760e+01 &1.803e+03 &6.246e+01 &7.959e+01 &1.365e+03 &2.594e+02 &3.248e+02 \\
$2^5$ &$2^7$ &9.286e+03 &3.210e+01 &3.760e+01 &6.901e+03 &6.246e+01 &7.959e+01 &4.650e+03 &2.594e+02 &3.248e+02 \\
\hline
$2^6$ &$2^4$ &3.499e+02 &7.158e+01 &8.368e+01 &5.309e+02 &1.614e+02 &2.048e+02 &1.293e+03 &8.164e+02 &1.013e+03 \\
$2^6$ &$2^5$ &7.559e+02 &7.164e+01 &8.375e+01 &8.100e+02 &1.615e+02 &2.049e+02 &1.461e+03 &8.171e+02 &1.014e+03 \\
$2^6$ &$2^6$ &2.475e+03 &7.166e+01 &8.377e+01 &2.034e+03 &1.616e+02 &2.050e+02 &2.181e+03 &8.173e+02 &1.014e+03 \\
$2^6$ &$2^7$ &9.428e+03 &7.166e+01 &8.378e+01 &7.123e+03 &1.616e+02 &2.050e+02 &5.343e+03 &8.173e+02 &1.014e+03 \\
\hline
$2^7$ &$2^4$ &6.949e+02 &1.670e+02 &1.950e+02 &1.372e+03 &4.339e+02 &5.492e+02 &4.452e+03 &2.639e+03 &3.262e+03 \\
$2^7$ &$2^5$ &1.077e+03 &1.671e+02 &1.952e+02 &1.636e+03 &4.343e+02 &5.496e+02 &4.617e+03 &2.641e+03 &3.265e+03 \\
$2^7$ &$2^6$ &2.750e+03 &1.672e+02 &1.952e+02 &2.781e+03 &4.344e+02 &5.498e+02 &5.293e+03 &2.642e+03 &3.265e+03 \\
$2^7$ &$2^7$ &9.692e+03 &1.672e+02 &1.953e+02 &7.759e+03 &4.344e+02 &5.498e+02 &8.201e+03 &2.642e+03 &3.265e+03 \\
\hline
\end{tabular}
\caption{Spectral condition numbers of unpreconditioned $A$ and preconditioned matrices $\widetilde{P}^{-1}A =\widetilde{P}^{-1}_{\alpha,N,M}A_{\alpha,N,M}$ and $\widehat{P}^{-1}A = \widehat{P}^{-1}_{\alpha,N,M}A_{\alpha,N,M}$ - case $a(x)=x^2+1$.}
\label{tab:cond-numbers-es2}
\end{table}
%
\begin{figure}%
\centering
\includegraphics[width=\textwidth]{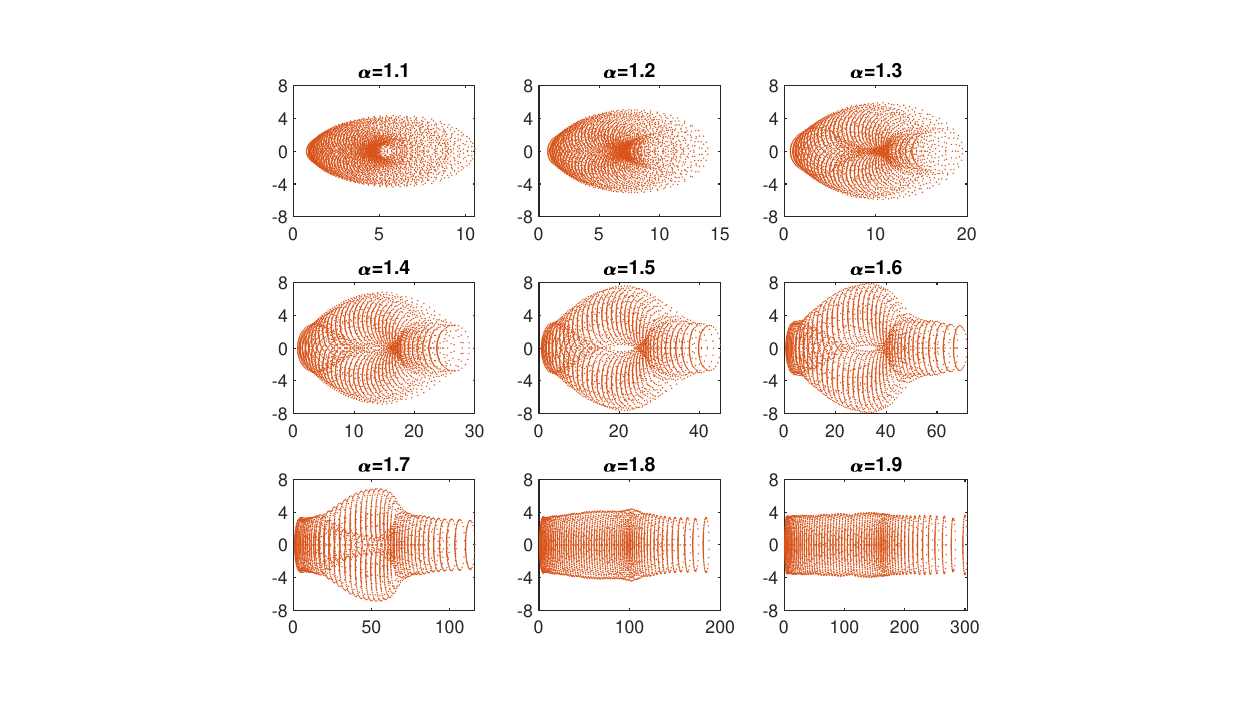}
\caption{Eigenvalues of $A_{\alpha,N,M}$  - $N=M=2^6$, $a(x)=x^2+1$.} \label{Fig:Aes2}
\end{figure}
\begin{figure}%
\centering
\includegraphics[width=\textwidth]{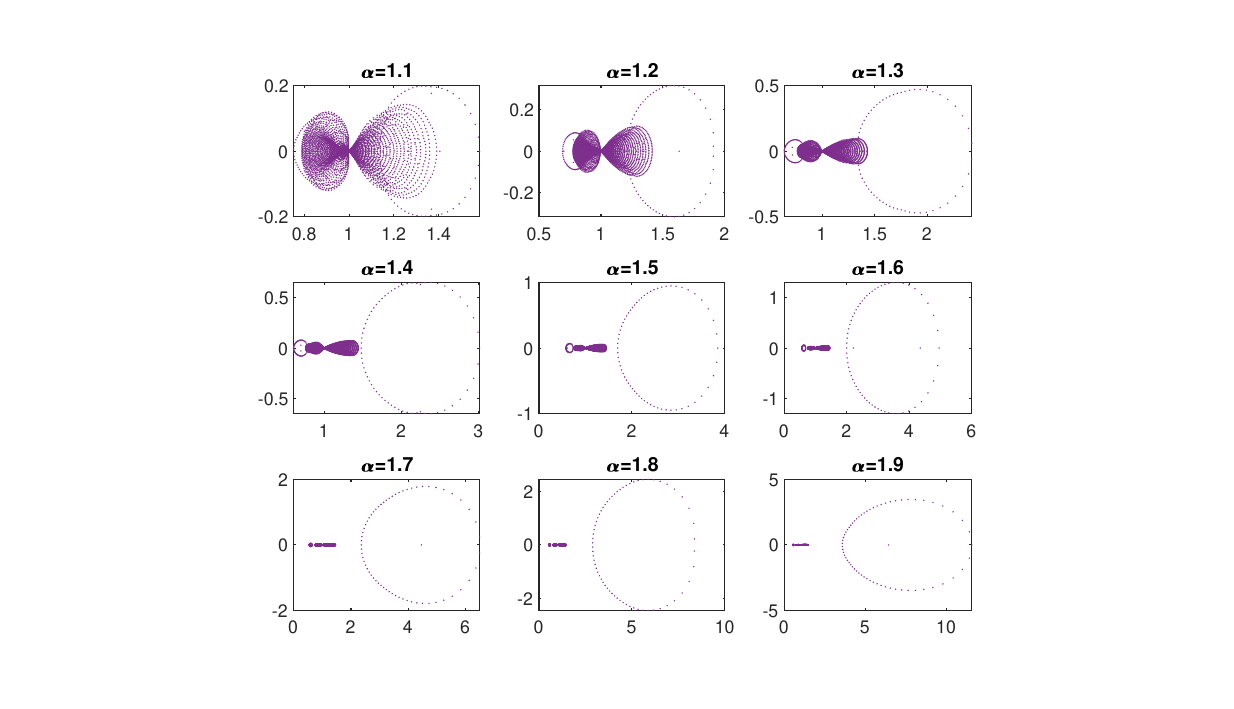}
\includegraphics[width=\textwidth]{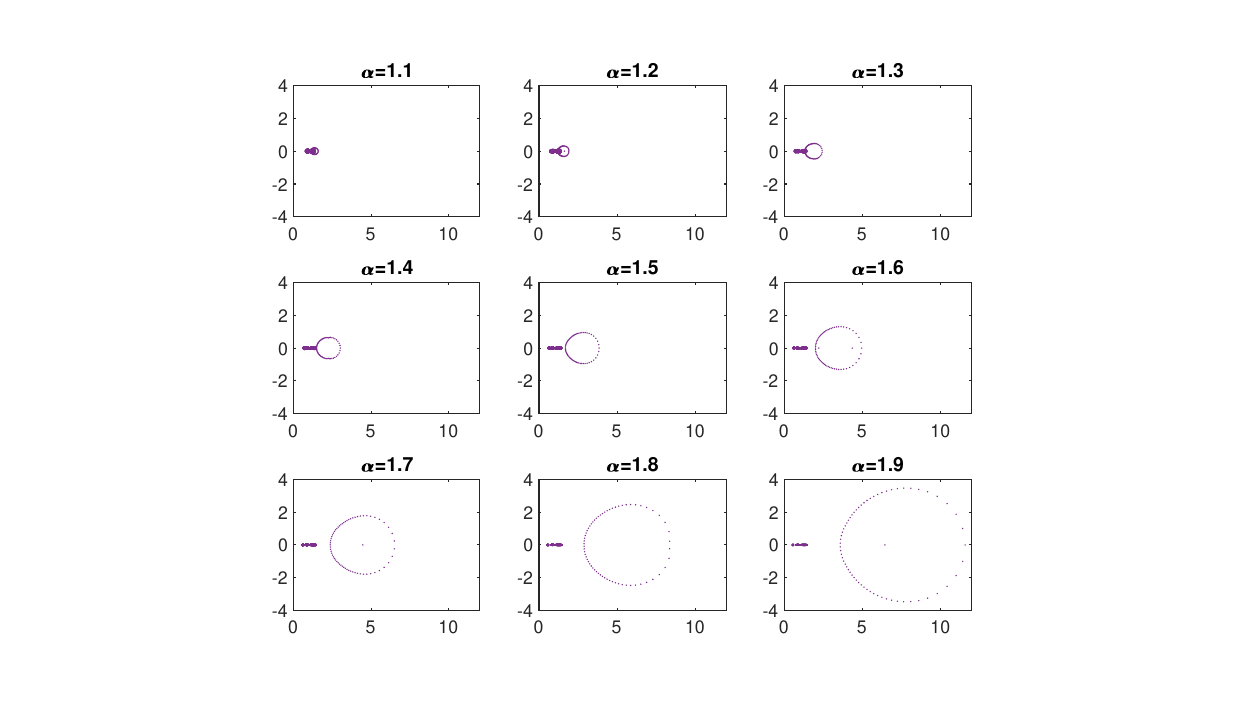}
\caption{Eigenvalues of $\widetilde{P}^{-1}_{\alpha,N,M}A_{\alpha,N,M}$  - $N=M=2^6$, $a(x)=x^2+1$.} \label{Fig:PcircStranges2}
\end{figure}
\begin{figure}%
\centering
\includegraphics[width=\textwidth]{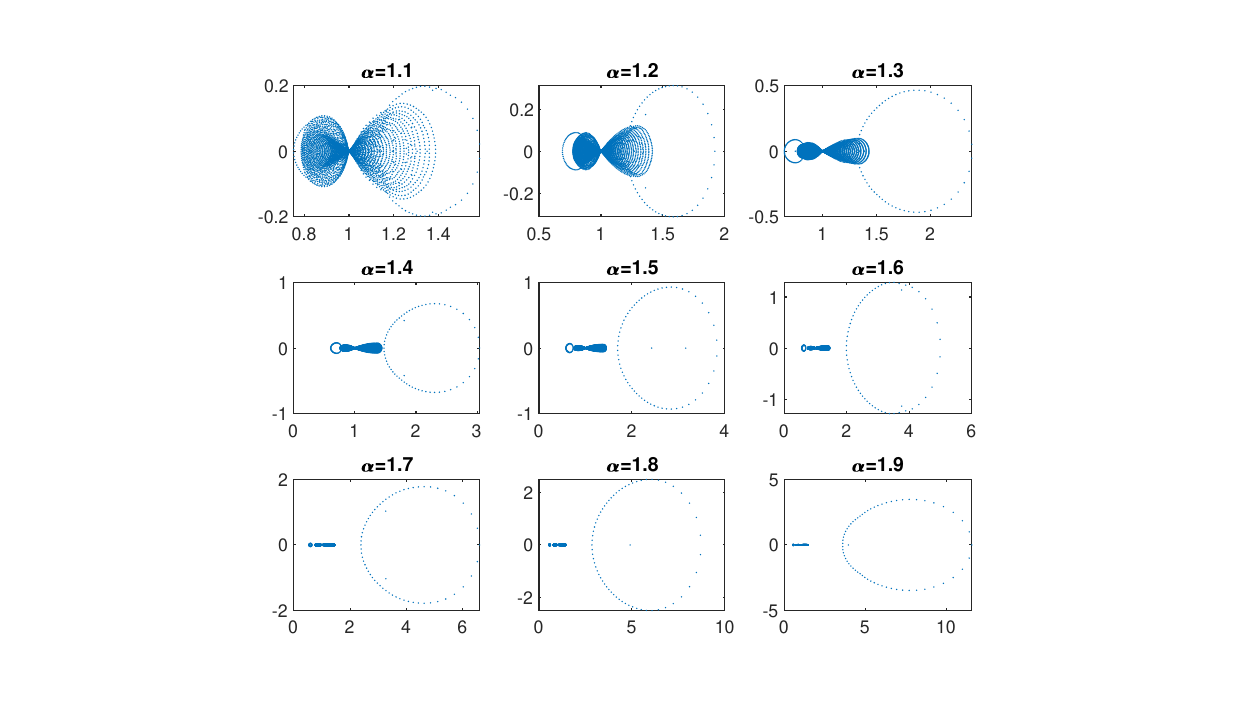}
\includegraphics[width=\textwidth]{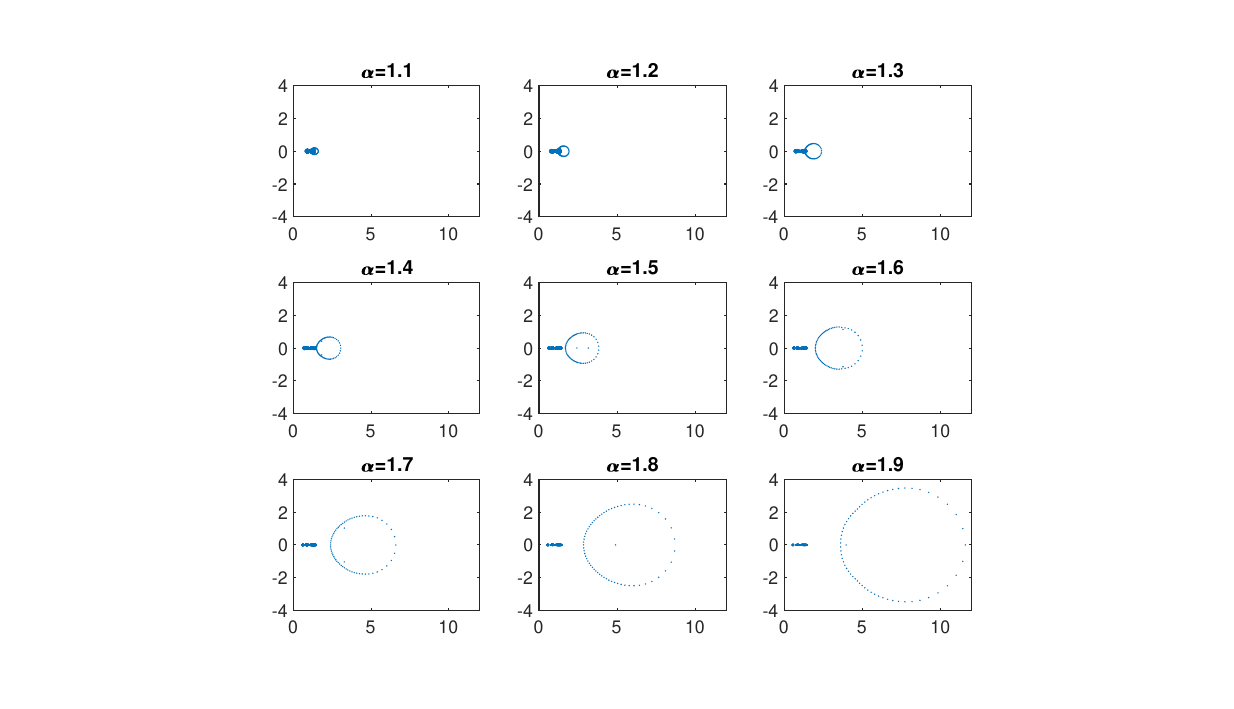}
\caption{Eigenvalues of $\widehat{P}^{-1}_{\alpha,N,M}A_{\alpha,N,M}$ - $N=M=2^6$, $a(x)=x^2+1$.} \label{Fig:Pcircg0es2}
\end{figure}
%
\begin{figure}%
\centering
\includegraphics[width=\textwidth]{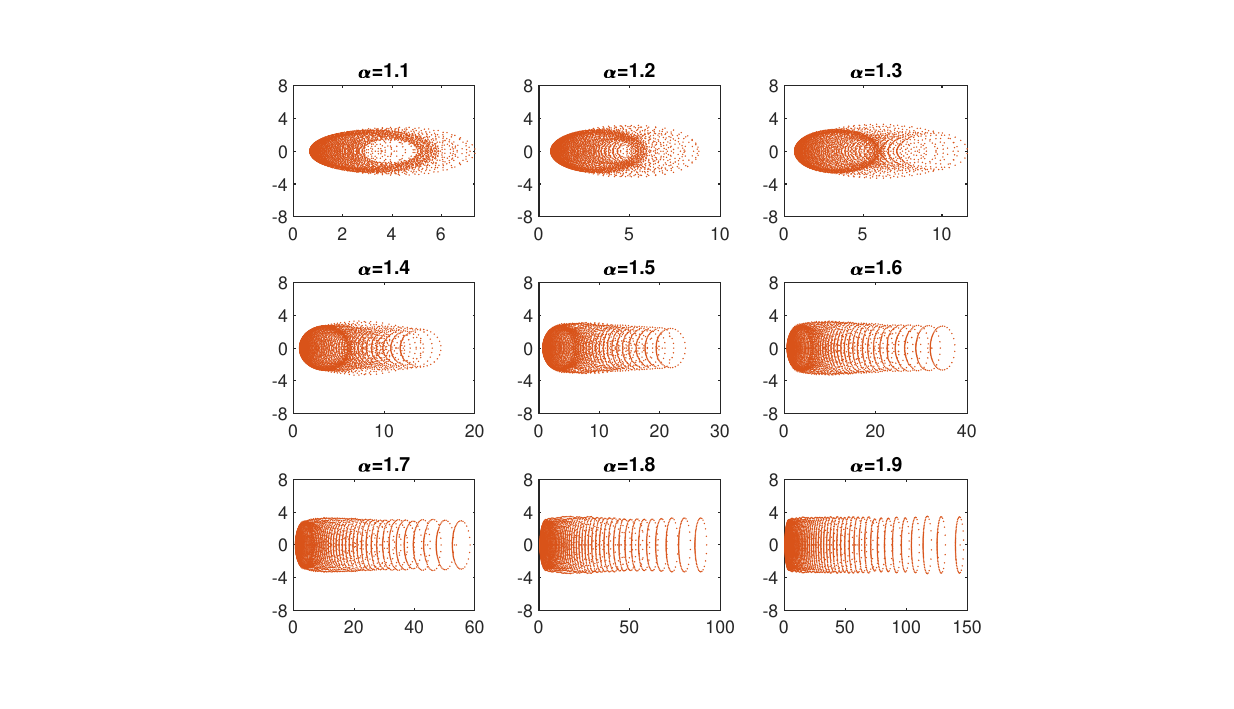}
\caption{Eigenvalues of $A_{\alpha,N,M}$  - $N=M=2^6$, $a(x)=x^2$.} \label{Fig:Aes4}
\end{figure}
\begin{figure}%
\centering
\includegraphics[width=\textwidth]{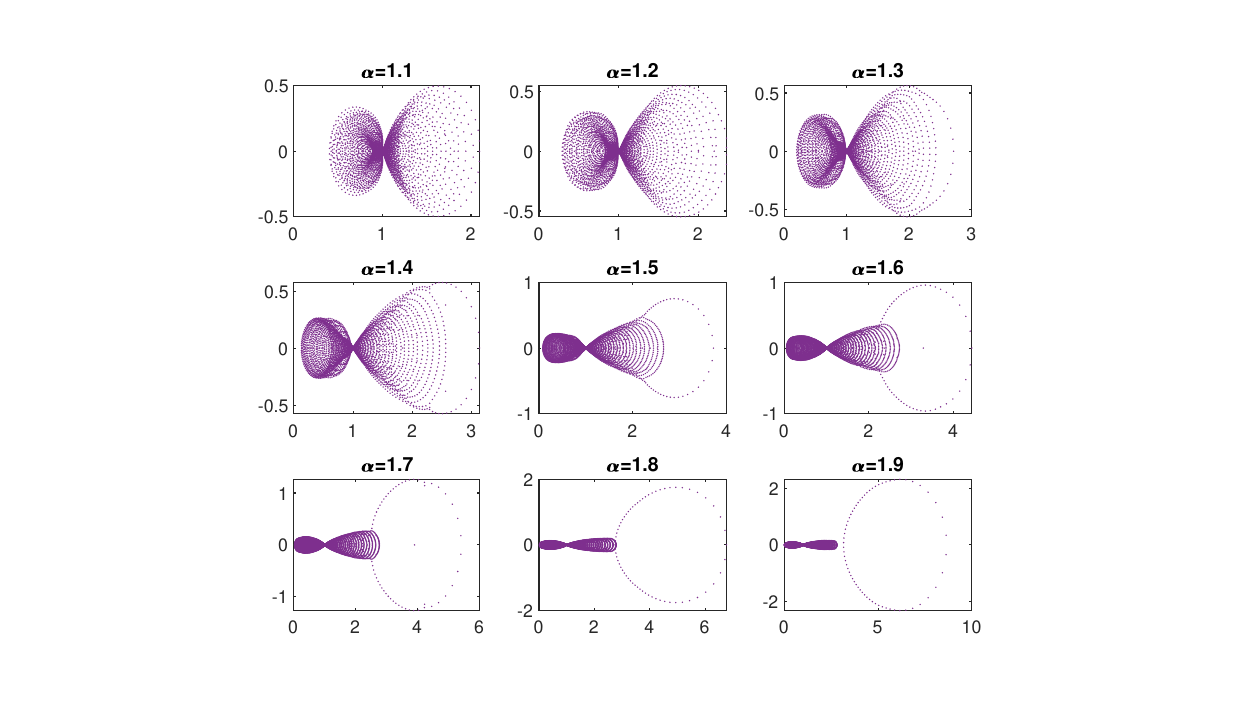}
\includegraphics[width=\textwidth]{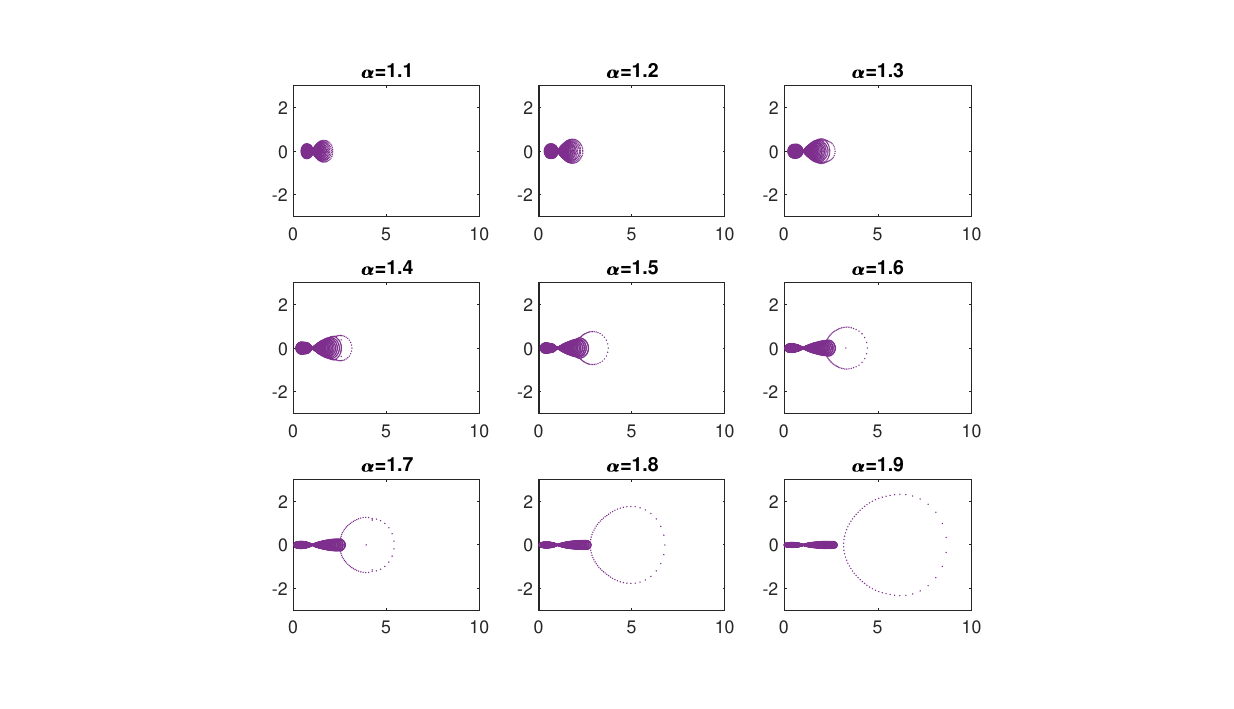}
\caption{Eigenvalues of $\widetilde{P}^{-1}_{\alpha,N,M}A_{\alpha,N,M}$ - $N=M=2^6$, $a(x)=x^2$.} \label{Fig:PcircStranges4}
\end{figure}
\begin{figure}%
\centering
\includegraphics[width=\textwidth]{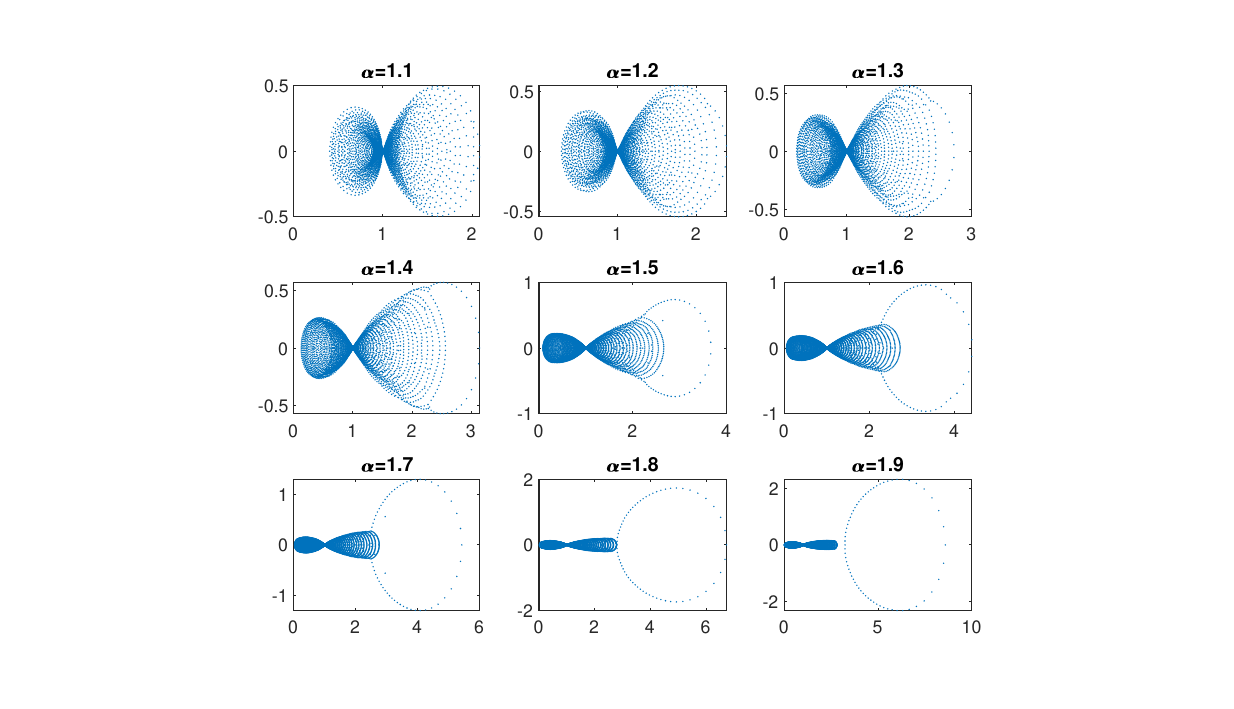}
\includegraphics[width=\textwidth]{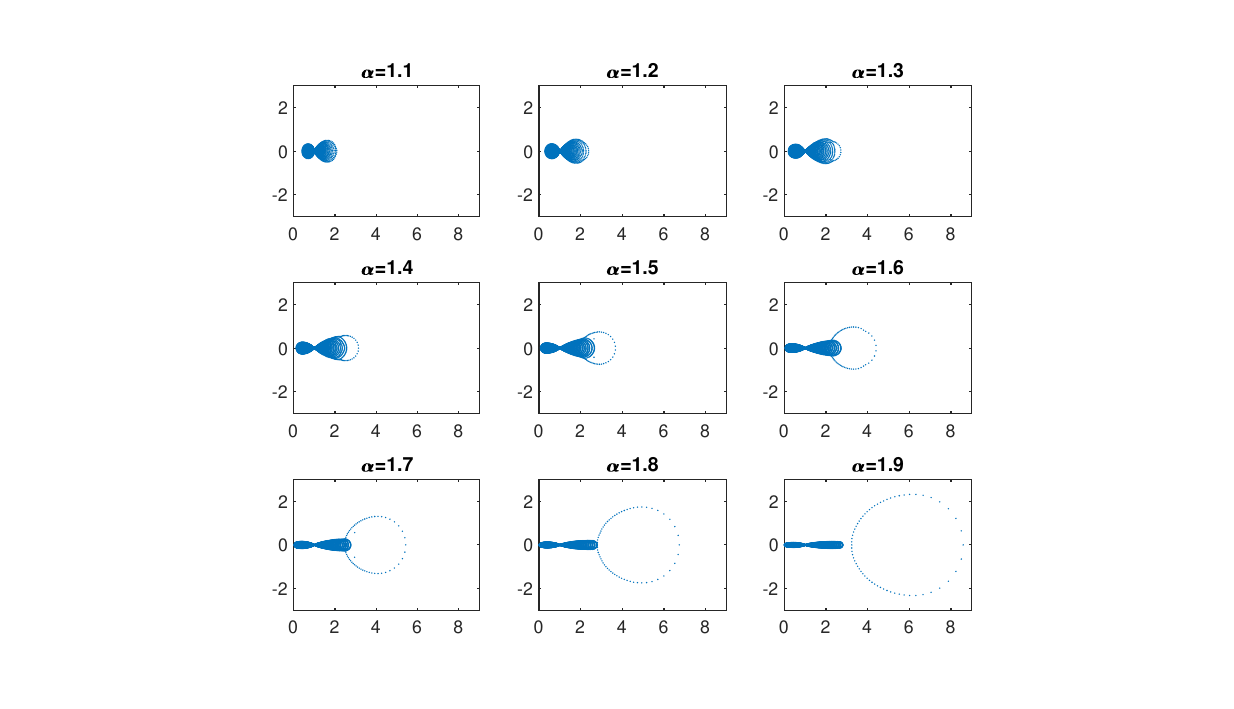}
\caption{Eigenvalues of $\widehat{P}^{-1}_{\alpha,N,M}A_{\alpha,N,M}$  - $N=M=2^6$, $a(x)=x^2$.} \label{Fig:Pcircg0es4}
\end{figure}
\clearpage
{\subsection{GMRES performance in the $d$-dimensional space setting, $d=1,2$}}
We consider the one-dimensional variable-coefficient space-fractional diffusion equation (\ref{eq2-1d}):
\begin{equation*}
\frac{\partial u(x,t)}{\partial t} = x^2\, D_{x}^{\alpha} u(x,t) + f(x,t), \quad x \in [0, 1],\, t \in [0, 1],
\end{equation*}
with homogeneous Dirichlet boundary conditions $u(0, t) = u(1, t) = 0$ and initial condition $u(x, 0) = 0$.

We choose the exact solution as
\begin{equation*}
u(x, t) = t^2 x^2 (1-x)^2.
\end{equation*}
The corresponding source term is given by
\begin{align*}
f(x, t) =\; & 2t\, x^2 (1-x)^2  - x^2\, t^2 \left[
\frac{2}{\Gamma(3-\alpha)} x^{2-\alpha}
- \frac{12}{\Gamma(4-\alpha)} x^{3-\alpha}
+ \frac{24}{\Gamma(5-\alpha)} x^{4-\alpha}
\right].
\end{align*}
Thus, in this section, we investigate the practical performance of the GMRES iterative solver for the all-at-once linear system arising from the one-dimensional discretization given in equation~\eqref{equ12}, with the time-stepping parameter set to $\theta = 0.5$ (Crank–Nicolson method). We compare the behavior of GMRES applied to both the original (non-preconditioned) system and the system preconditioned by a circulant structure-preserving preconditioners as defined in Section \ref{sec:prec}, to illustrate the impact of preconditioning on convergence and computational efficiency. First, the experiments are conducted for the one-dimensional constant-coefficient fractional diffusion problem ($a(x)=1$), and then for the one-dimensional variable-coefficient fractional diffusion problem
with coefficient $a(x) = x^2+1$, and $a(x) = x^2$.
The initial vector is set to zero, the tolerance is $10^{-8}$, and the maximum number of iterations equals the matrix dimension. Tables~\ref{tab:gmres-iter-es1}, \ref{tab:gmres-iter-es2}, and \ref{tab:gmres-iter-es4} summarize the numerical performance of the GMRES method both for the unpreconditioned system and the preconditioned ones.
Here, $N$ denotes the number of spatial grid points and $M$ the number of time levels and the average number of GMRES iterations per time step are reported in the columns.
Tables~\ref{tab:Timings_es1} and \ref{tab:Timings_es2}
report the averaged total computational time (in seconds) for solving the entire problem 10 times.
\par
As the discretization becomes finer, the unpreconditioned system  exhibits a substantial growth in both iteration counts and CPU times, reflecting the increasing ill-conditioning of the coefficient matrix.
The computational effort is mainly influenced by the spatial parameter $N$, as larger $N$ values lead to denser matrices and slower convergence, while variations in $M$ have a comparatively smaller effect.
\par
When the preconditioners are applied, GMRES shows a significant improvement in convergence.
The iteration numbers remain nearly constant for increasing $N$ and $M$, and the corresponding CPU times grow only mildly with the problem size.
This behavior demonstrates the mesh-independent convergence of the preconditioned system and confirms that the preconditioners effectively captures the spectral structure of the matrix, resulting in a well-clustered spectrum and reduced computational cost.
\begin{table}
\centering
\renewcommand{\arraystretch}{1}
\setlength{\tabcolsep}{6pt}
\begin{tabular}{|c c |ccc|ccc|ccc|}
\hline
\multicolumn {2}{|c}{} & \multicolumn {3}{|c}{$\alpha = 1.3$} & \multicolumn {3}{|c}{$\alpha = 1.5$} & \multicolumn {3}{|c|}{$\alpha = 1.8$} \\
\hline
$N$ & $M$ & $A$ & $\widetilde{P}$ &  $\widehat{P}$ & $-$ & $\widetilde{P}$ &  $\widehat{P}$ & $-$ & $\widetilde{P}$ &  $\widehat{P}$  \\
\hline
$2^4$ &$2^4$ &51 &13 &13 &59 &16 &16 &81 &18 &18 \\
$2^4$ &$2^5$ &79 &13 &13 &89 &16 &16 &120 &20 &20 \\
$2^4$ &$2^6$ &199 &13 &13 &200 &16 &16 &223 &20 &20 \\
$2^4$ &$2^7$ &626 &13 &13 &609 &16 &16 &604 &20 &20 \\
\hline
$2^5$ &$2^4$ &82 &15 &14 &98 &17 &17 &140 &20 &20 \\
$2^5$ &$2^5$ &112 &15 &15 &135 &18 &18 &200 &24 &24 \\
$2^5$ &$2^6$ &229 &15 &15 &247 &18 &18 &327 &24 &24 \\
$2^5$ &$2^7$ &717 &15 &15 &699 &18 &18 &751 &24 &24 \\
\hline
$2^6$ &$2^4$ &145 &16 &16 &175 &20 &19 &253 &22 &22 \\
$2^6$ &$2^5$ &181 &16 &16 &232 &21 &20 &355 &28 &28 \\
$2^6$ &$2^6$ &300 &16 &16 &365 &21 &20 &545 &29 &28 \\
$2^6$ &$2^7$ &791 &16 &16 &841 &21 &20 &1102 &29 &28 \\
\hline
$2^7$ &$2^4$ &273 &17 &17 &329 &21 &21 &481 &22 &22 \\
$2^7$ &$2^5$ &318 &17 &17 &427 &24 &23 &666 &33 &34 \\
$2^7$ &$2^6$ &451 &18 &17 &617 &24 &23 &996 &34 &35 \\
$2^7$ &$2^7$ &957 &18 &17 &1180 &24 &23 &1864 &34 &35 \\
\hline
\end{tabular}
\caption{Iterations  required for solving the unpreconditioned $A$ and preconditioned linear system with $\widetilde{P} =\widetilde{P}_{\alpha,N,M}$ and $\widehat{P} = \widehat{P}_{\alpha,N,M}$ with the GMRES method - $a(x)=1$.}
\label{tab:gmres-iter-es1}
\end{table}
\begin{table}
\centering
\renewcommand{\arraystretch}{1}
\setlength{\tabcolsep}{3pt}
\begin{tabular}{|c c |ccc|ccc|ccc|}
\hline
\multicolumn {2}{|c}{} & \multicolumn {3}{|c}{$\alpha = 1.3$} & \multicolumn {3}{|c}{$\alpha = 1.5$} & \multicolumn {3}{|c|}{$\alpha = 1.8$} \\
\hline
$N$ & $M$ & $A$ & $\widetilde{P}^{-1}A$ & $\widehat{P}^{-1}A$ & $A$ & $\widetilde{P}^{-1}A$ & $\widehat{P}^{-1}A$ & $A$ & $\widetilde{P}^{-1}A$ & $\widehat{P}^{-1}A$ \\
\hline
$2^4$ &$2^4$ &0.0035 &0.0132 &0.0058 &0.0023 &0.0069 &0.0066 &0.0005 &0.0081 &0.0061 \\
$2^4$ &$2^5$ &0.0014 &0.0059 &0.0049 &0.0009 &0.0076 &0.0055 &0.0010 &0.0074 &0.0067 \\
$2^4$ &$2^6$ &0.0043 &0.0085 &0.0077 &0.0043 &0.0090 &0.0077 &0.0051 &0.0113 &0.0092 \\
$2^4$ &$2^7$ &0.0758 &0.0113 &0.0107 &0.0701 &0.0133 &0.0116 &0.0693 &0.0161 &0.0141 \\
\hline
$2^5$ &$2^4$ &0.0006 &0.0070 &0.0065 &0.0008 &0.0086 &0.0120 &0.0014 &0.0088 &0.0162 \\
$2^5$ &$2^5$ &0.0015 &0.0098 &0.0081 &0.0024 &0.0112 &0.0099 &0.0058 &0.0124 &0.0118 \\
$2^5$ &$2^6$ &0.0098 &0.0160 &0.0146 &0.0124 &0.0180 &0.0170 &0.0191 &0.0222 &0.0226 \\
$2^5$ &$2^7$ &0.3930 &0.0250 &0.0255 &0.3764 &0.0297 &0.0290 &0.4587 &0.0393 &0.0501 \\
\hline
$2^6$ &$2^4$ &0.0023 &0.0371 &0.0141 &0.0032 &0.0180 &0.0179 &0.0079 &0.0221 &0.0266 \\
$2^6$ &$2^5$ &0.0091 &0.0237 &0.0183 &0.0110 &0.0256 &0.0237 &0.0291 &0.0364 &0.0461 \\
$2^6$ &$2^6$ &0.0917 &0.0457 &0.0403 &0.2056 &0.0571 &0.0542 &0.3715 &0.0830 &0.0958 \\
$2^6$ &$2^7$ &1.7004 &0.0815 &0.0771 &1.8737 &0.0944 &0.0879 &3.8111 &0.2778 &0.2250 \\
\hline
$2^7$ &$2^4$ &0.0184 &0.0364 &0.0336 &0.0250 &0.0396 &0.0395 &0.0940 &0.0738 &0.0731 \\
$2^7$ &$2^5$ &0.1173 &0.0551 &0.0916 &0.2095 &0.0733 &0.0825 &0.4578 &0.1577 &0.1877 \\
$2^7$ &$2^6$ &0.6572 &0.1909 &0.1518 &1.4853 &0.1912 &0.3413 &3.3912 &0.2840 &0.3046 \\
$2^7$ &$2^7$ &4.2326 &0.2572 &0.1810 &9.4866 &0.4533 &0.2786 &16.0755 &0.4172 &0.4750 \\
\hline
\end{tabular}
\caption{Averaged timings for iterations in Table \ref{tab:gmres-iter-es1} repeated 10 times with  $A= A_{\alpha,N,M}$,
$\widetilde{P}^{-1}A =\widetilde{P}^{-1}_{\alpha,N,M}A_{\alpha,N,M}$, and $\widehat{P}^{-1}A = \widehat{P}^{-1}_{\alpha,N,M}A_{\alpha,N,M}$ - case $a(x)=1$.}
\label{tab:Timings_es1}
\end{table}
%
\begin{table}
\centering
\renewcommand{\arraystretch}{1}
\setlength{\tabcolsep}{6pt}
\begin{tabular}{|c c |ccc|ccc|ccc|}
\hline
\multicolumn {2}{|c}{} & \multicolumn {3}{|c}{$\alpha = 1.3$} & \multicolumn {3}{|c}{$\alpha = 1.5$} & \multicolumn {3}{|c|}{$\alpha = 1.8$} \\
\hline
$N$ & $M$ & $A$ & $\widetilde{P}$ &  $\widehat{P}$ & $-$ & $\widetilde{P}$ &  $\widehat{P}$ & $-$ & $\widetilde{P}$ &  $\widehat{P}$  \\
\hline
$2^4$ &$2^4$ &55 &15 &15 &67 &18 &18 &96 &21 &21 \\
$2^4$ &$2^5$ &84 &16 &16 &98 &19 &19 &139 &24 &24 \\
$2^4$ &$2^6$ &198 &16 &16 &201 &19 &19 &235 &24 &24 \\
$2^4$ &$2^7$ &618 &16 &16 &604 &19 &19 &600 &24 &24 \\
\hline
$2^5$ &$2^4$ &91 &17 &18 &114 &21 &21 &171 &25 &25 \\
$2^5$ &$2^5$ &127 &18 &18 &157 &23 &23 &244 &31 &31 \\
$2^5$ &$2^6$ &234 &18 &18 &263 &23 &23 &381 &31 &31 \\
$2^5$ &$2^7$ &701 &18 &18 &693 &23 &23 &793 &31 &31 \\
\hline
$2^6$ &$2^4$ &162 &19 &19 &207 &23 &23 &314 &26 &25 \\
$2^6$ &$2^5$ &211 &20 &20 &273 &26 &26 &444 &36 &36 \\
$2^6$ &$2^6$ &324 &20 &20 &411 &26 &26 &681 &38 &38 \\
$2^6$ &$2^7$ &787 &20 &20 &875 &27 &26 &1282 &38 &38 \\
\hline
$2^7$ &$2^4$ &302 &21 &21 &394 &25 &25 &594 &26 &26 \\
$2^7$ &$2^5$ &377 &22 &22 &506 &30 &30 &834 &40 &39 \\
$2^7$ &$2^6$ &526 &22 &22 &730 &31 &30 &1287 &45 &45 \\
$2^7$ &$2^7$ &1013 &22 &22 &1321 &31 &30 &2365 &46 &45 \\
\hline
\end{tabular}
\caption{Iterations  required for solving the unpreconditioned $A$ and preconditioned linear system with $\widetilde{P} =\widetilde{P}_{\alpha,N,M}$ and $\widehat{P} = \widehat{P}_{\alpha,N,M}$ with the GMRES method - $a(x)=x^2+1$.}
\label{tab:gmres-iter-es2}
\end{table}
\begin{table}
\centering
\renewcommand{\arraystretch}{1}
\setlength{\tabcolsep}{3pt}
\begin{tabular}{|c c |ccc|ccc|ccc|}
\hline
\multicolumn {2}{|c}{} & \multicolumn {3}{|c}{$\alpha = 1.3$} & \multicolumn {3}{|c}{$\alpha = 1.5$} & \multicolumn {3}{|c|}{$\alpha = 1.8$} \\
\hline
$N$ & $M$ & $A$ & $\widetilde{P}^{-1}A$ & $\widehat{P}^{-1}A$ & $A$ & $\widetilde{P}^{-1}A$ & $\widehat{P}^{-1}A$ & $A$ & $\widetilde{P}^{-1}A$ & $\widehat{P}^{-1}A$ \\
\hline
$2^4$ &$2^4$ &0.0008 &0.0059 &0.0045 &0.0003 &0.0063 &0.0049 &0.0079 &0.0181 &0.0077 \\
$2^4$ &$2^5$ &0.0007 &0.0054 &0.0067 &0.0008 &0.0058 &0.0057 &0.0037 &0.0092 &0.0097 \\
$2^4$ &$2^6$ &0.0044 &0.0087 &0.0079 &0.0039 &0.0087 &0.0086 &0.0062 &0.0134 &0.0107 \\
$2^4$ &$2^7$ &0.0812 &0.0235 &0.0203 &0.0651 &0.0173 &0.0133 &0.0744 &0.0167 &0.0169 \\
\hline
$2^5$ &$2^4$ &0.0007 &0.0174 &0.0167 &0.0010 &0.0095 &0.0101 &0.0021 &0.0160 &0.0113 \\
$2^5$ &$2^5$ &0.0048 &0.0192 &0.0138 &0.0026 &0.0123 &0.0116 &0.0061 &0.0174 &0.0157 \\
$2^5$ &$2^6$ &0.0233 &0.0338 &0.0323 &0.0131 &0.0212 &0.0222 &0.0296 &0.0350 &0.0587 \\
$2^5$ &$2^7$ &0.5191 &0.0555 &0.0515 &0.3705 &0.0354 &0.0364 &0.6981 &0.1120 &0.1018 \\
\hline
$2^6$ &$2^4$ &0.0053 &0.0346 &0.0385 &0.0042 &0.0199 &0.0223 &0.0163 &0.0395 &0.0371 \\
$2^6$ &$2^5$ &0.0134 &0.0436 &0.0415 &0.0145 &0.0268 &0.0285 &0.0475 &0.0614 &0.0611 \\
$2^6$ &$2^6$ &0.1129 &0.0620 &0.0863 &0.1322 &0.0492 &0.0691 &0.4915 &0.1285 &0.1319 \\
$2^6$ &$2^7$ &1.5034 &0.1394 &0.1122 &1.8962 &0.1108 &0.1084 &4.3610 &0.2908 &0.2636 \\
\hline
$2^7$ &$2^4$ &0.0378 &0.0749 &0.0612 &0.0327 &0.0457 &0.0459 &0.1315 &0.0986 &0.0871 \\
$2^7$ &$2^5$ &0.1386 &0.0995 &0.0609 &0.2669 &0.0844 &0.0802 &1.1671 &0.2150 &0.1563 \\
$2^7$ &$2^6$ &0.9735 &0.2023 &0.1671 &1.3623 &0.1680 &0.2396 &5.9787 &0.4148 &0.7292 \\
$2^7$ &$2^7$ &5.0880 &0.3373 &0.4153 &8.4437 &0.4889 &0.5794 &33.0959 &1.1614 &0.5574 \\
\hline
\end{tabular}
\caption{Averaged timings for iterations in Table \ref{tab:gmres-iter-es2} repeated 10 times with  $A= A_{\alpha,N,M}$,
$\widetilde{P}^{-1}A =\widetilde{P}^{-1}_{\alpha,N,M}A_{\alpha,N,M}$, and $\widehat{P}^{-1}A = \widehat{P}^{-1}_{\alpha,N,M}A_{\alpha,N,M}$ - case $a(x)=x^2+1$.}
\label{tab:Timings_es2}
\end{table}
%
\begin{table}
\centering
\renewcommand{\arraystretch}{1}
\setlength{\tabcolsep}{6pt}
\begin{tabular}{|c c |ccc|ccc|ccc|}
\hline
\multicolumn {2}{|c}{} & \multicolumn {3}{|c}{$\alpha = 1.3$} & \multicolumn {3}{|c}{$\alpha = 1.5$} & \multicolumn {3}{|c|}{$\alpha = 1.8$} \\
\hline
$N$ & $M$ & $A$ & $\widetilde{P}$ &  $\widehat{P}$ & $-$ & $\widetilde{P}$ &  $\widehat{P}$ & $-$ & $\widetilde{P}$ &  $\widehat{P}$  \\
\hline
$2^4$ &$2^4$ &47 &26 &26 &59 &38 &38 &86 &65 &65 \\
$2^4$ &$2^5$ &79 &27 &27 &87 &43 &43 &119 &82 &82 \\
$2^4$ &$2^6$ &211 &28 &28 &210 &45 &45 &227 &96 &96 \\
$2^4$ &$2^7$ &660 &28 &28 &656 &45 &45 &655 &99 &99 \\
\hline
$2^5$ &$2^4$ &71 &42 &42 &98 &66 &65 &154 &123 &123 \\
$2^5$ &$2^5$ &100 &47 &47 &129 &84 &84 &207 &168 &167 \\
$2^5$ &$2^6$ &232 &49 &49 &249 &101 &101 &343 &219 &218 \\
$2^5$ &$2^7$ &767 &49 &49 &757 &103 &103 &780 &261 &260 \\
\hline
$2^6$ &$2^4$ &119 &66 &66 &170 &109 &109 &284 &216 &215 \\
$2^6$ &$2^5$ &148 &86 &85 &210 &152 &151 &379 &271 &270 \\
$2^6$ &$2^6$ &284 &91 &90 &349 &195 &193 &602 &353 &351 \\
$2^6$ &$2^7$ &811 &93 &91 &843 &223 &223 &1161 &463 &461 \\
\hline
$2^7$ &$2^4$ &217 &99 &99 &317 &147 &146 &543 &198 &197 \\
$2^7$ &$2^5$ &251 &122 &120 &378 &195 &193 &718 &256 &255 \\
$2^7$ &$2^6$ &400 &138 &138 &571 &255 &250 &1124 &338 &336 \\
$2^7$ &$2^7$ &921 &153 &146 &1128 &315 &311 &2062 &462 &458 \\
\hline
\end{tabular}
\caption{Iterations  required for solving the unpreconditioned $A$ and preconditioned linear system with $\widetilde{P} =\widetilde{P}_{\alpha,N,M}$ and $\widehat{P} = \widehat{P}_{\alpha,N,M}$ with the GMRES method - $a(x)=x^2$.}
\label{tab:gmres-iter-es4}
\end{table}
\clearpage
{\subsection*{GMRES performance in the two-dimensional setting}}
Finally, we report the results of a few numerical tests performed in the the two-dimensional setting, both in the constant-coefficient case and in the variable-coefficient setting, where the diffusion function $a(x_1,x_2)$ is chosen as the most natural bivariate extension of the ones previously considered.
Though, discretization parameters are a bit small, due to the global dimension of the all-at-once systems,  no relevant difference is observed in the convergence properties, which we already reported and commented in the one-dimensional setting; see Table \ref{tab:Test12d}, Table \ref{tab:Test22d}, and Table \ref{tab:Test42d} for the numerics in the two-dimensional space case.
\par
It is worth stressing that we do not report numerical experiments when $\alpha_1 \neq \alpha_2$ since the results in practice are of the same quality as in the case $\alpha_1=\alpha_2$ (and even slightly better indeed): the reason of such behavior is contained in Theorem \ref{2D th:}, Theorem \ref{2D orig mat}, Theorem \ref{thm:inv(P)A clustering 2D}, where it is evident that $\alpha_1 \neq \alpha_2$ implies a simplification in the spectral behavior, because the matrix sequence associated to the smaller parameter is spectrally negligible, when the matrix-sizes increase.

\begin{table}
\centering
\begin{tabular}{|c c c|ccc|ccc|ccc|}
\hline
\multicolumn {3}{|c}{} & \multicolumn {3}{|c}{$\alpha_1 =\alpha_2 = 1.3$} & \multicolumn {3}{|c}{$\alpha_1 =\alpha_2  = 1.5$} & \multicolumn {3}{|c|}{$\alpha_1 =\alpha_2  = 1.8$} \\
\hline
$N_1$ &$N_2$ & $M$ & $A$ & $\widetilde{P}^{-1}A$ & $\widehat{P}^{-1}A$ & $A$ & $\widetilde{P}^{-1}A$ & $\widehat{P}^{-1}A$ & $A$ & $\widetilde{P}^{-1}A$ & $\widehat{P}^{-1}A$ \\
\hline
$2^2$ &$2^2$ &$2^2$ &16 &10 &10 &19 &11 &11 &21 &12 &12 \\
$2^2$ &$2^2$ &$2^3$ &22 &11 &11 &25 &12 &12 &27 &13 &14 \\
$2^2$ &$2^2$ &$2^4$ &32 &11 &12 &35 &13 &13 &40 &15 &16 \\
$2^2$ &$2^2$ &$2^5$ &65 &11 &12 &66 &13 &14 &68 &16 &16 \\
\hline
$2^3$ &$2^3$ &$2^2$ &27 &13 &13 &31 &14 &15 &37 &15 &15 \\
$2^3$ &$2^3$ &$2^3$ &33 &13 &13 &37 &15 &15 &45 &17 &17 \\
$2^3$ &$2^3$ &$2^4$ &45 &14 &14 &51 &16 &16 &63 &18 &18 \\
$2^3$ &$2^3$ &$2^5$ &77 &14 &14 &81 &16 &17 &97 &20 &20 \\
\hline
$2^4$ &$2^4$ &$2^2$ &49 &15 &15 &55 &17 &17 &67 &18 &17 \\
$2^4$ &$2^4$ &$2^3$ &56 &16 &16 &62 &19 &19 &79 &21 &21 \\
$2^4$ &$2^4$ &$2^4$ &69 &16 &16 &79 &19 &19 &106 &23 &23 \\
$2^4$ &$2^4$ &$2^5$ &103 &16 &16 &116 &20 &20 &155 &25 &25 \\
\hline
$2^5$ &$2^5$ &$2^2$ &91 &18 &18 &102 &20 &20 &124 &20 &20 \\
$2^5$ &$2^5$ &$2^3$ &101 &18 &18 &112 &22 &22 &145 &24 &24 \\
$2^5$ &$2^5$ &$2^4$ &118 &18 &18 &139 &23 &23 &189 &29 &29 \\
$2^5$ &$2^5$ &$2^5$ &160 &19 &18 &187 &24 &24 &268 &31 &31 \\
\hline
\end{tabular}
\caption{Iterations  required for solving the unpreconditioned $A$ and preconditioned linear system with $\widetilde{P} =\widetilde{P}_{\boldsymbol{\alpha},\textbf{n},M}$ and $\widehat{P} = \widehat{P}_{\boldsymbol{\alpha},\textbf{n},M}$ with the GMRES method - $a(x_1,x_2)=1$, $\theta = 0.5$.} \label{tab:Test12d}
\end{table}
\begin{table}
\centering
\begin{tabular}{|c c c|ccc|ccc|ccc|}
\hline
\multicolumn {3}{|c}{} & \multicolumn {3}{|c}{$\alpha_1 =\alpha_2 = 1.3$} & \multicolumn {3}{|c}{$\alpha_1 =\alpha_2  = 1.5$} & \multicolumn {3}{|c|}{$\alpha_1 =\alpha_2  = 1.8$} \\
\hline
$N_1$ &$N_2$ & $M$ & $A$ & $\widetilde{P}^{-1}A$ & $\widehat{P}^{-1}A$ & $A$ & $\widetilde{P}^{-1}A$ & $\widehat{P}^{-1}A$ & $A$ & $\widetilde{P}^{-1}A$ & $\widehat{P}^{-1}A$ \\
\hline
$2^2$ &$2^2$ &$2^2$ &20 &13 &13 &22 &14 &14 &24 &14 &14 \\
$2^2$ &$2^2$ &$2^3$ &25 &14 &14 &28 &15 &15 &31 &16 &16 \\
$2^2$ &$2^2$ &$2^4$ &36 &14 &15 &40 &16 &16 &45 &18 &19 \\
$2^2$ &$2^2$ &$2^5$ &66 &14 &15 &67 &16 &17 &74 &20 &20 \\
\hline
$2^3$ &$2^3$ &$2^2$ &33 &15 &15 &38 &17 &17 &44 &17 &17 \\
$2^3$ &$2^3$ &$2^3$ &40 &16 &16 &45 &18 &18 &54 &21 &20 \\
$2^3$ &$2^3$ &$2^4$ &53 &17 &17 &62 &19 &19 &77 &23 &23 \\
$2^3$ &$2^3$ &$2^5$ &82 &17 &17 &93 &20 &20 &118 &25 &25 \\
\hline
$2^4$ &$2^4$ &$2^2$ &59 &19 &18 &67 &21 &20 &84 &21 &21 \\
$2^4$ &$2^4$ &$2^3$ &68 &19 &19 &78 &23 &22 &101 &25 &25 \\
$2^4$ &$2^4$ &$2^4$ &86 &20 &20 &102 &23 &23 &139 &29 &29 \\
$2^4$ &$2^4$ &$2^5$ &120 &21 &20 &146 &25 &24 &207 &32 &31 \\
\hline
$2^5$ &$2^5$ &$2^2$ &112 &22 &22 &127 &25 &24 &162 &27 &26 \\
$2^5$ &$2^5$ &$2^3$ &125 &23 &23 &144 &27 &27 &193 &31 &30 \\
$2^5$ &$2^5$ &$2^4$ &152 &23 &23 &183 &28 &28 &259 &36 &36 \\
$2^5$ &$2^5$ &$2^5$ &198 &24 &24 &251 &30 &30 &381 &39 &39 \\
\hline
\end{tabular}
\caption{Iterations  required for solving the unpreconditioned $A$ and preconditioned linear system with $\widetilde{P} =\widetilde{P}_{\boldsymbol{\alpha},\textbf{n},M}$ and $\widehat{P} = \widehat{P}_{\boldsymbol{\alpha},\textbf{n},M}$ with the GMRES method - $a(x_1,x_2)=x_1^2+x_2^2+1$, $\theta = 0.5$.} \label{tab:Test22d}
\end{table}
\begin{table}
\centering
\begin{tabular}{|c c c|ccc|ccc|ccc|}
\hline
\multicolumn {3}{|c}{} & \multicolumn {3}{|c}{$\alpha_1 =\alpha_2 = 1.3$} & \multicolumn {3}{|c}{$\alpha_1 =\alpha_2  = 1.5$} & \multicolumn {3}{|c|}{$\alpha_1 =\alpha_2  = 1.8$} \\
\hline
$N_1$ &$N_2$ & $M$ & $A$ & $\widetilde{P}^{-1}A$ & $\widehat{P}^{-1}A$ & $A$ & $\widetilde{P}^{-1}A$ & $\widehat{P}^{-1}A$ & $A$ & $\widetilde{P}^{-1}A$ & $\widehat{P}^{-1}A$ \\
\hline
$2^2$ &$2^2$ &$2^2$ &23 &18 &18 &27 &20 &21 &32 &23 &23 \\
$2^2$ &$2^2$ &$2^3$ &27 &19 &18 &31 &22 &22 &38 &26 &26 \\
$2^2$ &$2^2$ &$2^4$ &37 &19 &19 &42 &23 &23 &50 &29 &30 \\
$2^2$ &$2^2$ &$2^5$ &72 &20 &20 &73 &25 &24 &79 &32 &32 \\
\hline
$2^3$ &$2^3$ &$2^2$ &41 &29 &29 &50 &36 &35 &63 &44 &44 \\
$2^3$ &$2^3$ &$2^3$ &44 &29 &29 &54 &38 &37 &70 &52 &51 \\
$2^3$ &$2^3$ &$2^4$ &54 &30 &30 &66 &42 &42 &88 &63 &63 \\
$2^3$ &$2^3$ &$2^5$ &86 &32 &32 &96 &46 &46 &122 &74 &74 \\
\hline
$2^4$ &$2^4$ &$2^2$ &76 &41 &41 &93 &57 &56 &130 &79 &78 \\
$2^4$ &$2^4$ &$2^3$ &77 &44 &43 &96 &61 &61 &141 &94 &93 \\
$2^4$ &$2^4$ &$2^4$ &87 &48 &48 &111 &74 &74 &166 &125 &124 \\
$2^4$ &$2^4$ &$2^5$ &116 &54 &53 &144 &91 &90 &217 &152 &150 \\
\hline
$2^5$ &$2^5$ &$2^2$ &149 &56 &55 &186 &86 &84 &253 &123 &122 \\
$2^5$ &$2^5$ &$2^3$ &143 &61 &60 &187 &97 &95 &270 &158 &156 \\
$2^5$ &$2^5$ &$2^4$ &152 &74 &73 &204 &123 &122 &315 &202 &196 \\
$2^5$ &$2^5$ &$2^5$ &181 &93 &92 &247 &164 &161 &402 &249 &243 \\
\hline
\end{tabular}
\caption{Iterations  required for solving the unpreconditioned $A$ and preconditioned linear system with $\widetilde{P} =\widetilde{P}_{\boldsymbol{\alpha},\textbf{n},M}$ and $\widehat{P} = \widehat{P}_{\boldsymbol{\alpha},\textbf{n},M}$ with the GMRES method - $a(x_1,x_2)=x_1^2+x_2^2$, $\theta = 0.5$.} \label{tab:Test42d}
\end{table}

\section{Conclusions}\label{sec:end}

We have studied the spectral properties of large matrices and preconditioning of linear systems, arising from the finite difference discretization of a time-dependent space-fractional diffusion equation with a variable coefficient $a$ defined on $\Omega \subset \mathbb{R}^{d}$, $d=1,2$.

A preconditioning strategy has been developed based on the structural properties of the discretized operator. The spectral features of the unpreconditioned and preconditioned matrix sequences have been studied with the help of GLT theory. The main novelty is that the analysis fully covers the case where the variable coefficient $a$ is nonconstant. Numerical results have been reported and have been critically discussed.

Several open problems remain. The most relevant are reported below.

\begin{itemize}
\item The extension of the result in \cite{eig-nonH} to the variable coefficient case: this seems not trivial and it is not covered by the GLT theory due to the intrinsic non-Hermitian nature of the underlying matrix sequences;
\item Estimates of the number of the outliers and of the asymptotical behaviour of the outliers in the preconditioned matrix sequences.
\end{itemize}

While the first item has a theoretical appeal, the second could have a precise practical impact in estimating a priori the number of needed iterations of the proposed  preconditioned GMRES procedures for reaching a preassigned accuracy.



\begin{thebibliography}{99}


\bibitem{reduced}
Barbarino, G.:
{A systematic approach to reduced GLT},
\textit{BIT} \textbf{62(3)}, 681–-743 (2022)

\bibitem{Barbarino2020a}
Barbarino, G., Garoni, C., Serra-Capizzano, S.:
Block generalized locally Toeplitz sequences: theory and applications in the unidimensional case.
\textit{Electron. Trans. Numer. Anal.} \textbf{53}, 28--112 (2020)

\bibitem{Barbarino2020b}
Barbarino, G., Garoni, C., Serra-Capizzano, S.:
Block generalized locally Toeplitz sequences: theory and applications in the multidimensional case.
\textit{Electron. Trans. Numer. Anal.} \textbf{53}, 113--216 (2020)

\bibitem{BarbarinoSerra2020}
Barbarino, G., Serra-Capizzano, S.:
Non-Hermitian perturbations of Hermitian matrix-sequences and applications to the spectral analysis of the numerical approximation of partial differential equations.
\textit{Numer. Linear Algebra Appl.} \textbf{27}(3), e2286 (2020)

\bibitem{Benzi2002Preconditioning}
Benzi, M.:
Preconditioning techniques for large linear systems: a survey.
\textit{J. Comput. Phys.} \textbf{182}(2), 418--477 (2002)

\bibitem{Bhatia1997}
Bhatia, R.:
\textit{Matrix Analysis}.
Graduate Texts in Mathematics, vol. 169, Springer-Verlag, New York (1997)



\bibitem{eig-nonH}
Bogoya, J.M., Böttcher, A., Grudsky, S.M.:
Asymptotics of individual eigenvalues of a class of large Hessenberg Toeplitz matrices.
\textit{Operator Th.: Adv. Appl.} \textbf{220}, 77--95 (2012)


\bibitem{rev-CN}
Chan, R.H., Ng, M.:
Conjugate gradient methods for Toeplitz systems, \textit{SIAM Rev.} \textbf{38}, 427–-482 (1996)


\bibitem{Brociek2016}
Brociek, R.:
Crank-Nicolson scheme for space fractional heat conduction equation with mixed boundary condition.
In: Damaševičius, R., Napoli, C., Tramontana, E., Woźniak, M. (Eds.)
\textit{Proceedings of the Symposium for Young Scientists in Technology, Engineering and Mathematics}, CEUR Workshop Proceedings, 41--45 (2016).

\bibitem{ChenDeng2014}
Chen, M.H., Deng, W.H.:
Fourth order accurate scheme for the space fractional diffusion equations.
\textit{SIAM J. Numer. Anal.} \textbf{52}, 1418--1438 (2014)


\bibitem{Commenges1984}
Commenges, D., Monsion, M.:
Fast inversion of triangular Toeplitz matrices.
\textit{IEEE Trans. Autom. Control},
29(3), 250--251 (1984)

\bibitem{donatelli2016}
Donatelli, M., Mazza, M., Serra-Capizzano, S.:
Spectral analysis and structure preserving preconditioners for fractional diffusion equations.
\textit{J. Comput. Phys.} \textbf{307}, 262--279 (2016)

\bibitem{superopt-ESC}
Estatico, C., Serra-Capizzano, S.: Superoptimal approximation for unbounded symbols.
\textit{Linear Algebra Appl.} \textbf{428}(2/3), 564–-585 (2008)


\bibitem{tutorial2}
Garoni, C.:
Introduction to the theory of generalized locally Toeplitz sequences and its applications.
\textit{Math. Notes} (2025), in press; url{https://arxiv.org/abs/2506.02151}


\bibitem{GaroniSerra2017}
Garoni, C., Serra-Capizzano, S.:
\textit{Generalized Locally Toeplitz Sequences: Theory and Applications. Vol. I}.
Springer, Cham (2017)

\bibitem{GaroniSerra2018}
Garoni, C., Serra-Capizzano, S.:
\textit{Generalized Locally Toeplitz Sequences: Theory and Applications. Vol. II}.
Springer, Cham (2018)

\bibitem{tutorial1}
Garoni, C., Serra-Capizzano, S.:
Generalized locally Toeplitz sequences: a spectral analysis tool for discretized differential equations.
Splines and PDEs: from approximation theory to numerical linear algebra, 161--236,
Lecture Notes in Math., 2219, Springer, Cham (2018)


\bibitem{Golinskii2007}
Golinskii, L., Serra-Capizzano, S.:
The asymptotic properties of the spectrum of nonsymmetrically perturbed Jacobi matrix sequences.
\textit{J. Approx. Theory} \textbf{144}(1), 84--102 (2007)

\bibitem{hon2024}
Hon, S., Fung, P.Y., Dong, J., Serra-Capizzano, S.:
A sine transform based preconditioned MINRES method for all-at-once systems from constant and variable-coefficient evolutionary PDEs.
\textit{Numer. Algor.} \textbf{95}(4), 1769--1799 (2024)


\bibitem{LinNgSun2018}
Lin, X.-L., Ng, M.K., Sun, H.-W.:
Efficient preconditioner of one-sided space fractional diffusion equation.
\textit{BIT Numer. Math.} \textbf{58}(3), 729--748 (2018)

\bibitem{lin2021}
Lin, X.-L., Ng, M.K.:
An all-at-once preconditioner for evolutionary partial differential equations.
\textit{SIAM J. Sci. Comput.} \textbf{43}(4), A2766--A2784 (2021)

\bibitem{Rosita-immersed}
Mazza, M., Serra-Capizzano, S., Sormani R.L.:
Algebra preconditionings for 2D Riesz distributed-order space-fractional diffusion equations on convex domains.
\textit{Numer. Linear Algebra Appl.} \textbf{31}(3), e2536 (2024)


\bibitem{mcdonald2018}
McDonald, E., Pestana, J., Wathen, A.:
Preconditioning and iterative solution of all-at-once systems for evolutionary partial differential equations.
\textit{SIAM J. Sci. Comput.} \textbf{40}(2), A1012--A1033 (2018)

\bibitem{Meerschaert2002}
Meerschaert, M.M., Benson, D.A., Scheffler, H.P., Becker-Kern, P.:
Governing equations and solutions of anomalous random walk limits.
\textit{Phys. Rev. E} \textbf{66}, 102R--105R (2002)

\bibitem{MeerschaertTadjeran2006}
Meerschaert, M.M., Tadjeran, C.:
Finite difference approximations for two-sided space-fractional partial differential equations.
\textit{Appl. Numer. Math.} \textbf{56}, 80--90 (2006)

\bibitem{Mortici2014}
Mortici, C.:
Methods and algorithms for approximating the gamma function and related functions. A survey. Part I: asymptotic series.
\textit{Ann. Acad. Rom. Sci. Ser. Math. Appl.} \textbf{6}, 173 (2014)

\bibitem{pan2016}
Pan, J.Y., Ng, M.K., Wang, H.:
Fast iterative solvers for linear systems arising from time-dependent space-fractional diffusion equations.
\textit{SIAM J. Sci. Comput.} \textbf{38}(5), A2806--A2826 (2016)

\bibitem{Saad1986GMRES}
Saad, Y., Schultz, M.H.:
GMRES: A generalized minimal residual algorithm for solving nonsymmetric linear systems.
\textit{SIAM J. Sci. Stat. Comput.} \textbf{7}(3), 856--869 (1986)



\bibitem{Toe-diff}
Serra-Capizzano, S.:
The rate of convergence of Toeplitz based PCG methods for second order nonlinear boundary value problems.
\textit{Numer. Math.} \textbf{81}(3), 461--495 (1999)

\bibitem{Serra2001}
Serra-Capizzano, S.:
Spectral behavior of matrix sequences and discretized boundary value problems.
\textit{Linear Algebra Appl.} \textbf{337}, 37--78 (2001)

\bibitem{serra2003}
Serra-Capizzano, S.: {Generalized locally Toeplitz sequences: spectral analysis and applications to discretized partial differential equations},
\textit{Linear Algebra Appl.} \textbf{366}, 371–-402 (2003)


\bibitem{serra2006}
Serra-Capizzano, S.:
{The GLT class as a generalized Fourier analysis and applications},
\textit{Linear Algebra Appl.} \textbf{419}, 180–-233 (2006)


\bibitem{she2023}
She, Z.-H., Qiu, L.-M.:
Fast TTTS iteration methods for implicit Runge--Kutta temporal discretization of Riesz space fractional advection--diffusion equations.
\textit{Comput. Math. Appl.} \textbf{141}, 42--53 (2023)



\bibitem{Tadjeran2007}
Tadjeran, C., Meerschaert, M.M.:
A second-order accurate numerical method for the two-dimensional fractional diffusion equation.
\textit{J. Comput. Phys.} \textbf{220}, 813--823 (2007)

\bibitem{Tadjeran2006}
Tadjeran, C., Meerschaert, M.M., Scheffler, H.P.:
A second-order accurate numerical approximation for the fractional diffusion equation.
\textit{J. Comput. Phys.} \textbf{213}, 205--213 (2006)

\bibitem{Tian2015}
Tian, W.Y., Zhou, H., Deng, W.H.:
A class of second order difference approximations for solving space fractional diffusion equations.
\textit{Math. Comput.} \textbf{84}, 1703--1727 (2015)

\bibitem{Wang2010}
Wang, H., Wang, K.X., Sircar, T.:
A direct \( \mathcal{O}(N \log^2 N) \) finite difference method for fractional diffusion equations.
\textit{J. Comput. Phys.} \textbf{229}, 8095--8104 (2010)


\bibitem{zyg-bible}
Zygmund, A.:
\textit{Trigonometric Series.}
Cambridge University Press, Cambridge (1959)


\end{thebibliography}
\end{document}